\documentclass[11pt, a4paper, oneside]{amsart}

\usepackage[english]{babel}
\usepackage{amsmath, amsthm, amsfonts, mathrsfs, amssymb}
\usepackage{mathtools}
\mathtoolsset{centercolon}
\usepackage{booktabs}
\usepackage[shortlabels]{enumitem}
\usepackage[colorlinks, citecolor = blue]{hyperref}
\usepackage{color}
\usepackage{graphicx}

\setlist[itemize]{leftmargin=25pt}
\setlist[enumerate]{leftmargin=25pt}
\newcommand{\N}{\ensuremath{\mathbb{N}}}

\newcommand{\R}{\ensuremath{\mathbb{R}}}
\newcommand{\C}{\ensuremath{\mathbb{C}}}
\newcommand{\K}{\ensuremath{\mathbb{K}}}
\newcommand{\E}{\ensuremath{\mathbb{E}}}
\renewcommand{\P}{\ensuremath{\mathbb{P}}}

\newcommand{\mc}{\mathcal}
\newcommand{\ms}{\mathscr}

\DeclarePairedDelimiter\abs{\lvert}{\rvert}

\DeclarePairedDelimiter\cbrace\{\}
\DeclarePairedDelimiter\ha()
\DeclarePairedDelimiter{\ip}\langle\rangle
\DeclarePairedDelimiter{\nrm}\lVert\rVert
\DeclarePairedDelimiter{\nrmn}{|\!|\!|}{|\!|\!|}

\newcommand{\nrmb}[1]{\bigl\|#1\bigr\|}
\newcommand{\absb}[1]{\bigl|#1\bigr|}
\newcommand{\hab}[1]{\bigl(#1\bigr)}
\newcommand{\cbraceb}[1]{\bigl\{#1\bigr\}}
\newcommand{\ipb}[1]{\bigl\langle#1\bigr\rangle}

\newcommand{\nrms}[1]{\Bigl\|#1\Bigr\|}
\newcommand{\abss}[1]{\Bigl|#1\Bigr|}
\newcommand{\has}[1]{\Bigl(#1\Bigr)}
\newcommand{\cbraces}[1]{\Bigl\{#1\Bigr\}}
\newcommand{\ips}[1]{\Bigl\langle#1\Bigr\rangle}

\DeclareMathOperator{\re}{Re}
\DeclareMathOperator{\Tr}{Tr}

\DeclareMathOperator{\loc}{loc}
\DeclareMathOperator{\supp}{supp}
\DeclareMathOperator{\ind}{\mathbf{1}}
\DeclareMathOperator{\UMD}{UMD}

\DeclareMathOperator{\HL}{HL}
\DeclareMathOperator{\BMO}{BMO}
\DeclareMathOperator{\SMR}{SMR}

\DeclareMathOperator*{\esssup}{ess\,sup}
\DeclareMathOperator*{\essinf}{ess\,inf}

\DeclareMathOperator{\Rad}{{Rad}}

\DeclareMathOperator{\Dinir}{-Dini}
\newcommand{\Dini}[2]{(#1,#2)\Dinir}
\DeclareMathOperator{\Hormanderr}{-H\ddot{o}rm}
\newcommand{\Hormander}[1]{#1\Hormanderr}
\DeclareMathOperator{\Standardr}{-std}
\newcommand{\Standard}[2]{(#1,#2)\Standardr}

\newcommand{\dd}{\hspace{2pt}\mathrm{d}}
\newcommand{\ddn}{\mathrm{d}}
\newcommand{\ee}{\mathrm{e}}

\def\avint_#1{\mathchoice{\mathop{\kern 0.2em\vrule width 0.6em height 0.69678ex depth -0.58065ex \kern -0.8em \intop}\nolimits_{\kern -0.4em#1}}{\mathop{\kern 0.1em\vrule width 0.5em height 0.69678ex depth -0.60387ex \kern -0.6em \intop}\nolimits_{#1}} {\mathop{\kern 0.1em\vrule width 0.5em height 0.69678ex depth -0.60387ex \kern -0.6em \intop}\nolimits_{#1}} {\mathop{\kern 0.1em\vrule width 0.5em height 0.69678ex depth -0.60387ex \kern -0.6em \intop}\nolimits_{#1}}}

\DeclareFontFamily{U}{mathx}{\hyphenchar\font45}
\DeclareFontShape{U}{mathx}{m}{n}{<5> <6> <7> <8> <9> <10> <10.95> <12> <14.4> <17.28> <20.74> <24.88> mathx10}{}
\DeclareSymbolFont{mathx}{U}{mathx}{m}{n}
\DeclareFontSubstitution{U}{mathx}{m}{n}
\DeclareMathAccent{\widecheck}{0}{mathx}{"71}

\newtheorem{theorem}{Theorem}
\newtheorem{corollary}[theorem]{Corollary}
\newtheorem{lemma}[theorem]{Lemma}
\newtheorem{proposition}[theorem]{Proposition}

\theoremstyle{remark}
\newtheorem{remark}[theorem]{Remark}
\newtheorem{example}[theorem]{Example}

\theoremstyle{definition}
\newtheorem{definition}[theorem]{Definition}

\numberwithin{theorem}{section}
\numberwithin{equation}{section}

\allowdisplaybreaks

\title[Singular stochastic integral operators]{Singular stochastic integral operators}
\author{Emiel Lorist and Mark Veraar}
\thanks{The authors are supported by the VIDI subsidy 639.032.427 of the Netherlands Organisation for Scientific Research (NWO)}

\address{Delft Institute of Applied Mathematics \\ Delft University of Technology \\ P.O. Box 5031\\ 2600 GA Delft \\The Netherlands}
\email{emiellorist@gmail.com}
\email{m.c.veraar@tudelft.nl}

\begin{document}

\begin{abstract}
In this paper we introduce Calder\'on--Zygmund theory for singular stochastic integrals with operator-valued kernel. In particular, we prove $L^p$-extrapolation results under a H\"ormander condition on the kernel. Sparse domination and sharp weighted bounds are obtained under a Dini condition on the kernel, leading to a stochastic version of the solution to the $A_2$-conjecture. The results are applied to obtain $p$-independence and weighted bounds for stochastic maximal $L^p$-regularity both in the complex and real interpolation scale. As a consequence we obtain several new regularity results for the stochastic heat equation on $\R^d$ and smooth and angular domains.
\end{abstract}




\maketitle

\tableofcontents

\section{Introduction}

In the study of stochastic partial differential equations (SPDEs), one often needs sharp regularity results for the linear equations. This together with fixed point arguments can be used to obtain existence, uniqueness and regularity for the solution to nonlinear SPDEs. During the last decades so-called maximal regularity results for SPDEs have been obtained in many papers. We refer to \cite[Section 6.3]{DPZ14} for an overview on the subject in the Hilbert space setting. In the $L^q$-setting sharp regularity results have been obtained in \cite{Kr99} by real analysis and PDE methods, and in \cite{NVW12} by functional calculus techniques.

In the above approaches one needs to prove sharp regularity estimates for singular stochastic integral operator  of the form
\begin{equation}\label{eq:SKdefintro}
  S_K G(s) := \int_{0}^\infty K(s,t) G(t) \dd W_H(t), \qquad s\in \R_+,
\end{equation}
where $G$ is an adapted process and $W_H$ is a cylindrical Brownian motion (see Section \ref{section:SIO}) and $K$ is a given operator-valued kernel $K\colon\R_+\times \R_+\to \mc{L}(X,Y)$. An important example of a kernel $K$ is
\begin{align}\label{eq:SMRkernel}
K(s,t) = e^{-(s-t)A} \ind_{t<s},
\end{align}
where $-A$ is the generator of an analytic semigroup and $Y$ is either the real interpolation space $(X, D(A))_{1/2,2}$, the complex interpolation space $[X,D(A)]_{1/2}$ or the fractional domain space $D((\lambda+A^{1/2}))$, where $\lambda\in \rho(-A)$. This kernel has a singularity of the form $\|K(s,t)\|\leq C(s-t)^{-1/2}$ for $|s-t|<1$. The $L^p$-boundedness of singular stochastic integrals with this kernel leads to {\em stochastic maximal $L^p$-regularity}.

Unlike in the deterministic setting, there is no general theory for the $L^p$-boundedness of singular stochastic integral operators. The aim of this paper is to provide a version of this theory and to use it to obtain new regularity results for abstract classes of SPDEs and more concrete examples such as the heat equation.

\subsection{Deterministic singular integrals}\label{subs:detintro}
Before Calder\'on-Zygmund theory \cite{CZ52} was developed, the $L^p$-boundedness of singular integral operators
\[T f := \int_{\R^d} K(s,t) f(t) \dd t\]
was considered on a case by case basis. Typically the singularity of the kernel $K$ is of the form $|K(s,t)|\leq C(s-t)^{-1}$. Important examples are the Hilbert transform
(for $d=1$), and the Riesz transforms (for $d\geq 2$) in which case the integral has to be interpreted as a principal value. Positive kernels are usually easier to deal with as in this case there is absolute convergence and one can apply Schur's lemma (see \cite[Appendix A]{Gr14b}).

In the convolution setting, i.e. $K(s,t) = k(s-t)$,  the Marcinkiewicz--Mihlin multiplier theorem gives simple sufficient conditions on $\widehat{k}$ (the Fourier transform of $k$) under which $T_K$ is a bounded operator on $L^p(\R^d)$  for all $p\in (1, \infty)$. For detailed expositions on Calder\'on--Zygmund operators and beyond, we refer to \cite{Gr14a,Gr14b,St93} and references therein.

The above results have been partially extended to the case where $K$ is scalar valued and $f$ takes values in a Banach space (see the monograph \cite{GR85}). A breakthrough result by \cite{Bu83} and \cite{Bo83} was that the Hilbert transform is bounded on $L^p(\R;X)$ with $p\in (1, \infty)$ if and only if $X$ is a so-called UMD space (see \cite[Chapter 4 and 5]{HNVW16} for details). Another major breakthrough was given in \cite{McC84}, \cite{Bo86} and \cite{Zi89}, where the Marcinkiewicz--Mihlin multiplier theorem and Littlewood--Paley decomposition were obtained for the class of UMD spaces.

For a long time operator-valued extensions of the latter results were unavailable outside Hilbert spaces. In \cite{We01b} the notion of $\mc{R}$-boundedness was used to obtain a Marcinkiewicz--Mihlin multiplier theorem in the operator-valued setting for $d=1$. This result was motivated by its applications to {\em maximal $L^p$-regularity} for parabolic PDEs, which were also discussed in \cite{We01b}. In this context the kernel $K:\R^2_+\to \mc{L}(X,D(A))$ is given by
\[K(s,t) = e^{-(s-t)A} \ind_{t<s},\]
where $-A$ is the generator of an analytic semigroup. This kernel satisfies $\|K(s,t)\|\leq C(s-t)^{-1}$ for $|t-s|<1$. Using Calder\'on--Zygmund theory one can therefore easily deduce that the $L^p$-boundedness of the associated singular integral operator for some $p_0\in [1, \infty]$ implies $L^p$-boundedness for all $p\in (1, \infty)$ (see \cite[Theorem 7.1]{Do00}). We refer to \cite{DHP03,KW04,PS16} for a detailed discussion on the history of maximal $L^p$-regularity and to \cite{KPW10,PSW18,PW17} for applications to nonlinear PDEs.

\subsection{Singular stochastic integrals}
A Calder\'on--Zygmund theory for stochastic integral operators as in \eqref{eq:SKdefintro} is not available. The behavior of stochastic singular integral operators \eqref{eq:SKdefintro} differs a lot from the deterministic setting. Due to the It\^o $L^2$-isometry the integrals converge absolutely and thus no principal values are needed. As a consequence, in contrast with the deterministic setting, the scalar-valued setting can easily be characterized, see Section \ref{section:scalarcase}.  In the operator-valued setting cancellation can for example occur in the following form:
\begin{align}\label{eq:HS:square}
\Big(\int_{\R_+}\|K(s,t)x\|^2_Y \dd t\Big)^{1/2} \leq C \|x\|_X,  \qquad s\in \R_+, \quad x\in X,
\end{align}
where $X$ and $Y$ are Banach spaces. If the kernel is of this form, then using a simple Fubini argument one can check that $S_K$ is $L^2$-bounded (see Propositions \ref{proposition:detcharacterization} and \ref{proposition:simplesufficient}\ref{it:sufficientL2}). In particular, this method was used for the kernel in \eqref{eq:SMRkernel} in the classical monograph \cite[Section 6.3]{DPZ14}. A sophisticated extension of this type of argument was used in  \cite{Br95}, \cite{BH09} and \cite{DL98} to cover $L^p$-boundedness in the scale of real interpolation spaces $(X, D(A))_{\theta,p}$.

The complex interpolation scale is more complicated. In particular, for $X = L^p(\R^d)$ \eqref{eq:HS:square} is often not true. For example it fails for $A = -\Delta$. To obtain $L^p$-estimates in this case, \cite{Kr94,Kr99, Kr08} use sharp estimates for stochastic integrals and sophisticated real analysis arguments. Moreover, by using PDE arguments the operator $A$ can be replaced by a second order elliptic operator with coefficients depending on $(t,\omega,x)\in [0,\infty)\times\Omega\times\R^d$, where some regularity in $x$ is assumed, but only progressive measurability is assumed in $(t,\omega)$. By an elaborate trick in \cite{Kr00} the estimates were extended to an $L^p(L^q)$-setting with $p\geq q\geq 2$. There are many sophisticated variations of the above methods in the literature in which different operators than $\Delta$ are considered and equations on different domains $D\subseteq \R^d$ are treated (see e.g.\ \cite{CKLL18,CKL18,Du18,Ki05, KK18,Kr09,Li14} and references therein).

On the scale of tent spaces stochastic maximal regularity for elliptic operators in divergence form is shown in \cite{ANP14}. This is done through extrapolation using off-diagonal estimates, which are substitutes for the classical pointwise kernel estimates of Calder\'on-–Zygmund theory. See also \cite{AKMP12} for the more general harmonic analysis framework developed to analyze this scale.

In \cite{NVW12,NVW15b} the $L^p$-boundedness of stochastic singular integrals with kernel \eqref{eq:SMRkernel} was obtained using the boundedness of the $H^\infty$-functional calculus together with the sharp two-sided estimates for stochastic integrals in UMD spaces developed in \cite{NVW07}. One of the advantages of this approach is that it can be used for an abstract operator $A$ as long as it has an $H^\infty$-calculus. Secondly, the stochastic integral operator is automatically $L^p$-bounded for any $p\in (2, \infty)$. Some geometric restrictions on $X$ are required, but these are fulfilled for $L^q$, $W^{s,q}$, etc.\ as long as $q\in [2, \infty)$ (see Section \ref{section:Rbddness}). In particular, mixed $L^p(L^q)$-regularity can be obtained for all $q\in [2, \infty)$ and $p\in (2, \infty)$, where $p=q=2$ is allowed as well.
The results of \cite{NVW12,NVW15b} have been applied to semilinear equations in \cite{NVW12b}, to quasilinear equations in \cite{Ho18} and to fully nonlinear equations in \cite{Ag18}.

Recently, in \cite{PV18} the framework of \cite{NVW12,NVW15b} has been extended to cover the case where $A$ depends on time and $\Omega$, as long as $D(A(t,\omega))$ is constant. The method is based on a  reduction to the time and $\Omega$-independent setting and gives a new approach to \cite{Kr99}, which additionally includes new optimal space-time regularity estimates and is applicable to a large class of SPDEs.

A large part of the theory of maximal $L^p$-regularity for deterministic PDEs was developed after the Calder\'on-Zygmund theory for operator-valued kernels was founded. In the stochastic case such a Calder\'on--Zygmund theory is not available yet, and our main motivation is to build such a theory and discover its potential for stochastic maximal $L^p$-regularity (see Subsection \ref{subsection:SMRintro}). Our first main theorem in this direction is as follows:

\begin{theorem}[$L^p$-boundedness of stochastic Calder\'on-Zygmund operators]\label{theorem:CZintro}
Let $X$ and $Y$ be Banach spaces with type $2$ and assume $Y$ is a UMD space. Let $K:\R_+\times\R_+\to \mc{L}(X,Y)$ be strongly measurable and assume that for every ball $B\subset \R_+$ we have the following H\"ormander condition
\begin{align}
   \has{\int_{\R_+ \setminus B} \nrm{K(s,t)-K(s',t)}^2\dd t}^{1/2} \leq C && s,s'\in \frac{1}{2}B \label{eq:horm1intro}\\
  \label{eq:horm2intro}
    \has{\int_{\R_+ \setminus B} \nrm{K(s,t)-K(s,t')}^2\dd s}^{1/2} \leq C && t,t'\in \frac{1}{2}B
\end{align}
for some constant $C>0$ independent of $B$. Fix $p\in [2, \infty)$ and suppose that the mapping $S_K$ as defined in \eqref{eq:SKdefintro} is bounded from $L^p_{\ms{F}}(\Omega\times \R_+;\gamma(H,X))$ into $L^p(\Omega\times \R_+;Y)$. Then for all $q\in (2, \infty)$ the mapping $S_K$ is bounded from $L^q_{\ms{F}}(\Omega\times \R_+;\gamma(H,X))$ into $L^q(\Omega\times \R_+;Y)$.
\end{theorem}
The above theorem follows from Proposition \ref{proposition:detcharacterization}, Theorem \ref{theorem:extrapolationdown} and  Theorem \ref{theorem:extrapolationup} in the homogenous setting (see Section \ref{section:homogeneoustype}). In Theorems \ref{theorem:extrapolationdown} and  \ref{theorem:extrapolationup} we prove a general extrapolation result for singular $\gamma$-kernels. In this setting we also obtain the end point estimates $L^2\to L^{2,\infty}$ and $L^\infty\to \BMO$. The results are a stochastic version of the classical extrapolation results for Calder\'on--Zygmund operators (see \cite{Ho60} for the scalar case and \cite{BCP62,GR85} for the vector-valued case).

The conditions \eqref{eq:horm1intro} and \eqref{eq:horm2intro} are $L^2$-variants of what is usually called the H\"ormander condition. The $L^r$-variant for $r\in [1, \infty]$ also appears in \cite[Definition 2.1]{Ho60} in the scalar case and in \cite[Section 5.1]{RV17} in the vector-valued case, where it was used to extrapolate (deterministic) boundedness of operators from $L^p$ into $L^q$ with $\frac{1}{p} - \frac1q = \frac1r$ to other pairs $(u,v)$ satisfying $1<u\leq v<\infty$ and $\frac1u - \frac1v = \frac1r$.

\subsection{Weighted boundedness}
Next we will state a result on weighted boundedness of stochastic singular integral operators.  For deterministic Calder\'on-Zygmund operators satisfying the standard conditions, this result was obtained in \cite{Hy12}. It settles the so-called $A_2$-conjecture for standard Calder\'on-Zygmund operators and states that under standard assumptions on the kernel $K$ one has for all $p\in (1, \infty)$
\begin{align}\label{eq:A2hyt}
\nrms{s \mapsto \int_{\R^d} K(s,t)f(t)\dd t}_{L^p(\R^d,w)} \leq C_p [w]_{A_p}^{\max\{1,1/(p-1)\}} \|f\|_{L^p(\R^d,w)}.
\end{align}
The bound \eqref{eq:A2hyt} with a non-optimal dependence on the weight characteristic has been known for a much longer time (see \cite{GR85} and references therein).
Originally the $A_2$-conjecture was formulated for the Beurling--Ahlfors transform in \cite{AIS01} where it is shown to imply quasiregularity of certain complex functions. Shortly afterwards it was settled for this operator in \cite{PV02} and subsequently many other operators were treated, which eventually led to \cite{Hy12}.

A new proof was obtained in \cite{Le13a} where it was shown that any standard Calder\'on-Zygmund operator can be pointwise dominated by a positive sparse operator. A further extension to Calder\'on-Zygmund operators satisfying a weaker regularity condition (the so-called Dini condition) was obtained in \cite{La17b}. During the last few years simplified proofs have been obtained by several authors. For our purposes the method in \cite{LO19}, generalized to our setting in \cite{Lo19b}, is the most suitable and it can be used to obtain the following stochastic version of the $A_2$-theorem:

\begin{theorem}[Sharp weighted bounds]\label{theorem:sharpweighted}
Let $X$ and $Y$ be Banach spaces with type $2$ and assume $Y$ is a UMD space. Let $K:\R_+\times\R_+\to \mc{L}(X,Y)$ be strongly measurable and assume that
  \begin{align*}
   \nrm{K(s,t)-K(s',t)} &\leq \omega\has{\frac{\abs{s-s'}}{\abs{s-t}}}\frac{1}{\abs{s-t}^{1/2}} \quad &&\abs{s-s'}\leq \frac{1}{2}\abs{s-t},\\
    \nrm{K(s,t)-K(s,t')} &\leq \omega\has{\frac{\abs{t-t'}}{\abs{s-t}}}\frac{1}{\abs{s-t}^{1/2}} \quad &&\abs{t-t'}\leq \frac{1}{2}\abs{s-t},
  \end{align*}
where $\omega:[0,1]\to [0,\infty)$ is increasing and subadditive, $\omega(0)=0$ and
\begin{equation*}
\has{\int_0^1\omega(r)^2\frac{\dd r}{r}}^{1/2} <\infty.
\end{equation*}
Suppose $S_K$ as defined in \eqref{eq:SKdefintro} is bounded from $L^p_{\ms{F}}(\Omega\times \R_+;\gamma(H,X))$ into $L^p(\Omega\times \R_+;Y)$ for some $p\in [2, \infty)$. Then
$S_K$ is bounded from $L^q_{\ms{F}}(\Omega\times \R_+,w;\gamma(H,X))$ into $L^q(\Omega\times \R_+,w;Y)$ for all $q\in (2, \infty)$ and $w\in A_{q/2}(\R_+)$, and the following weighted bound on the operator norm holds
\[\|S_K\|_{L^q\to L^q}\lesssim [w]_{A_{q/2}}^{\max\{\frac12,\frac{1}{q-2}\}}.\]
\end{theorem}
The above result follows from Theorem \ref{theorem:sparsedomination} in the homogenous setting (see Section \ref{section:homogeneoustype}), which is deduced from a sparse domination result. Furthermore we prove that the above estimate is sharp in terms of the dependence on the weight characteristic. Note that the difference with \eqref{eq:A2hyt} occurs because the $L^p$-norm of \eqref{eq:SKdefintro} is equivalent to a certain generalized square function. The conditions \eqref{eq:dini2} and \eqref{eq:dini3} together with the integrability condition on $\omega$ are $L^2$-versions of the so-called Dini condition. The integrability condition on $\omega$ holds in particular if $\omega(t) = Ct^{\epsilon}$ for some $C>0$ and $\epsilon\in (0,1/2]$.

\subsection{Consequences for stochastic maximal $L^p$-regularity}\label{subsection:SMRintro}

From Theorem \ref{theorem:CZintro} we find that in many instances stochastic maximal $L^p$-regularity for some $p\in [2, \infty)$ implies stochastic maximal $L^q$-regularity for all $q\in (2, \infty)$ (see Section \ref{section:pindSMR}). In order to state a particular result here consider the following stochastic evolution equation on a Banach space $X$:
\begin{equation}\label{eq:SEEintro}
du(t) + A u(t) \dd t  = G(t) d W_H(t),  \ \ \ u(0) = 0.
\end{equation}
If $-A$ generates a $C_0$-semigroup $(e^{-tA})_{t\geq 0}$ on $X$, then the solution $u$ is given by
\[u(t) = \int_0^t e^{-(t-s)A} G(s) \dd W_H(s).\]

\begin{theorem}\label{theorem:SMRintro}
Assume $-A$ is the generator of a bounded $C_0$-semigroup on a UMD Banach space $X$ with type $2$. Let $p\in [2, \infty)$ and $I = (0,T)$ with $T\in (0,\infty]$. Suppose that for all $G\in L^p_{\ms{F}}(\Omega\times (0,T);\gamma(H,X))$ the solution $u$ to \eqref{eq:SEEintro} satisfies
\[\|A^{\frac12} u\|_{L^p(\Omega\times I;X)}\lesssim \|G\|_{L^p_{\ms{F}}(\Omega\times I;\gamma(H,X))}.\]
Then for all $q\in (2, \infty)$, $w\in A_{q/2}(I)$
and $G\in L^q_{\ms{F}}(\Omega\times (0,T),w;\gamma(H,X))$ the solution to \eqref{eq:SEEintro} satisfies
\[\|A^{\frac12} u\|_{L^q(\Omega\times I,w;X)}\lesssim \|G\|_{L^q_{\ms{F}}(\Omega\times I,w;\gamma(H,X))}.\]
\end{theorem}
The boundedness of the semigroup is only needed if $T=\infty$.
A more general result is contained in Theorem \ref{theorem:SMRsemigroup} below. For this we should note that the above $L^p$-estimate implies sectoriality of angle $<\pi/2$ (see \cite[Theorem 4.1]{AV19}). Theorem \ref{theorem:SMRintro} can be seen as the stochastic analogue of \cite[Theorem 7.1]{Do00}. Moreover, the weighted estimates are a stochastic version of \cite[Corollary 4]{CF14} and \cite[Theorem 5.1]{CK18}.

For many differential operators $A$ one can directly apply the results in \cite{NVW12,NVW15b} to obtain stochastic maximal $L^p$-regularity. However, there are numerous situations where this is not the case, for example if:
\begin{enumerate}[(i)]
\item \label{it:no1} $A$ does not have a bounded $H^\infty$-calculus;
\item \label{it:no2} There is no explicit characterization of $D(A^{1/2})$ known;
\item \label{it:no3} $A(t)$ and its domain $D(A(t))$ are time-dependent;
\item \label{it:no4} $X$ does not satisfy the $\mc{R}$-boundedness condition of \cite{NVW12,NVW15b}.
\end{enumerate}
In Corollary \ref{corollary:HScaseSMR} and Remark \ref{remark:noHinftymagreg} we give a situation where \ref{it:no1} occurs, i.e. we give an example of an operator $A$ without a bounded $H^\infty$-calculus which has stochastic maximal $L^p$-regularity. In Example \ref{ex:wedge}  it seems unknown if \ref{it:no1}  holds and \ref{it:no2} seems unavailable as well. In Subsection \ref{subsection:nonauto} we present applications to certain non-autonomous problems where \ref{it:no3} occurs and in Theorem \ref{theorem:realinterpMR} we have avoided the geometric restriction mentioned in \ref{it:no4}.

Another, important novelty is that it is possible to deduce mixed $L^p(L^q)$-boundedness from $L^2(L^2)$-boundedness by combining our extrapolation result with the stochastic-deterministic extrapolation result of \cite{KK16}. This reduces the study of the stochastic maximal $L^p$-regularity problem on $L^q$ to the study of the maximal $L^2$-regularity problem on $L^2$ and Green's function estimates (see Remark \ref{remark:KimKim} and Examples \ref{ex:BesselHeat} and \ref{example:domainneumann}).

The use of temporal $A_{p/2}(\R_+)$-weights in stochastic maximal $L^p$-regularity is new. In most of the results in \cite{NVW12,NVW15b} such weights can also be added without causing major difficulties, but it is very natural to deduce this from extrapolation theory. Moreover with our method we actually obtain a sharp dependence on the $A_{p/2}$-characteristic. Weights of the form $t^{\alpha}$ have already been considered before in \cite{AV19,PV18} and can be used to allow rough initial data in stochastic evolution equations. This has become a central tool in deterministic evolution equations (see e.g.\ \cite{KPW10,MS12,PS04} and references therein). General $A_p$-weights in parabolic PDEs have been used in \cite{DK18,DK19,GLV16,GV17,GV17b} to derive mixed $L^p(L^q)$-regularity estimates by Rubio de Francia's weighted extrapolation theorem \cite{GR85,CMP11}.

In Theorem \ref{theorem:CZintro} and Theorem \ref{theorem:sharpweighted} one always starts from an $L^p$-bounded stochastic integral operator. It would be interesting to find general sufficient conditions from which boundedness can be derived. In the deterministic case this can be done using $T(1)$ and $T(b)$-theorems (see e.g.\ \cite{HW06, Hy06, HH16} for the vector-valued case).
At least in the Hilbert space setting in the convolution case we obtain a full characterization in Corollary \ref{corollary:Hilbertspacesch} and Corollary  \ref{corollary:Hilbertspaceschweight} assuming a H\"ormander and Dini condition, respectively. Finally we mention that it would be interesting to develop a stochastic Calder\'on--Zygmund theory for noises other than cylindrical Brownian motion.

This paper is organized as follows:
\begin{itemize}
\item In Section \ref{section:preliminaries} we give some preliminaries on Banach space geometry, $\gamma$-radonifying operators, Lorentz spaces, maximal operators and Muckenhoupt weights.
\item In Section \ref{section:SIO} we introduce the stochastic integral operators and establish a connection with $\gamma$-integral operators.
\item In Section \ref{section:kernels} we give the definitions of $2$-H\"ormander kernels and  $(\omega,2)$-Dini kernels and provide some classes of examples.
\item In Section \ref{section:extrapolation} we prove the extrapolation results for $\gamma$-integral operators under a H\"ormander condition.
\item In Section \ref{section:sparse} we obtain sparse domination under a Dini condition and use this to prove sharp weighted bounds.
\item In Section \ref{section:homogeneoustype} we explain how the results of Sections \ref{section:extrapolation} and \ref{section:sparse} can be extended to spaces of homogeneous type. Motivated by the application to stochastic integral operators our main example here is the time interval $(0,T)$ with $T\in (0,\infty]$.
\item In Section \ref{section:pindSMR} we will apply the results of the previous sections to study the $p$-independence of stochastic maximal $L^p$-regularity. Here we cover both the time-independent setting and the time-dependent setting using the conditions of Acquistapace and Terreni. Moreover, applications to the (time-dependent) heat equation are given, leading to regularity results in both the complex and real interpolation scale.
\item In Section \ref{section:Rbddness} we prove a $p$-independence result on the $\mc{R}$-boundedness of stochastic convolutions.
\item Finally in Appendix \ref{section:technical} we prove some technical kernel estimates.
\end{itemize}

\subsection*{Notation}
We denote the Lebesgue measure of a set $E \in \mc{B}(\R^d)$ by $\abs{E}$ and we often abbreviate the integral of a function $f$ on $E$ as $\int_E f := \int_E f(s) \dd s$  and the mean as
$\avint_E f := \frac{1}{\abs{E}}\int f$. A ball with center $s$ and radius $r$ is denoted by $B(s,r)$ and by a cube $Q \subseteq \R^d$ we mean a cube with its sides parallel to the coordinate axes.

For Banach spaces $X$ and $Y$, $\mc{L}(X,Y)$ denotes the bounded linear operators from $X$ to $Y$. If we say that a function $f:\R^d \to \mc{L}(X,Y)$ is strongly measurable, we mean that $K$ is strongly measurable in the strong operator topology on $\mc{L}(X,Y)$.

Throughout the paper we write $C_{a,b,\cdots}$ to denote a constant, which only depends on
the parameters $a,b,\cdots$ and which may change from line to line. By $\lesssim_{a,b,\cdots}$ we mean that there is a constant $C_{a,b,\cdots}$ such that inequality holds and $\eqsim_{a,b,\cdots}$ implies $\lesssim_{a,b,\cdots}$ and $\gtrsim_{a,b,\cdots}$.

\subsubsection*{Acknowledgements}
The authors thank the anonymous referees for careful reading and helpful comments and Petru Cioica-Licht for pointing out a gap in the proof of Proposition \ref{proposition:propertiesheatwedge}\ref{it:wedgesectorial}.

\section{Preliminaries}\label{section:preliminaries}
\subsection{Banach space geometry}
Let $X$ be a Banach space and let $(\varepsilon_k)_{k=1}^\infty$ a sequence of independent \emph{Rademacher variables}, i.e. uniformly distributed random variables taking values in $\cbrace{z \in \K:\abs{z} \leq 1}$. We say that $X$ has type $p \in [1,2]$ is there exists a constant $C \geq 0$ such that for all $n \in \N$ and $x_1,\cdots,x_n \in X$ we have
\begin{equation}\label{eq:type}
  \has{\E\nrms{\sum_{k=1}^n\varepsilon_k x_k}_X^p}^{1/p} \leq C \, \has{\sum_{k=1}^n \nrm{x_k}_X^p}^{1/p}.
\end{equation}
We say that $X$ has cotype $q \in [2,\infty]$ if there exists a constant $C \geq 0$ such that for all $n \in \N$ and $x_1,\cdots,x_n \in X$ we have
\begin{equation}\label{eq:cotype}
   \has{\sum_{k=1}^n \nrm{x_k}_X^q}^{1/q}\leq C \,  \has{\E\nrms{\sum_{k=1}^n\varepsilon_k x_k}_X^q}^{1/q}.
\end{equation}
with the usual modification if $q=\infty$. The least admissible constants $C$ will be denoted by $\tau_{p,X}$ and $c_{q,X}$ respectively. Note that by randomization equations \eqref{eq:type} and \eqref{eq:cotype} imply the same estimates with the Rademacher sequence replaced by a \emph{Gaussian sequence}, i.e. a sequence of independent standard Gaussian random variables.

All Banach spaces have type $1$ and cotype $\infty$. We say that $X$ has \emph{nontrivial type} if $X$ has type $p \in (1,2]$ and finite cotype if $X$ has cotype $q \in [2,\infty)$. As example we note that the Lebesgue spaces $L^p(\R^d)$ and Sobolev spaces $W^{k,p}(\R^d)$ have type $p \wedge 2$ and cotype $p \vee 2$. For more details and examples we refer to \cite[Chapter 7]{HNVW17}.

We say that a Banach space $X$ has the $\UMD$ property if the martingale difference sequence of any finite martingale in $L^p(\Omega;X)$ is unconditional for some (equivalently all) $p \in (1,\infty)$. That is, if
 there exists a constant $C>0$ such that for all finite martingales $(f_k)_{k=1}^n$ in $L^p(\Omega;X)$ and scalars $|\epsilon_k|=1$, $k=1,\cdots,n$, we have
 \begin{equation}\label{eq:UMD}
   \has{\E\nrms{\sum_{k=1}^n \epsilon_kdf_k}_X^p}^{1/p} \leq C \, \has{\E\nrms{\sum_{k=1}^n df_k}_X^p}^{1/p}.
 \end{equation}
 The least admissible constant $C$ in \eqref{eq:UMD} will be denoted by $\beta_{p,X}$.
 Standard examples of Banach spaces with the $\UMD$ property include reflexive $L^p$-spaces, Lorentz spaces, Sobolev spaces and Besov spaces. For a thorough introduction to the theory of $\UMD$ spaces we refer the reader to \cite{Pi16} and \cite[Chapter 4]{HNVW16}.

\subsection{$\gamma$-radonifying operators} We recall the definition and some basic properties of $\gamma$-radonifying operators, for details we refer to \cite[Chapter 9]{HNVW17}.

Let $X$ be a Banach space and $H$ be a Hilbert space. For $x \in X$ and $e \in H$ we let $e \otimes x$ be the rank-one operator from $H$ to $X$ given by $h \mapsto \ip{h,e}x$. The \emph{$\gamma$-radonifying norm} of a finite-rank operator of the form $T=\sum_{k=1}^n e_k \otimes x_k$, with $e_1,\cdots,e_n \in H$ orthonormal and $x_1,\cdots,x_n \in X$, is defined by
\begin{equation*}
  \nrm{T}_{\gamma(H;X)} := \nrms{\sum_{k=1}^n \gamma_kx_k}_{L^2(\Omega;X)},
\end{equation*}
where $(\gamma_k)_{k=1}^n$ is a Gaussian sequence on a probability space $(\Omega,\P)$. By the invariance of Gaussians in $\R^d$ under orthogonal transformations (see \cite[Proposition 6.1.23]{HNVW16}), this norm is well-defined. The completion of all finite rank operators from $H$ into $X$ with respect to $\nrm{\cdot}_{\gamma(H;X)}$ is denoted by $\gamma(H,X)$. Note that in particular $\gamma(H,X)\hookrightarrow \mc{L}(H,X)$.

For a measure space $(S,\mu)$, we write $\gamma(S;H,X):=\gamma(L^2(S;H),X)$ and in particular $\gamma(S;X):=\gamma(L^2(S),X)$. Any strongly measurable $f\colon S \to X$, for which $\ip{f,x^*} \in L^2(S)$ for all $x^* \in X^*$, defines a bounded linear operator $T_f\colon L^2(S) \to X$ by
\begin{equation*}
  T_f \varphi := \int_S f\varphi \dd \mu, \qquad \varphi \in L^2(S),
\end{equation*}
where the integral is well-defined in the Pettis sense (see \cite[Theorem 1.2.37]{HNVW16}). If $T_f \in \gamma(S;X)$ we say that $f$ represents $T_f$ and write $f \in  \gamma(S;X)$.

If the Banach space $X$ has type $2$ we have the following embedding properties for the $\gamma$-spaces, which follow directly from  \cite[Theorem 9.2.10 and Proposition 7.1.20]{HNVW17}. See \cite[Proposition 2.5]{AV19} for the details.
\begin{lemma}\label{lemma:gammaL2}
Let $X$ be a Banach space with type $2$, $H$ a Hilbert space and $(S,\mu)$ a $\sigma$-finite measure space. Then we have the embeddings
\begin{align*}
L^2(S;\gamma(H;X)) &\hookrightarrow \gamma(S;\gamma(H;X)) \hookrightarrow \gamma(L^2(S;H);X)
\end{align*}
with both embedding constants bounded by $\tau_{2,X}$.
\end{lemma}
Finally, for our $\gamma$-version of the $A_2$-theorem in Section \ref{section:sparse} we will need the following lemma, which follows directly from \cite[Proposition 9.4.13]{HNVW17}.
\begin{lemma}\label{lemma:2sublinear}
Let $X$ be a Banach space with type $2$ and let $f_1,\cdots,f_n \in \gamma(S;X)$ be disjointly supported. Then we have
\begin{equation*}
  \nrms{\sum_{k=1}^n f_k}_{\gamma(S;X)}^2 \leq \tau_{2,X} \sum_{k=1}^n\nrm{f_k}_{\gamma(S;X)}^2.
\end{equation*}
\end{lemma}

\subsection{Lorentz spaces} \label{section:lorentz}We recall the definition and some elementary properties of Lorentz spaces,
for details we refer to  \cite{Gr14a, Tr78}. For a Banach space $X$ a $\sigma$-finite measure space $(S,\mu)$, $p\in (1,\infty)$ and $q\in [1, \infty]$ let
$$L^{p,q}(S;X) = (L^{1}(S;X),L^\infty(S;X))_{\theta,q}, \qquad \theta = \tfrac{1}{p'}.$$
The space $L^{p,q}(S;X)$ is called the $X$-valued {\em Lorentz space}. An equivalent {\em quasi-norm} is given by (see \cite[Proposition 1.4.9]{Gr14a} and \cite[Theorem 1.18.6]{Tr78}):
\begin{align*}
\nrmn{ f}_{L^{p,q}(S;X)}
 & := \|t\mapsto t^{1/p} f^*(t)\|_{L^q(\R_+,\frac{dt}{t})}
\\ &\phantom{:}=p^{1/q} \nrmb{t\mapsto t \mu\hab{\{s\in S:\|f(s)\|_X>t\}}^{1/p}}_{L^q(\R_+,\frac{dt}{t})},
 \end{align*}
where $f^*$ denotes the decreasing rearrangement of $\nrm{f}_X$ (see \cite[Section 1.4]{Gr14a}). An equivalent {\em norm} can be extracted from \cite[Exercise 1.4.3]{Gr14a}.
For $p\in (1, \infty)$ one has $L^{p,p}(S;X) = L^p(S;X)$. In the scalar case $L^{p,q}(S)$ is a Banach function space.

If $p\in (1, \infty)$ and $q\in [1, \infty)$, then the simple functions are dense in $L^{p,q}(S;X)$. Indeed, this follows from \cite[Theorems 1.6.2 and 1.18.6.2]{Tr78} and the density of the simple functions in $L^{r}(S;X)$ for $r\in [1, \infty)$.

If $p\in (1, \infty)$ and $q\in [1, \infty]$ and $\mu(S) <\infty$ we have $L^{p,\infty}(S) \hookrightarrow L^1(S)$ with
\begin{equation}\label{eq:weakLpembeddingL1}
  \nrm{f}_{L^1(S)} \leq \mu\ha{S}^{1/p'}\nrm{f}_{L^{p,\infty}(S)} ,\qquad f \in L^{p,\infty}(S),
\end{equation}
which follows directly from the definition and the embedding $L^\infty(S) \hookrightarrow L^1(S)$ with constant $\mu(S)$.

In the next result we extend the $\gamma$-Fubini theorem of \cite[Theorem 9.4.8]{HNVW17} to Lorentz spaces.
\begin{proposition}[$\gamma$-Fubini]\label{proposition:gammaFubLorentz}
Let $X$ be a Banach space and let $(S,\mu)$ be a measure space. Then the following assertions hold:
\begin{enumerate}[(i)]
\item\label{it:generalgammaFub} For all $p\in (1, \infty]$, $$\gamma(H,L^{p,\infty}(S;X)) \hookrightarrow L^{p,\infty}(S;\gamma(H,X))$$
\item\label{it:concgammaFub}  For all $p\in (1, \infty)$ and $q\in [1, \infty)$, $$\gamma(H,L^{p,q}(S;X)) = L^{p,q}(S;\gamma(H,X))$$ isomorphically.
\end{enumerate}
\end{proposition}
\begin{proof}
Let $E = L^{p,q}(S)$ and for a Banach space $Y$ we write $E(Y)$ for the space of strongly measurable functions $f:S\to Y$ such that $\|f\|_{E(Y)}:=\big\|\|f\|_Y \big\|_E<\infty$.

We make two preliminary observations.
Since $E$ is a Banach space the triangle inequality in $E$ implies that for all simple functions $\xi:\Omega\to E$,
\begin{equation}\label{eq:1convex}
  \|\xi\|_{E(L^1(\Omega))}\leq \|\xi\|_{L^1(\Omega;E)}.
\end{equation}
By density this extends to a contractive embedding $L^1(\Omega;E)\hookrightarrow E(L^1(\Omega))$.

The second observation is a certain converse estimate to the above. If $p\leq q<\infty$, we set $r=q$ and if $q\leq p<\infty$ we set $r = p+1$. Then $E$ is $r$-concave (see \cite[Theorems 4.6 and 5.1]{Ma04}). This implies that for all simple functions $\xi:S\to L^r(\Omega)$,
\begin{equation}\label{eq:rconcave}
  \|\xi\|_{L^r(\Omega;E)} \leq  C_{p,q}\,\|\xi\|_{E(L^r(\Omega))}.
\end{equation}
By density this can be extended to a contractive embedding $E(L^r(\Omega))\hookrightarrow L^r(\Omega;E)$.

Let $(h_j)_{j=1}^n$ be an orthonormal system in $H$ and let $f = \sum_{j=1}^n h_j\otimes \xi_{j}$ with $\xi_j\in E(X)$.
Now setting $\xi = \big\|\sum_{j=1}^n \gamma_j \xi_{j}\big\|_X$, where $(\gamma_j)_{j=1}^n$ is a Gaussian sequence, we can write
\begin{align*}
\|f\|_{\gamma(H,E(X))} &= \Big\|\sum_{j=1}^n \gamma_j \xi_{j}\Big\|_{L^2(\Omega;E(X))} = \|\xi\|_{L^2(\Omega;E)},
\\ \|f\|_{E(\gamma(H,X))} &= \Big\|\sum_{j=1}^n \gamma_j \xi_{j}\Big\|_{E(L^2(\Omega;X))} = \|\xi\|_{E(L^2(\Omega))}
\end{align*}
By the Kahane--Khintchine inequalities replacing the $L^2(\Omega)$-norm on the right-hand sides of the above identities with $L^r(\Omega)$ with $r\in [1, \infty)$ leads to an equivalent norm. Taking $r=1$ we have by \eqref{eq:1convex}  that $$\|f\|_{E(\gamma(H,X))} \leq C\,  \|f\|_{\gamma(H,E(X))},$$ which by density proves $\gamma(H,E(X))\hookrightarrow E(\gamma(H,X))$ and this proves \ref{it:generalgammaFub} and one of the embeddings in \ref{it:concgammaFub}.

To prove \ref{it:concgammaFub} note that by the above with $r=p+1$  we find by \eqref{eq:rconcave}  that $$\|f\|_{\gamma(H,E(X))}\leq C_{p,q} \, \|f\|_{E(\gamma(H,X))}.$$  Again by density  this gives $E(\gamma(H,X))\hookrightarrow \gamma(H,E(X))$.
\end{proof}

\begin{remark}
Actually in Proposition \ref{proposition:gammaFubLorentz}\ref{it:generalgammaFub} the Lorentz space can be replaced by any Banach function space $E$. Moreover, the extension of \ref{it:concgammaFub} to this setting holds if $E$ is $r$-concave for some $r<\infty$.

The result of Proposition \ref{proposition:gammaFubLorentz} can also be extended to quasi-Banach function spaces which are $p$-convex and $q$-concave. For the definition of $\gamma(H,Y)$ for quasi-Banach spaces we refer to \cite{CCV18}. In particular, by \cite[Theorems 4.6 and 5.1]{Ma04} and \cite[Section 6]{Ka80} it follows that Proposition \ref{proposition:gammaFubLorentz}\ref{it:concgammaFub} holds for $L^{p,q}(S;X)$ for all $p,q\in (0,\infty)$.
\end{remark}

\subsection{Maximal operators}
We define the \emph{Hardy--Littlewood maximal operator} $M$ for an $f \in L^1_{\loc}(\R^d)$  by
\begin{equation*}
  Mf(s) := \sup_{Q \ni s} \avint_Q \abs{f},\qquad s \in \R^d,
\end{equation*}
where the supremum is taken over all cubes $Q \subseteq \R^d$ containing $s$. For $r \in (0,\infty)$ and $f \in L^r_{\loc}(\R^d)$ we define $M_rf = M(\abs{f}^r)^{1/r}$.
These operators satisfy the following bounds:
\begin{lemma}\label{lemma:maximaloperator} Let $0<r<p<\infty$, then
\begin{align*}
  \nrm{M_rf}_{L^p(\R^d)} &\leq C_{p,r,d} \,\nrm{f}_{L^p(\R^d)}, &&f \in L^p(\R^d),\\
  \nrm{M_rf}_{L^{p,\infty}(\R^d)} &\leq C_{p,r,d}\,\nrm{f}_{L^{p,\infty}(\R^d)}, &&f \in L^{p,\infty}(\R^d),\\
  \nrm{M_pf}_{L^{p,\infty}(\R^d)} &\leq C_d\,\nrm{f}_{L^p(\R^d)}, &&f \in L^p(\R^d).
\end{align*}
\end{lemma}
The case $r=1$ in the first two inequalities and the case $p=1$ in the third inequality follow for example from Doob's maximal inequalities and a covering argument (see \cite[Theorem 3.2.3 and Lemma 3.2.26]{HNVW16}). The general cases follow from a simple rescaling argument.

Let $X$ be a Banach space. We define the \emph{sharp maximal operator} for an $f \in L^1_{\loc}(\R^d;X)$  by
\begin{equation*}
  M^{\#}f(s) := \sup_{Q \ni s} \avint_Q \nrms{f(t)-\avint_Q f(t')\dd t'}_X\dd t,\qquad s \in \R^d,
\end{equation*}
where the supremum is again taken over all cubes $Q \subseteq \R^d$ containing $s$. Note that  it is immediate from this definition that $M^{\#}f \leq 2 M(\nrm{f}_X)$, so by Lemma \ref{lemma:maximaloperator} we have in particular that $M^{\#}f \in L^p(\R^d)$ if $f \in L^p(\R^d;X)$. The converse is also true, which is the content of the next lemma. The proof for the case $X = \C$ can be found in \cite[Corollary 3.4.6]{Gr14b}, the general case follows analogously replacing absolute values by norms.

\begin{lemma}[Fefferman-Stein]\label{lemma:sharpmaximal}
  Let $X$ be a Banach space, $1<r\leq p<\infty$ and $f \in L^r(\R^d;X)$. Then $f \in L^p(\R^d;X)$ if and only if $M^{\#}f \in L^p(\R^d)$ and
  \begin{equation*}
    C_{p,d}^{-1}\, \nrmb{M^{\#}f}_{L^p(\R^d)} \leq\nrm{f}_{L^p(\R^d;X)} \leq C_{p,d} \, \nrmb{M^{\#}f}_{L^p(\R^d)}.
  \end{equation*}
\end{lemma}

Lemma \ref{lemma:sharpmaximal} is not valid for $p=\infty$. In this case the space of all $f \in L^1_{\loc}(\R^d;X)$ such that $M^{\#}f \in L^\infty(\R^d;X)$ is strictly larger than $L^\infty(\R^d;X)$. We let $\BMO(\R^d;X)$ be the space of all $f \in L^1_{\loc}(\R^d;X)$ such that
\begin{equation*}
  \nrm{f}_{\BMO(\R^d;X)}:= \sup_{Q} \inf_{c \in X} \avint_Q\nrm{f(s)-c}_X\dd s<\infty
\end{equation*}
where the supremum is taken over all cubes $Q \subseteq \R^d$. In analogy with  Lemma \ref{lemma:sharpmaximal} we have
\begin{equation*}
    \frac{1}{2}\nrm{M^\#f}_{L^\infty(\R^d)} \leq \nrm{f}_{\BMO(\R^d;X)} \leq \nrm{M^\#f}_{L^\infty(\R^d)}.
\end{equation*}
Note that $\nrm{\cdot}_{\BMO(\R^d;X)}$ is not a norm, since $\nrm{c \ind_{\R^d}}_{\BMO(\R^d;X)} = 0$ for any $c \in X$.

\subsection{Muckenhoupt weights} We recall the basic properties of  Muckenhoupt weights on $\R^d$, for a general overview see \cite[Chapter 7]{Gr14a}. Analogous definitions can be given for weights on $(0,T)$ for $T \in (0,\infty]$.

A \emph{weight} is a locally integrable function $w\colon\R^d\to (0,\infty)$. For $p \in [1,\infty)$ and a weight $w$ and a Banach space $X$ we let $L^p(\R^d,w;X)$ be the subspace of all $f \in L^0(\R^d;X)$ such that
\begin{equation*}
  \nrm{f}_{L^p(\R^d,w;X)}:= \has{\int_{\R^d}\nrm{f}_X^pw}^{1/p}<\infty.
  \end{equation*}
and let $L^{p,\infty}(\R^d,w;X)$ be defined as in Section \ref{section:lorentz}.
We will say that a weight $w$ lies in the \emph{Muckenhoupt class $A_p$} and write $w\in A_p$ if it satisfies
\[
[w]_{A_p}:=\sup_{Q}\avint_Q w \cdot \has{\avint_Q w^{1-p'}}^{p-1}<\infty,
\]
where the supremum is taken over all cubes $Q\subseteq\R^d$ and the second factor is replaced by $(\essinf_Q w)^{-1}$ if $p=1$. Note that $w \in A_p$ if and only if $w^{1-p'} \in A_{p'}$ with $[w]^{\frac{1}{p}}_{A_p} = [w^{1-p'}]_{A_{p'}}^{\frac{1}{p'}}$ for $p \in (1,\infty)$.

We will say that a weight $w$ lies in  $A_\infty$ and write $w\in A_\infty$ if
\begin{equation*}
  [w]_{A_\infty} := \sup_Q \frac{\int_Q M(w \ind_Q)   }{\int_Qw  } <\infty,
\end{equation*}
where the supremum is taken over all cubes $Q\subseteq\R^d$. Then $A_\infty  = \bigcup_{p\geq 1} A_p$ and for all $w \in A_p$ we have
\begin{equation*}
  [w]_{A_\infty} \leq C_d \, [w]_{A_p},
\end{equation*}
See e.g. \cite{HP13} for the proof of these facts and a more thorough introduction of the Fuji--Wilson $A_\infty$-characteristic.

\section{Stochastic integral operators}\label{section:SIO}

For details of the introduced notions in this section we refer to \cite{NVW07, NVW15b}.
Let $X$ be a Banach space and $H$ be a Hilbert space. Let $(\Omega, \mathcal{A}, \P)$ be a probability space with filtration $(\ms{F}_t)_{t\geq 0}$.

Let $\mathcal{W}\in \mc{L}(L^2(\R_+;H),L^2(\Omega))$ denote an {\em isonormal mapping} (see \cite{Kal02}) such that $\mathcal{W} f$ is $\ms{F}_t$-measurable if $f=0$ on $(t, \infty)$. Define a {\em cylindrical Brownian motion} $(W_H(t))_{t\geq 0}$ by $W_H(t) h := \mathcal{W} (\ind_{[0,t]} h)$.

For $f = \ind_{(a,b]} \otimes h \otimes \xi$, where $0\leq a<b<\infty$ and $\xi\in L^\infty(\Omega;X)$ is strongly $\ms{F}_a$-measurable, define
\[\int_0^s f(t) \dd W_H(t) := (W_H(b\wedge s) - W_H(a\wedge s))h \otimes \xi\in L^p(\Omega;X).\]
for each $s\geq 0$. The functions in the linear span of such $f$ are called the {\em finite rank adapted step processes}. We extend the definition of the stochastic integral by linearity.

For $T\in (0,\infty]$ and $p\in [1, \infty)$, we let $L^p_{\ms{F}}(\Omega;\gamma(0,T;H,X))$ denote the closure of all finite rank adapted step processes in $L^p(\Omega;\gamma(0,T;H,X))$. One has that $f\in L^p_{\ms{F}}(\Omega;\gamma(0,T;H,X))$ if and only if
$f (\ind_{[0,s]}\otimes h)$ is strongly $\ms{F}_s$-measurable and $f\in L^p(\Omega;\gamma(0,T;H,X))$ for all $s\in (0,T)$ and $h\in H$. The following result provides two-sided estimates for the stochastic integral with respect to an $H$-cylindrical Brownian motion $(W_H(t))_{t\geq 0}$.
\begin{theorem}[It\^o isomorphism]\label{theorem:Ito}
Let $X$ be a $\UMD$ Banach space, let $p\in (1, \infty)$ and $T\in (0,\infty]$.
For every adapted finite rank step process $f:(0,T)\times \Omega\to \mc{L}(H,X)$, one has
\[\Big\|\int_0^T f(t) \dd W_H(t)\Big\|_{L^p(\Omega;X)} \eqsim_{p,X} \|f\|_{L^p(\Omega;\gamma(0,T;H,X))}. \]
In particular, the mapping $f\mapsto \int_0^T f(t) \dd W_H(t)$ extends to an isomorphism from $L^p_{\ms{F}}(\Omega;\gamma(0,T;H,X))$ to $L^p(\Omega;X)$.
\end{theorem}

\subsection{Stochastic integral operators\label{section:SII}}
For $p\in [2, \infty)$, $T \in (0,\infty]$ and a weight $w$ on $(0,T)$, let $L^p_{\ms{F}}(\Omega\times (0,T),w;\gamma(H,X))$ denote the closure of all finite rank adapted step processes in $L^p(\Omega\times (0,T),w;\gamma(H,X))$, where we omit the weight if $w \equiv 1$. The reason we consider $p\in [2, \infty)$ will become clear in Subsection \ref{section:scalarcase}. Although we will not assume type $2$ for the moment, it follows from \cite[Proposition 6.2]{NVW15} that already for very easy kernels $K$ in order to have boundedness of $S_K$ a type $2$ condition on $Y$ is necessary.

\begin{definition}[Stochastic integral operator]\label{definition:SIOW}
Let $X$ be a Banach space and $Y$ a UMD Banach space.   Let $p\in [2, \infty)$, $T \in (0,\infty]$, $w$ be a weight on $(0,T)$ and let
$$K\colon(0,T)\times (0,T)\to \mc{L}(X,Y)$$ be strongly measurable.  We say that $K\in \mc{K}_{W}^H(L^{p}((0,T),w))$ if for $f\in L^p_{\ms{F}}\ha{\Omega\times (0,T), w;\gamma(H,X)}$ and a.e. $s \in (0,T)$ the mapping $t\mapsto K(s,t) f(t)$ is in $L^p_{\ms{F}}(\Omega;\gamma(0,T;H,Y))$ and the operator $S_K$ given by
\begin{align*}
S_K f(s) := \int_{0}^T K(s,t) f(t) \dd W_H(t), \qquad s \in (0,T)
\end{align*}
is bounded from $L^p_{\ms{F}}(\Omega\times (0,T),w;\gamma(H,X))$ into $L^p(\Omega\times (0,T),w;Y)$.
We norm $\mc{K}_{W}^H(L^{p}((0,T),w))$ by
$$\|K\|_{\mc{K}_{W}^H(L^{p}((0,T),w))} := \|S_K\|_{L^p_{\ms{F}}(\Omega\times (0,T),w;\gamma(H,X)) \to L^p(\Omega\times (0,T),w;Y)}.$$
We omit the weight if $w \equiv 1$ and we omit the Hilbert space if $H = \R$.
\end{definition}

We want to study the boundedness properties of $S_K$. In the next results we will reformulate this problem by reducing to the deterministic setting using square functions ($\gamma$-norms in time) and reduce considerations to the case $H = \R$.

\begin{definition}[$\gamma$-integral operator]\label{definition:gammaintegral}
Let $X$ and $Y$ be a Banach spaces.  Let $p\in [2, \infty)$, $w$ be a weight on $\R^d$ and let $$K:\R^d\times \R^d\to \mc{L}(X,Y)$$ be strongly measurable. We say that $K\in \mathcal{K}_\gamma^H(L^{p}(\R^d,w))$ (resp. $K\in \mathcal{K}_\gamma^H(L^{p,\infty}(\R^d,w))$) if for $f \in L^p\ha{\R^d,w;\gamma(H,X)}$ and a.e  $s\in \R^d$ the mapping $t\mapsto K(s,t) f(t)$ is in $\gamma(\R^d;H,Y)$ and the operator $T_K$ given by
\begin{align*}
T_{K} f(s) := K(s,\cdot) f(\cdot), \qquad s \in \R^d
\end{align*}
is bounded from $L^p(\R^d,w;\gamma(H,X))$ into $L^p(\R^d,w;\gamma(\R^d;H,Y))$ (resp. from $L^p$ into $L^{p,\infty}$). We norm these spaces by
\begin{align*}
\|K\|_{\mathcal{K}_\gamma^H(L^{p}(\R^d,w))} & := \|T_K\|_{L^p(\R^d,w;\gamma(H,X))\to L^p(\R^d,w;\gamma(\R^d;H,Y))},
\\ \|K\|_{\mathcal{K}_\gamma^H(L^{p,\infty}(\R^d,w))} & := \|T_K\|_{L^p(\R^d,w;\gamma(H,X))\to L^{p,\infty}(\R^d,w;\gamma(\R^d;H,Y))}.
\end{align*}
We omit the weight if $w \equiv 1$ and we omit the Hilbert space if $H = \R$. We make the same definitions for $\R^d$ replaced by any measure space $(S,\mu)$ in the obvious way.
\end{definition}

We start by connecting the definitions of stochastic and $\gamma$-integral operators.

\begin{proposition}[Deterministic characterization]\label{proposition:detcharacterization}
Let $X$ be a Banach space and $Y$ a UMD Banach space. Let $p\in [2, \infty)$, $T \in (0,\infty]$ and let $w$ be a weight on $(0,T)$.
Then $$\mathcal{K}_{W}^H(L^p((0,T),w)) = \mathcal{K}_\gamma^{H}(L^p((0,T),w))$$ isomorphically.
\end{proposition}
\begin{proof}
The proof is completely straightforward from Theorem \ref{theorem:Ito}. Indeed if $K\in \mathcal{K}_{\gamma}^H(L^p((0,T),w))$, then for $f\in L^p_{\ms{F}}(\Omega\times(0,T),w;\gamma(H,X))$ one has
\begin{align*}
\|S_Kf(s)\|_{L^p(\Omega;Y)} \eqsim_{p,Y} \|T_K f(s)\|_{L^p(\Omega;\gamma(0,T;H,Y))}.
\end{align*}
Therefore by Fubini's theorem we have
\begin{align}
\nonumber \|S_Kf\|_{L^p(\Omega\times(0,T),w;Y))}& \eqsim_{p,Y} \|T_K f\|_{L^p(\Omega;L^{p}(w;\gamma(0,T;H,Y)))}
\\ & \label{eq:estSKTK}  \leq \|K\|_{\mc{K}_\gamma^H(L^p((0,T),w))} \|f\|_{L^p(\Omega;L^p(w;\gamma(H,X)))}.
\\ & =\|K\|_{\mc{K}_\gamma^H(L^p((0,T),w))} \|f\|_{L^p(\Omega\times (0,T),w;\gamma(H,X))} \nonumber
\end{align}
Conversely, taking $f$ independent of $\Omega$, a similar argument yields that $K\in \mathcal{K}_{W}^H(L^p((0,T),w))$ implies $K\in \mathcal{K}_{\gamma}^H(L^p((0,T),w))$.
\end{proof}

In the next result we show that one can take $H = \R$. The result extends \cite[Theorem 5.4]{AV19} where a particular kernel was considered.

\begin{proposition}[Independence of $H$]\label{proposition:independenceH}
Let $X$ and $Y$ be a Banach spaces and $(S,\mu)$ a measure space. Assume $Y$ has type $2$, let $p\in [2, \infty)$ and let $w$ be a weight on $S$.
Then \begin{align*}
  \mathcal{K}_\gamma^{H}(L^p(S,w)) &= \mathcal{K}_\gamma(L^p(S,w))\\
  \mathcal{K}_\gamma^{H}(L^{p,\infty}(S,w)) &= \mathcal{K}_\gamma(L^{p,\infty}(S,w))
\end{align*}isomorphically.
\end{proposition}
\begin{proof}
By considering a $1$-dimensional subspace of $H$, we immediately see that $\subseteq$ holds. For the converse let $T_K^H$ and $T_K^\R$ be the $\gamma$-integral operators on $L^p\ha{S,w;\gamma(H,X)}$ and $L^p\ha{S,w;X}$ respectively.
By Lemma \ref{lemma:gammaL2} one has
\begin{align*}
\|T_K^H f(s)\|_{\gamma(S;H,Y)}
& \lesssim_Y \|T_K^{\R} f(s)\|_{\gamma(S;\gamma(H,Y))}
  = \|T_K^{\R} f(s)\|_{\gamma(H,\gamma(S,Y))}
\end{align*}
Taking $L^p(S,w)$-norms and using Proposition \ref{proposition:gammaFubLorentz}\ref{it:concgammaFub} with $p=q$ we obtain
\begin{equation*}
\begin{aligned}
   \|T_K^H f\|_{L^p(w;\gamma(S;H,Y))} &\lesssim_{Y} \|T_K^{\R} f\|_{L^p(S,w;\gamma(H,\gamma(S;Y)))}
\\  & \eqsim_{p} \|T^{\R}_K f\|_{\gamma(H,L^p(S,w;\gamma(S;Y)))}
\\ & \leq \|K\|_{\mathcal{K}_\gamma(L^p(S,w))} \|f\|_{\gamma(H,L^p(S,w;X))}
\\ & \eqsim_p \|K\|_{\mathcal{K}_\gamma(L^p(S,w))}  \|f\|_{L^p(S,w;\gamma(H,X))}.
\end{aligned}
\end{equation*}
The $L^{p,\infty}$-case follows analogously using
Proposition \ref{proposition:gammaFubLorentz}\ref{it:generalgammaFub} instead.
\end{proof}

\subsection{Truncations}
We will now illustrate a major difference between stochastic and deterministic integral operators. Indeed, we will show that even when the kernel $K$ has a singularity, the ``$\gamma$-integrals'' converge absolutely. In particular, we show that if we truncate the singularity of $K$, then the operators associated to these truncations converge back to the operator associated to $K$ without any regularity assumptions on $K$. This is in contrast to the deterministic setting (see \cite[Section 5.3]{Gr14a}). For this let $X$ and $Y$ be Banach spaces and suppose that $K\colon\R^{d}\times \R^d \to \mc{L}(X,Y)$ is strongly measurable. We define for $\varepsilon>0$
\begin{equation*}
K_{\varepsilon}(s,t) := K(s,t) \ind_{B(0,1/\varepsilon)\setminus B(0,\varepsilon)}(s-t),\qquad s,t\in \R^d.
\end{equation*}
Let $p \in [2,\infty)$ and $w$ a weight on $\R^d$. If $K_{\varepsilon} \in \mc{K}_\gamma(L^p(\R^d,w))$ for all $\varepsilon>0$ we define for $f \in L^p(\R^d,w;X)$ the maximal truncation operator
\begin{equation*}
  T_K^\star f(s) :=  \sup_{\varepsilon>0}\,\nrm{T_{K_\varepsilon}f(s)}_{\gamma(\R^d;Y)} \qquad s \in \R^d.
\end{equation*}

\begin{proposition}[Truncations]\label{proposition:truncations}
Let $X$ and $Y$ Banach spaces and assume that $Y$ has type $2$. Let $p\in [2, \infty)$ and let $w$ be a weight on $\R^d$. Let $$K\colon\R^{d}\times \R^d \to \mc{L}(X,Y)$$ be strongly measurable such that $K_\varepsilon \in \mc{K}_\gamma(L^{p,\infty}(\R^d,w))$ for all $\varepsilon>0$. Then for $f \in L^p(\R^d;X)$ we have
\begin{equation*}
  T_{K}^\star f(s) = \|T_{K} f(s)\|_{\gamma(\R^d;Y)}, \qquad s \in \R^d,
\end{equation*}
and in particular
\begin{align*}
 \|K\|_{\mathcal{K}_\gamma(L^p(\R^d,w))}  &=  \sup_{\varepsilon> 0} \|K_{\varepsilon}\|_{\mathcal{K}_\gamma(L^p(\R^d,w))},
\\   \|K\|_{\mathcal{K}_\gamma(L^{p,\infty}(\R^d,w))} &= \sup_{\varepsilon> 0} \|K_{\varepsilon}\|_{\mathcal{K}_\gamma(L^{p,\infty}(\R^d,w))}.
\end{align*}
Furthermore if $K\in {\mathcal{K}_\gamma(L^{p}(\R^d,w))}$, then $T_{K_{\varepsilon}} \to T_K$ in the strong operator topology.
\end{proposition}
\begin{proof}
Fix $f\in L^p(\R^d;X)$ and $s \in \R^d$. Assume that $\|T_{K} f(s)\|_{\gamma(\R^d;Y)}<\infty$ and take $\varepsilon>0$. Then by domination (see \cite[Theorem 9.4.1]{HNVW17})
\begin{equation}\label{eq:Kepsilon}
  \|T_{K_\varepsilon } f(s)\|_{\gamma(\R^d;Y)} \leq \|T_{K} f(s)\|_{\gamma(\R^d;Y)}
\end{equation}
which yields $ T_{K}^\star f(s) \leq \|T_{K} f(s)\|_{\gamma(\R^d;Y)}$.

 Conversely assume that $T_{K}^\star f(s)<\infty$. Note that since $\gamma(\R^d,Y) \hookrightarrow \mc{L}(L^2(\R^d),Y)$, we have
\begin{align*}
\int_{\R^d} |\ip{ K(s,t) f(t), y^*}|^2 \dd t
& \leq \sup_{\varepsilon> 0} \int_{\R^d} |\ip{ K_{\varepsilon}(s,t) f(t), y^*}|^2 \dd t
\\ & \leq \sup_{\varepsilon>0}\|t\mapsto K_{\varepsilon}(s,t) f(t)\|_{\gamma(\R^d;Y)}^2 \|y^*\|^2<\infty.
\end{align*}
Therefore, $t\mapsto K(s,t) f(t)$ is weakly in $L^2$ and thus $T_Kf(s)$ is a bounded operator from $L^2(\R^d)$ into $Y$. Moreover, for all $\varphi\in L^2(\R^d)$ and $y^*\in Y^*$, the dominated convergence theorem yields that
\[\ip{T_Kf(s)\varphi,y^*} = \lim_{\varepsilon\to 0}\ip{T_{K_\varepsilon}f(s)\varphi,y^*}.\]
Now the $\gamma$-Fatou lemma (See \cite[Proposition 9.4.6]{HNVW17}) yields
\begin{align*}
\|T_Kf(s)\|_{\gamma(\R^d;Y)}\leq \lim_{\varepsilon\to 0}\|T_{K_\varepsilon}f(s)\|_{\gamma(\R^d;Y)} = \sup_{\varepsilon> 0}\|T_{K_\varepsilon}f(s)\|_{\gamma(\R^d;Y)} .
\end{align*}
where the equality follows again by domination. This concludes the proof of the equality
$$ T_{K}^\star f(s) = \|T_{K} f(s)\|_{\gamma(\R^d;Y)}.$$

By taking $L^p$-norms and using the density of $L^p(\R^d;X)$ in $L^p(\R^d,w;X)$ (see \cite[Exercise 7.4.1]{Gr14a}), we directly obtain
\begin{align*}
\|K\|_{\mathcal{K}(L^p(\R^d,w))}  &=  \nrm{T_K^\star}_{L^p(\R^d,w)} \leq \sup_{\varepsilon> 0} \|K_{\varepsilon}\|_{\mathcal{K}(L^p(\R^d,w))},
\end{align*}
and the converse inequality follows from \eqref{eq:Kepsilon}. The estimate for $L^{p,\infty}$ follows analogously. Finally, the strong convergence follows from \eqref{eq:Kepsilon}, the dominated convergence theorem and the $\gamma$-dominated convergence theorem (see \cite[Theorem 9.4.2]{HNVW17}).
\end{proof}

Next we prove a version of the above result for stochastic integral operators. For this let $X$ and $Y$ be Banach spaces, $p \in [2,\infty)$ and $w$ a weight on $\R_+$. If $K_{\varepsilon} \in \mc{K}_W^H(L^p(\R_+,w))$ for all $\varepsilon>0$ we define  for $f \in L^p_{\ms{F}}(\Omega \times \R_+;\gamma(H,Y))$, analogous to $T_K^\star$, the operator
\begin{equation*}
  S_K^{\star} f(s) = \sup_{\varepsilon>0}\,\nrm{S_{K_\varepsilon}f(s)}_Y, \qquad s \in \R_+.
\end{equation*}

\begin{theorem}\label{theorem:stochintoperatormaximal}
Let $X$ and $Y$ Banach spaces and assume that $Y$ has $\UMD$ and type $2$. Let $p\in [2, \infty)$ and let $w$ be a weight on $\R_+$. Let $$K\colon\R_+ \times \R_+ \to \mc{L}(X,Y)$$ be strongly measurable such that $K_\varepsilon \in \mc{K}_W^H(L^{p}(\R_+,w))$ for all $\varepsilon>0$. Then
\begin{align*}
  \|S_K^{\star}\|_{L^p_{\ms{F}}(\Omega\times \R_+,w;\gamma(H,Y)) \to L^p(\Omega\times \R_+,w)} &\eqsim_{p}  \sup_{\varepsilon> 0}\|K_{\varepsilon}\|_{\mathcal{K}_{W}^H(L^p(\R_+,w))} \\&\eqsim_{Y,p}
  \|K\|_{\mathcal{K}_{W}^H(L^p(\R_+w))}
\end{align*}
Furthermore if $K \in \mathcal{K}_{W}^H(L^p(\R_+w))$, then  $S_{K_{\varepsilon}} \to S_K$ in the strong operator topology.
\end{theorem}
\begin{proof}
It is clear from Propositions \ref{proposition:detcharacterization}, \ref{proposition:independenceH} and \ref{proposition:truncations} that the second and third expression are norm equivalent. Moreover, it is clear that $$\|S_K^{\star}\|_{L^p_{\ms{F}}(\Omega\times \R_+,w;\gamma(H,Y)) \to L^p(\Omega\times \R_+,w)}\geq \sup_{\varepsilon>0}\|K_{\varepsilon}\|_{\mathcal{K}_{W}^H(L^p(\R_+,w))}.$$ Thus it remains to prove the converse estimate. In order to show this let $f\in L^p_{\ms{F}}(\Omega\times \R_+,w;\gamma(H,Y))$ and $\varepsilon \in (0,1)$.
Since $K\in \mathcal{K}_{W}^H(L^p(\R_+,w))$, by Doob's maximal inequality we can write
\begin{align*}
\|S_K^{\star} f(s)\|_{L^p(\Omega)} & \leq  \Big(\E\,\sup_{\varepsilon>0} \Big\|\int_{\max\cbrace{s-1/\varepsilon,0}}^{\max\cbrace{s-\varepsilon,0}} K(s,t) f(t) \dd W_H(t)\Big\|_Y^p\Big)^{1/p}\\
&\hspace{2cm} +\Big(\E\,\sup_{\varepsilon>0} \Big\|\int_{s+\varepsilon}^{s+1/\varepsilon} K(s,t) f(t) \dd W_H(t)\Big\|_Y^p\Big)^{1/p}
\\ & \leq \frac{4p}{p-1} \|S_K f(s)\|_{L^p(\Omega;Y)}.
\end{align*}
 Taking $L^p(\R_+,w)$-norms the desired estimate follows.

For the strong convergence note that by the proof of Proposition \ref{proposition:detcharacterization} we have \begin{equation*}
  \|S_K f - S_{K_{\varepsilon}}f\|_{L^p(\Omega\times \R_+,w;Y)}  \eqsim_{p,Y}\|T_K f - T_{K_{\varepsilon}} f\|_{L^p(\Omega\times \R_+,w;\gamma(\R^d;H,Y))}.
\end{equation*}
Here the right-hand side for fixed $\omega \in \Omega$ is independent of $H$ by Proposition \ref{proposition:independenceH}, so the strong convergence follows by Proposition \ref{proposition:truncations} and the dominated convergence theorem.
\end{proof}

\subsection{Necessary and sufficient conditions}\label{section:necessaryandsufficient}
Before we turn to more involved results in the subsequent sections, we first analyze the boundedness of $\gamma$-integral operators in a few special cases. We start with a necessary condition for $T_K$ to be bounded if $K$ is of convolution type.
\begin{proposition}[Necessary condition for convolution type]\label{proposition:Kconvolutionnecessary}
Let $X$ and $Y$ be Banach spaces, assume that $Y$ has type $2$ and let $p\in [2, \infty)$. Let $k:\R^d \to \mc{L}(X,Y)$ be strongly measurable and set $K(s,t) := k(s-t)$. If $K\in \mathcal{K}_\gamma(L^{p,\infty}(\R^d))$, then for all $x \in X$
\[\|t\mapsto k(t)x\|_{\gamma(\R^d;Y)}\leq  C_{d} \,\nrm{K}_{\mc{K}_\gamma(L^{p,\infty}(\R^d))} \|x\|_X.\]
The same holds for $\R_+$ instead of $\R^d$, where we set $K(s,t) = 0$ if $s \leq t$.
\end{proposition}

\begin{proof}
Let $r>0$, $x\in X$ and set $f = \ind_{B(0,2r)}\otimes x$. Then for all $s\in B(0,r)$,
\begin{align*}
L_{r} :=\|t\mapsto k(t)x\|_{\gamma(B(0,r);Y)} &= \|t\mapsto k(s-t)x\|_{\gamma(B(s,r);Y)}
\\ & = \|T_Kf(s)\|_{\gamma(B(s,r);Y)} \leq \|T_Kf(s)\|_{\gamma(\R^d;Y)}.
\end{align*}
Therefore, for any $0<\lambda<L_r$ we find that
\begin{align*}
\lambda  &\leq \lambda \,\abs{B(0,r)}^{-1/p} \, \absb{\{s\in B(0,r):\|Tf(s)\|_{\gamma(\R^d;Y)}>\lambda\}}^{1/p}
\\ & \leq \abs{B(0,r)}^{-1/p} \,\nrm{K}_{\mc{K}_\gamma(L^{p,\infty}(\R^d))} \|f\|_{L^p(\R^d;X)}\\&= C_d \, \,\nrm{K}_{\mc{K}_\gamma(L^{p,\infty}(\R^d))}  \|x\|_X.
\end{align*}
Taking $\lambda = \frac{1}{2} L_{r}$, we find that $L_{r} \leq C_{d}\,\nrm{K}_{\mc{K}_\gamma(L^{p,\infty}(\R^d))}  \|x\|$.
Now the proposition follows by letting $r\to \infty$ and applying the $\gamma$-Fatou lemma (see \cite[Proposition 9.4.6]{HNVW17}). The proof for $\R_+$ is analogous, taking $s \in (r,2r)$ instead.
\end{proof}
\begin{remark}
If we replace $\R^d$ by $(0,T)$ with $T\in (0,\infty)$  in Proposition \ref{proposition:Kconvolutionnecessary}, we can deduce that
      \[\|t\mapsto k(t)x\|_{\gamma((0,\frac12 T);Y)}\leq  C_{d} \,\nrm{K}_{\mc{K}_\gamma(L^{p,\infty}(0,T))} \|x\|_X.\]
      For specific kernels one can  stretch this estimate to the whole interval $(0,T)$ with a constant dependent on $T$, see \cite[Lemma 4.2]{AV19}.
\end{remark}

Next we provide some simple sufficient conditions on $K$ for $T_K$ to be bounded using Fubini's theorem and Young's inequality:

\begin{proposition}[Simple sufficient conditions]\label{proposition:simplesufficient}
Let $X$ and $Y$ be Banach spaces, assume that $Y$ has type $2$ and suppose that $K\colon\R^d \times \R^d \to \mc{L}(X,Y)$ is strongly measurable. Then the following hold:
\begin{enumerate}[(i)]
  \item \label{it:sufficientL2} If there is an $A_0>0$ such that
  \[\|s\mapsto K(s,t)x\|_{L^2(\R^d;Y)}\leq  A_0\, \|x\|_X, \qquad t\in \R^d,\quad x\in X,\]
then $K\in \mathcal{K}_\gamma(L^2(\R^d))$ with $\nrm{K}_{\mathcal{K}_\gamma(L^2(\R^d))} \leq \tau_{2,Y}\, A_0.$
  \item \label{it:sufficientconvolution} If $\nrm{K(s,t)} \leq k(s-t)$ for some $k \in L^2(\R^d)$, then $K\in \mathcal{K}_\gamma(L^p(\R^d))$ for all $p \in [2,\infty)$ with $\nrm{K}_{\mathcal{K}_\gamma(L^p(\R^d))} \leq \tau_{2,Y} \|k\|_{L^2(\R^d)}.$
\end{enumerate}
The same holds for $(0,T)$ with $T \in (0,\infty]$ instead of $\R^d$, where $K(s,t) = 0$ if $s \leq t$.
\end{proposition}

\begin{proof}
For \ref{it:sufficientL2} we have by Lemma \ref{lemma:gammaL2} that
\begin{align*}
\|T_K f(s)\|_{\gamma(\R^d;Y)}\leq \tau_{2,Y} \Big(\int_{\R^d} \|K(s,t) f(t)\|_Y^2 \dd t\Big)^{1/2}, \qquad s\in \R^d.
\end{align*}
Taking $L^2$-norms on both sides and applying Fubini's theorem we obtain
\begin{align*}
\|T_K f\|_{L^2(\R^d;\gamma(\R^d;Y))}& \leq \tau_{2,Y} \has{\int_{\R^d} \|s\mapsto K(s,t) f(t)\|_{L^2(\R^d;Y)}^2 \dd t}^{1/2} \\& \leq\tau_{2,Y} \nrm{f}_{L^2(\R^d;X)}.
\end{align*}
For \ref{it:sufficientconvolution} we have by Lemma \ref{lemma:gammaL2}
\begin{align*}
\|T_K f(s)\|_{\gamma(\R^d;Y)}\leq \tau_{2,Y} \Big(\int_{\R^d} |k(s-t)|^2 \|f(t)\|_Y^2 \dd t\Big)^{1/2}, \qquad s\in \R^d.
\end{align*}
Taking $L^p$-norms on both sides and applying Young's inequality we obtain
\begin{align*}
\|T_K f\|_{L^p(\R^d;\gamma(\R^d;Y))}& \leq \tau_{2,Y}\,\|k\|_{L^2(\R^d)} \nrm{f}_{L^p(\R^d;X)}.
\end{align*}
The $(0,T)$ case follows similarly, where we extend $K$ and $f$ by $0$ outside $(0,T)$ to apply Young's inequality for \ref{it:sufficientconvolution}.
\end{proof}

If $Y$ is a Hilbert space and $K$ is of convolution type, we can actually characterize the boundedness of $T_K$, since in this case $\gamma(\R^d;Y) = L^2(\R^d;Y)$. In Corollaries \ref{corollary:Hilbertspacesch} and \ref{corollary:Hilbertspaceschweight} the following result will be improved under regularity conditions on $K$.
\begin{corollary}\label{corollary:firstextrapol}Let $X$ be a Banach space and $Y$ be a Hilbert space. Let $k:\R^d \to \mc{L}(X,Y)$ be strongly measurable and set $K(s,t) := k(s-t)$. Then the following hold:
\begin{enumerate}[(i)]
\item \label{it:cornoreg1} $K \in \mc{K}_\gamma(L^2(\R^d))$ if and only if $\|t\mapsto k(t)x\|_{L^2(\R^d;Y)}\lesssim \|x\|_X.$
\item \label{it:cornoreg2} If $K\in \mathcal{K}_\gamma(L^{p,\infty}(\R^d))$ for some $p \in [2,\infty)$, then $K\in \mathcal{K}_\gamma(L^q(\R^d))$ for all $q\in [2, p)$.
\end{enumerate}
The same hold for $(0,T)$ with $T \in (0,\infty]$ instead of $\R^d$, where we set $K(s,t) = 0$ if $s \leq t$.
\end{corollary}

\begin{proof}
By \cite[Proposition 9.2.9]{HNVW17} one has for all $t \in \R^d$ that
\[\|s\mapsto K(s,t)x\|_{L^2(\R^d;Y)} =  \|s\mapsto k(s) x\|_{L^2(\R^d;Y)}= \|s\mapsto k(s) x\|_{\gamma(\R^d;Y)},\]
from which \ref{it:cornoreg1} follows using by Proposition \ref{proposition:Kconvolutionnecessary} and \ref{proposition:simplesufficient}\ref{it:sufficientL2}. Part \ref{it:cornoreg2} follows by combining Proposition \ref{proposition:Kconvolutionnecessary}, part \ref{it:cornoreg1} and the Marcinkiewicz interpolation theorem (see \cite[Theorem 2.23]{HNVW16}).
\end{proof}

\subsection{Scalar kernels}\label{section:scalarcase}
If we allow $X$ to be any Banach space with type $2$, but restrict $K$ to be scalar-valued, we can easily characterize the boundedness of $T_K$ if $K$ is of convolution type. This explains why we study the more interesting operator-valued case.
\begin{proposition}\label{proposition:scalarkernels}
Let $X$ be a Banach space with type $2$, let $p\in [2, \infty)$, let $k:\R^d \to \K$ be measurable and set $K(s,t) := k(s-t)$. Then $T_K$ is bounded from $L^p(\R^d;X)$ to  $L^p(\R^d;\gamma(\R^d;X))$ if and only if $k \in L^2(\R^d)$. Moreover, in this case
$\nrm{K}_{\mathcal{K}_\gamma(L^p(\R^d))} \leq \tau_{2,X} \|k\|_{L^2(\R^d)}$.
\end{proposition}
\begin{proof}
Since $k$ is scalar-valued, we have for $x \in X$
  \begin{equation*}
    \nrm{s\mapsto k(s) x}_{\gamma(\R^d;X)} = \nrm{x}_X\nrm{k}_{L^2(\R^d)}.
  \end{equation*}
  Therefore the result follows from Proposition \ref{proposition:Kconvolutionnecessary} and Proposition \ref{proposition:simplesufficient}\ref{it:sufficientconvolution}.
\end{proof}

 In the scalar case, i.e. $X = Y = \K$, the $L^p$-boundedness of $T_K$ can also be well-understood from existing theory for non-convolution kernels. Indeed, in this case $K\in \mathcal{K}_\gamma(L^p(\R^d))$ is equivalent to
\begin{equation}\label{eq:Kscalarchar}
\int_{\R^d} \Big(\int_{\R^d} |K(s,t)|^2 g(t) \dd t\Big)^{p/2} \dd s \leq C^p \|g\|^{p}_{L^{p/2}(\R^d)},
\end{equation}
where we have set $g(t) = |f(t)|^2$. The validity of the above estimate is completely characterized by the optimality of Schur's lemma (see \cite[Appendix A.2]{Gr14b}) applied to the positive kernel $|K(s,t)|^2$. Moreover, in this case $T_K$ is also bounded in the vector-valued setting where $X = Y$ has type $2$, since by Lemma \ref{lemma:gammaL2}
\[\|T_K f\|_{L^p(\R^d;\gamma(\R^d;X))}^p \leq \tau_{2,X}^p \int_{\R^d}\Big(\int_{\R^d} |K(s,t)|^2 g(t) \dd t\Big)^{p/2} \dd s,\]
where $g(t) = \|f(t)\|^2_X$.
Conversely, by considering a one-dimensional subspace of $X$, one obtains that \eqref{eq:Kscalarchar} is also necessary.

\begin{example}~\label{example:scalarcase}
\begin{enumerate}[(i)]
\item Let $d=1$ and $K(s,t) = \frac{1}{(s+t)^{1/2}}\ind_{s,t>0}$. Then by \cite[Theorem 5.10.1]{Ga07b} we know that $K\in \mathcal{K}_\gamma(L^p(\R))$ if and only if $p\in (2, \infty)$. More generally for $1\leq j \leq d$ set
    \begin{equation*}
  K_j\ha{s,t} := \frac{(s_j+t_j)^{1/2}}{\absb{s+t}^{(d+1)/2}}\ind_{s_j,t_j>0}, \qquad s,t \in \R^d.
\end{equation*}
Then we know by \cite[Theorem 1]{Os17} that $K_j\in \mathcal{K}_\gamma(L^p(\R^d))$ if and only if $p\in (2, \infty)$.
\item \label{example:hilbertanalog} If $K(s,t) = \frac{1}{|s-t|^{1/2}}$, then for all $p\in [2, \infty)$, we have $K\notin \mathcal{K}_\gamma(L^p(\R))$, which is immediate from Proposition \ref{proposition:Kconvolutionnecessary}.
\end{enumerate}
\end{example}

Example \ref{example:scalarcase}\ref{example:hilbertanalog} can be seen as the analog of the Hilbert transform. It is not bounded for any $p \in [2,\infty)$ due to the lack of cancellation in the stochastic, scalar-valued setting. This further exemplifies the difference between the deterministic and the stochastic theories.

\begin{remark}
  The scalar case also shows why we only consider $p\in [2, \infty)$. Boundedness for $p<2$ holds if and only if $K \equiv 0$ (see \cite{Ka78}). This also holds for the operator-valued case since $L^p$-boundedness with $p<2$ would imply that
$\ip{ K(t,s)x, y^*} = 0$ a.e.\ for all $x\in X$ and $y^*\in Y^*$. By strong measurability of $(t,s)\mapsto K(t,s)x$ this implies that for all $x\in X$, we have $K(t,s)x \equiv 0$. Thus by the density of $L^p(\R^d)\otimes X$ in $L^p(\R^d;X)$, we find that $K(s,t) f(t) = 0$.
\end{remark}

\section{Singular \texorpdfstring{$\gamma$}{y}-kernels of H\"ormander and Dini type}\label{section:kernels}
Motivated by the connection between stochastic integral operators and $\gamma$-integral operators proven in Proposition \ref{proposition:detcharacterization} and Proposition \ref{proposition:independenceH}, we will now start the systematic study of the $\mc{K}_\gamma$-classes for more involved kernels than those treated in Subsection \ref{section:necessaryandsufficient}. In particular, we will develop a $\gamma$-version of the Calder\'on--Zygmund theory for (deterministic) singular integral operators. This will first be done on $\R^d$ and afterwards in Section \ref{section:homogeneoustype} we will point out how our arguments carry over to the more general setting of spaces of homogeneous type, which for example includes the $(0,T)$-case for $T \in (0,\infty]$.

Let us first define our assumptions on the $\gamma$-kernels $K$:
\begin{definition}\label{definition:kernels}
  Let $X,Y$ be a Banach spaces and let $K\colon\R^d \times \R^d \to \mc{L}(X,Y)$ be strongly measurable.
  \begin{itemize}
    \item We say that $K$ is a $2$-H\"ormander kernel if for every ball $B\subseteq \R^d$ we have
  \begin{align}
   \has{\int_{\R^d \setminus B} \nrm{K(s,t)-K(s',t)}^2\dd t}^{1/2} \leq C && s,s'\in \frac{1}{2}B \label{eq:horm1}\\
  \label{eq:horm2}
    \has{\int_{\R^d \setminus B} \nrm{K(s,t)-K(s,t')}^2\dd s}^{1/2} \leq C && t,t'\in \frac{1}{2}B
  \end{align}
  for some constant $C>0$ independent of $B$. The least admissible $C$ will be denoted by $\nrm{K}_{\Hormander{2}}$.
  \item We say that $K$ is an $(\omega,2)$-Dini kernel if
  \begin{align}
   \nrm{K(s,t)-K(s',t)} &\leq \omega\has{\frac{\abs{s-s'}}{\abs{s-t}}}\frac{1}{\abs{s-t}^{d/2}} \quad \label{eq:dini2}&&\abs{s-s'}\leq \frac{1}{2}\abs{s-t},\\
    \nrm{K(s,t)-K(s,t')} &\leq \omega\has{\frac{\abs{t-t'}}{\abs{s-t}}}\frac{1}{\abs{s-t}^{d/2}} \quad \label{eq:dini3} &&\abs{t-t'}\leq \frac{1}{2}\abs{s-t},
  \end{align}
  where $\omega:[0,1]\to [0,\infty)$ is increasing, subadditive, $\omega(0)=0$ and
  \begin{equation*}
    \nrm{K}_{\Dini{\omega}{2}}:=\has{\int_0^1\omega(r)^2\frac{\dd r}{r}}^{1/2} <\infty.
  \end{equation*}
  \item We say that $K$ is an $(\epsilon,2)$-standard kernel if $K$ is an $(\omega,2)$-Dini kernel with $\omega(r) = C r^\epsilon$ for some $\epsilon \in (0,1]$ and we set  $$\nrm{K}_{\Standard{\epsilon}{2}}:= \nrm{K}_{\Dini{\omega}{2}}.$$
  \end{itemize}
  \end{definition}

  If $K$ is of convolution type, i.e. if $K(s,t) = k(s-t)$ for some $k\colon\R^d \to \mc{L}(X,Y)$, the H\"ormander and Dini conditions in Definition \ref{definition:kernels} can be reformulated using a change of variables. Indeed,
\eqref{eq:horm1} and \eqref{eq:horm2} both simplify to
\begin{equation}\label{eq:hormconvolution}
  \has{\int_{\frac{1}{2}\abs{s} \geq  \abs{t}} \nrm{k(s-t)-k(s)}^2\dd s}^{1/2} \leq C \qquad  t \in \R^d
\end{equation}
and \eqref{eq:dini2} and \eqref{eq:dini3} both simplify to
\begin{equation}\label{eq:diniconvolution}
  \nrm{k(s-t)-k(s)} \leq \omega\has{\frac{\abs{t}}{\abs{s}}}\frac{1}{\abs{s}^{d/2}} \qquad   \abs{t}\leq \frac{1}{2}\abs{s}.
\end{equation}
An $L^r$-variant of \eqref{eq:hormconvolution} already appeared in \cite{Ho60}.

 By definition an $(\epsilon,2)$-standard kernel is also an $(\omega,2)$-Dini kernel. As in the deterministic setting an  $(\omega,2)$-Dini kernel is also a $2$-H\"ormander kernel. The proof is a straightforward adaptation of the proof in the deterministic setting. For the convenience of the reader we include the details.
\begin{lemma}\label{lemma:hormanderdini}
   Let $X,Y$ be Banach spaces and suppose that $K\colon\R^d \times \R^d \to \mc{L}(X,Y)$ is an $(\omega,2)$-Dini kernel. Then $K$ is a $2$-H\"ormander kernel with
   \begin{equation*}
     \nrm{K}_{\Hormander{2}} \leq C_{d}\, \nrm{K}_{\Dini{\omega}{2}}.
   \end{equation*}
\end{lemma}

\begin{proof} We will only show \eqref{eq:horm1}, as \eqref{eq:horm2}
 follows analogously.
  Let $B=B(s,r)\subseteq \R^d$ be a ball and take $s^1,s^2 \in \frac{1}{2}B$. Then  $\abs{s-s^j} \leq \frac{1}{2}r \leq \frac{1}{2}\abs{s-t}$ for any $t \in \R^d \setminus B$, so
  \begin{align*}
    \has{\int_{\R^d \setminus B} &\nrm{K(s^1,t)-K(s^2,t)}^2\dd t}^{1/2}\\
     &\leq \sum_{j=1}^2 \has{\int_{\R^d \setminus B} \nrm{K(s,t)-K(s^j,t)}^2\dd t}^{1/2}\\
    &\leq \sum_{j=1}^2 \has{\int_{\abs{s-t}>r} \omega\has{\frac{\abs{s-s^j}}{\abs{s-t}}}^2\frac{1}{\abs{s-t}^d}\dd t }^{1/2}\\
    &\leq C_d \, \has{ \sum_{k=0}^\infty \omega(2^{-k-1})^2 \int_{2^kr<\abs{s-t}\leq r2^{k+1}} \frac{\dd t}{\abs{s-t}^d}}^{1/2}\\
    &\leq C_d \, \has{ \sum_{k=0}^\infty \omega(2^{-k-1})^2 \int_{B(0,r2^{k+1})} \frac{1}{(2^{k}r)^d}\dd t}^{1/2}\\
&\leq C_{d}\,\nrm{K}_{\Dini{\omega}{2}}.\qedhere
  \end{align*}
\end{proof}

If $K$ is differentiable, we can check the standard kernel conditions in terms of the derivatives of the kernel.
\begin{lemma}\label{lemma:standardkernelderivatives}
  Let $X$ and $Y$ be a Banach spaces and let
  $$K \in C^1\hab{\R^d\times \R^d\setminus \cbrace{(s,s):s\in \R^d};\mc{L}(X,Y)}.$$ Suppose that there is a constant $A_0>0$ such that
  \begin{align*}
  \nrmb{\partial^{\alpha}_s K(s,t)}&\leq A_0\cdot {\abs{s-t}^{-d/2 -1}} \qquad \abs{\alpha}=1,\,  s \neq t,\\
  \nrmb{\partial^{\alpha}_t K(s,t)}&\leq A_0\cdot  {\abs{s-t}^{-d/2 -1}} \qquad \abs{\alpha}= 1,\,  s \neq t.
  \end{align*}
Then $K$ is a $(1,2)$-standard kernel with $\nrm{K}_{\Standard{1}{2}} \leq C_d \, A_0 $.
\end{lemma}

\begin{proof}
We will prove \eqref{eq:dini2}, the proof of \eqref{eq:dini3} is analogous. Take  $s,s',t \in \R^d$ such that $0<\abs{s-s'}\leq \frac{1}{2}\abs{s-t}$. Then
we have for all $\lambda \in [0,1]$
\begin{align*}
  \abs{s-t-\lambda (s-s')} &\geq \abs{s-t} -\abs{\lambda(s-s')}\geq \abs{s-t} -\abs{s-s'} \geq \frac{1}{2}\abs{s-t}.
\end{align*}
Therefore using the fundamental theorem of calculus we obtain
\begin{align*}
  \nrmb{K(s,t)-K(s',t)}&= \nrms{\int_0^1 \frac{\ddn}{\ddn \lambda}  K(s-\lambda(s-s'),t)\dd \lambda} \\
  &\leq \sum_{j=1}^d\int_0^1 \nrmb{\partial_j K(s-\lambda(s-s'),t)\cdot(s_j-s'_j)} \dd \lambda\\
  &\leq A_0\, \sum_{j=1}^d \int_0^1 \frac{\abs{s-s'}}{\abs{s-t - \lambda(s-s')}^{d/2+1}} \dd \lambda\\
  &\leq C_{d}\,A_0 \frac{\abs{s-s'}}{\abs{s-t}} \frac{1}{\abs{s-t}^{d/2}}
\end{align*}
proving the lemma.
\end{proof}

If $K$ is of convolution type, a sufficient condition for \eqref{eq:horm1}, \eqref{eq:horm2}, \eqref{eq:dini2} and \eqref{eq:dini3} can also be formulated in terms of smoothness of the Fourier transform of $k$. For the usual H\"ormander and Dini kernels this is classical. The $L^2$-variants (or even the $L^r$-variants) in Definition \ref{definition:kernels} can be treated by similar methods (see e.g. \cite[Section 5.1]{RV17}).

We end this section with a sufficient condition for the standard kernel conditions in terms of fractional smoothness on $\R$.
\begin{lemma}\label{lemma:standardkernelfractional}
Let $\Phi\colon\R_+\to \mc{L}(X,Y)$ be strongly measurable and suppose there exists a constant $A_0>0$ and an $\epsilon\in (0,\tfrac12)$ such that
\begin{align*}
\|\Phi(s)\| \leq A_0 \,s^{-\frac12-\epsilon}, \qquad s>0.
\end{align*}
Let $k\colon \R \to \mc{L}(X,Y)$ be defined by
\begin{equation*}
k(s) =
\begin{cases}
    \frac1{\Gamma(\epsilon)} \int_0^s (s-r)^{\epsilon-1} \Phi(r) \dd r\qquad &s> 0\\ 0\qquad &s \leq 0
\end{cases}
\end{equation*}
Then $K(s,t):= k(s-t)$ is an $(\epsilon,2)$-standard kernel.
\end{lemma}
\begin{proof}
Let $s>0$ and assume $t \in [s, \frac{3}{2}s]$. By \eqref{eq:diniconvolution} it suffices to show
\[\|k(s) - k(t)\|\leq A_0\, C_{\epsilon} \frac{(t-s)^{\epsilon}}{s^\epsilon} \frac{1}{s^{1/2}}.\]
To show this note that
\begin{align*}
\Gamma(\epsilon) \,&\|k(s) - k(t)\| \\ & \leq  \int_s^t (t-r)^{\epsilon-1} \|\Phi(r)\| \dd r + \int_0^s \big((s-r)^{\epsilon-1} - (t-r)^{\epsilon-1}\big) \|\Phi(r)\| \dd r
\\ & \leq A_0\, \underbrace{\int_s^t (t-r)^{\epsilon-1} r^{-\epsilon-\frac12} \dd r}_{ \text{\framebox[15pt]{A}}} + A_0 \, \underbrace{\int_0^s \big((s-r)^{\epsilon-1} - (t-r)^{\epsilon-1}\big) r^{-\epsilon-\frac12} \dd r}_{ \text{\framebox[15pt]{B}}}.
\end{align*}
For \framebox[15pt]{A} note that
\begin{align*}
\text{\framebox[15pt]{A}} \leq s^{-\epsilon-\frac12}  \int_s^t (t-r)^{\epsilon-1} \dd r = \epsilon^{-1} \frac{(t-s)^{\epsilon}}{s^\epsilon} \frac{1}{s^{1/2}}.
\end{align*}
For \framebox[15pt]{B} we write $\text{\framebox[15pt]{B}}=\text{\framebox[20pt]{B1}}+\text{\framebox[20pt]{B2}}$ where we have split the integral into parts over $(0,s/2)$ and $(s/2,s)$. For \framebox[20pt]{B1} we can write
\begin{align*}
\framebox[20pt]{B1} &  = \frac{1}{1-\epsilon} \int_0^{s/2} \int_{s-r}^{t-r} x^{\epsilon-2} \dd x \, r^{-\epsilon-\frac12}  \dd r
\\ & \leq \frac{1}{1-\epsilon} \int_0^{s/2} (t-s) (s-r)^{\epsilon-2} r^{-\epsilon-\frac12}  \dd r
\\ & \leq \frac{(t-s) (s/2)^{\epsilon-2} }{1-\epsilon} \int_0^{s/2} r^{-\epsilon-\frac12}  \dd r
\\ & \leq \frac{(t-s) (s/2)^{\epsilon-2}}{(1-\epsilon)(\frac12-\epsilon)} (s/2)^{\frac12-\epsilon} = \frac{2\sqrt{2}}{(1-\epsilon)(\frac12-\epsilon)}  \frac{t-s}{s} \frac{1}{s^{1/2}}
\end{align*}
where we used $\epsilon<\frac12$.
Finally,  using $t\geq s$, we obtain
\begin{align*}
\framebox[20pt]{B2} & \leq (s/2)^{-\epsilon-\frac12} \int_{s/2}^s \big((s-r)^{\epsilon-1} - (t-r)^{\epsilon-1}\big)  \dd r
\\ & = \epsilon^{-1} (s/2)^{-\epsilon-\frac12} \big((t-s)^{\epsilon} + (s/2)^{\epsilon} - (t- s/2)^{\epsilon}\big)
\\ & \leq \epsilon^{-1} 2^{\epsilon+\frac12} \frac{(t-s)^{\epsilon}}{s^\epsilon} \frac{1}{s^{1/2} },
\end{align*}
which implies the required estimate.
\end{proof}

\section{Extrapolation for \texorpdfstring{$\gamma$}{y}-integral operators}\label{section:extrapolation}
In this section we will prove the first results regarding the extrapolation of the $L^p$-boundedness of an $\gamma$-integral operator $T_K$ to the $L^q$-boundedness of $T_K$ for all $q \in (2,\infty)$ under a $2$-H\"ormander assumption on $K$. We will also obtain a weak $L^2$-endpoint and a $\BMO$-endpoint result.

\subsection{Extrapolation for \texorpdfstring{$2<q<p$}{2<q<p}}
Let us start our analysis with an extrapolation result downwards. We will show that if $K \in \mc{K}_\gamma(L^{p,\infty}(\R^d))$ satisfies the $2$-H\"ormander condition, then also $K \in \mc{K}_\gamma(L^{q}(\R^d))$ for all $q \in (2,p)$ and $K \in \mc{K}_\gamma(L^{2,\infty}(\R^d))$. For this we will adapt the Calder\'on-Zygmund decomposition technique for singular integral operators to the $\gamma$-case. Our main tool will be the following $L^2$-Calder\'on--Zygmund decomposition.

\begin{proposition}[$L^2$-Calder\'on--Zygmund decomposition]
  Let $X$ be a Banach space. For every $f \in L^2(\R^d;X)$ and $\lambda >0$ there exists a decomposition $f = g +b$ with
  \begin{align*}
    &\nrm{g}_{L^\infty(\R^d;X)} \leq 2^{d/2}\lambda ,  &&\nrm{g}_{L^2(\R^d;X)} \leq \nrm{f}_{L^2(\R^d;X)}
  \intertext{and $b = \sum_j b_j$ with}
  &\supp b_j \subseteq Q_j \qquad &&\sum_j\abs{Q_j}\leq \lambda^{-2}\nrm{f}_{L^2(\R^d;X)}^2\\
  &\nrm{b_j}_{L^2(\R^d;X)}^2 \leq 2^{d+2} \lambda^2\abs{Q_j},\qquad &&\sum_k\nrm{b_j}_{L^2(\R^d;X)}^2 \leq 2^{d+2} \nrm{f}_{L^2(\R^d;X)}^2
  \end{align*}
  for disjoint cubes $Q_j \subseteq \R^d$.
\end{proposition}
For the proof in the case $X=\C$ we refer to \cite[Exercise 5.3.8]{Gr14a}, where the more general $L^q$-Calder\'on--Zygmund decomposition for any $q \in [1,\infty)$ is shown. The proof carries over verbatim to the vector-valued setting, replacing absolute values by norms.

Note that in the deterministic setting the functions $b_j$ in a Calder\'on--Zygmund decomposition are usually also taken such that $\int_{Q_k}b_j = 0$, but we will not be able to use this property for $\gamma$-integral operators. Instead we use the $L^2$-Calder\'on--Zygmund decomposition in a way that is inspired by \cite{DM99}, which builds upon ideas developed in \cite{ DR96, Fe70, He90}.

\begin{theorem}[Extrapolation downwards]\label{theorem:extrapolationdown}
  Let $X$ and $Y$ be Banach spaces with type $2$, let $p \in [2,\infty)$ and suppose that $K \in \mc{K}_\gamma(L^{p,\infty}(\R^d))$ satisfies the $2$-H\"ormander condition. Then
 \begin{enumerate}[(i)]
 \item \label{it:extrapolationdowni} $K \in \mc{K}_\gamma(L^{q}(\R^d))$ for all $q \in (2,p)$ with
 \begin{align*}
  \nrm{K}_{\mc{K}_\gamma(L^{q}(\R^d))} &\leq C_{p,q,d}\,\has{\tau_{2,X} \tau_{2,Y} \nrm{K}_{\mc{K}_\gamma(L^{p,\infty}(\R^d))} + \tau_{2,Y} \nrm{K}_{\Hormander{2}}}.
\end{align*}\item \label{it:extrapolationdownii}  $K \in \mc{K}_\gamma(L^{2,\infty}(\R^d))$ with
\begin{align*}
  \nrm{K}_{\mc{K}_\gamma(L^{2,\infty}(\R^d))} &\leq C_{p,d}\,\has{\tau_{2,X} \tau_{2,Y} \nrm{K}_{\mc{K}_\gamma(L^{p,\infty}(\R^d))} + \tau_{2,Y} \nrm{K}_{\Hormander{2}}}.
  \end{align*}
  \end{enumerate}
\end{theorem}

\begin{proof}
 It suffices to show \ref{it:extrapolationdownii}, as  \ref{it:extrapolationdowni} then follows directly  from the Marcin\-kie\-wicz interpolation theorem, see e.g. \cite[Theorem 2.2.3]{HNVW16}.

 Let $f \in L^{2}(\R^d;X)\cap L^{p}(\R^d;X)$ be compactly supported, $\lambda>0$ and set $A_0:= \nrm{K}_{\mc{K}_\gamma(L^{p,\infty}(\R^d))}$. Let $f = g+b$ be the $L^2$-Calder\'on--Zygmund decomposition of $f$ at level $\mu\lambda$ for some $\mu>0$ to be chosen later. Then we have
  \begin{equation}\label{eq:gunderf}
    \nrm{g}_{L^{p}(\R^d;X)}^{p} \leq \nrm{g}_{L^\infty(\R^d;X)}^{p-2}\nrm{g}_{L^2(\R^d;X)}^{2} \leq \has{{2^{d/2}\mu\lambda}}^{p-2}\nrm{f}_{L^2(\R^d;X)}^{2},
  \end{equation}
  so in particular $g \in L^{p}(\R^d;X)$. It follows that $b =f-g \in L^{p}(\R^d;X)$, and thus
  \begin{equation*}
    T_Kf = T_Kg + T_Kb
  \end{equation*}
  is well-defined.

\bigskip
  To estimate the $L^{2,\infty}(\R^d;\gamma(\R^d;Y))$-norm of $T_Kf$ we need to analyze the size of $\cbrace{\nrm{T_Kf}_{\gamma(\R^d;Y)}>\lambda}.$ We split as follows:
  \begin{equation}\label{eq:fgbsplit}
    \absb{\cbraceb{\nrm{T_Kf}_{\gamma(\R^d;Y)}>\lambda}}
    \leq \absb{\cbraceb{\nrm{T_Kg}_{\gamma(\R^d;Y)}>\tfrac{\lambda}{2}}} + \absb{\cbraceb{\nrm{T_Kb}_{\gamma(\R^d;Y)}>\tfrac{\lambda}{2}}}.
  \end{equation}
  For the term with the ``good'' part $g$  we have by our assumption on $T_K$ and \eqref{eq:gunderf} that
\begin{align*}
  \absb{\cbrace{\nrm{T_Kg}_{\gamma(\R^d;Y)}>\lambda/2}} &\leq \frac{A_0^{p}}{(\lambda/2)^{p}}\nrm{g}^{p}_{L^{p}(\R^d;X)}  \\
    &\leq C_{p,d}\, {A_0^{p}\,\mu^{p-2}} \,\frac{\nrm{f}_{L^2(\R^d;X)}^2}{\lambda^2}
\end{align*}

For the term with the ``bad'' part $b$, let $Q_j$ be the cube corresponding to $b_j$ with center $s_j$ and diameter $2r_j$. Set $B_j:=B(s_j,r_j)$, then $Q_j \subseteq B_j$ and $\abs{B_j} \leq C_d \abs{Q_j}$. Set $B_j' = 4B_j$ and $O := \bigcup_j B_j'$.

 for our estimates we will define some auxiliary operators. For $r>0$ define
\begin{equation*}
  \psi_r(s,t) := \frac{1}{\abs{B(t,r)}^{1/2}} \ind_{B(t,r)}(s), \qquad s,t \in \R^d.
\end{equation*}
Note that $\nrm{\psi_r(\cdot,t)}_{L^2(\R^d)}=1$ for all $t \in \R^d$ and
\begin{equation}\label{eq:suppsi}
  \sup_{t \in B_j} \psi_{r_j}(s,t)= \frac{1}{\abs{B_j}^{1/2}} \ind_{2B_j}(s), \qquad s \in \R^d
\end{equation}
Let $S_{j}\colon L^2(\R^d;X) \to L^2(\R^d;\gamma(\R^d;X))$ be the $\gamma$-integral operator given by
\begin{equation*}
  S_{j}h(s) :=  \psi_{r_j}(s,\cdot)h(\cdot) ,\qquad s\in \R^d, \quad h \in L^2(\R^d;X)
\end{equation*}
which is bounded by Proposition \ref{proposition:simplesufficient}\ref{it:sufficientL2}.
We claim that $\sum_j S_jb_j$ converges in $L^p(\R^d;\gamma(\R^d;X))$. To prove this we first estimate for fixed $j$ and a.e. $s \in \R^d$
\begin{align*}
 \nrm{S_{j}b_j(s)}_{\gamma(\R^d;X)}^2
 &\leq \tau_{2,X}^2 \int_{B_j} \psi_{r_j}(s,t)^2 \nrm{b_j(t)}_X^2\dd t\\
 &\leq  \tau_{2,X}^2 \, \nrm{b_j}_{L^2(\R^d;X)}^2\sup_{t \in B_j} \psi_{r_j}(s,t)^2
 \\&\leq C_{d}  \, \tau_{2,X}^2 \, (\mu\lambda)^2 \ind_{2B_j}(s)
\end{align*}
using Lemma \ref{lemma:gammaL2}, the norm estimate of $b_j$ in terms of $\abs{Q_j}$ and \eqref{eq:suppsi}.
Thus for $\varphi \in L^{p'}(\R^d)$ positive we have
\begin{align*}
  \ipb{\nrm{S_{j}b_j}_{\gamma(\R^d;X)},\varphi}  &\leq C_{d}  \, \tau_{2,X} \, \mu\lambda\int_{Q_j} \avint_{2 B_j} \varphi(s)\dd s \dd t \\&\leq C_{d}  \, \tau_{2,X} \, \mu\lambda \,\ip{\ind_{Q_j},M\varphi}.
\end{align*}
So summing over $j$ we get, using the boundedness of the maximal operator $M$ as in Lemma \ref{lemma:maximaloperator} and the fact that the $Q_j$'s are disjoint, that
\begin{equation*}
\begin{aligned}
    \nrmb{\sum_{j} S_{j}b_j}_{L^p(\R^d;\gamma(\R^d;X))} &\leq \nrmb{\sum_j \nrm{S_{j}b_j}_{\gamma(\R^d;X)}}_{L^{p}(\R^d)}\\
  &= \sup_{\nrm{\varphi}_{L^{p'}(\R^d)}\leq 1} \ips{\sum_j \nrm{S_{j}b_j}_{\gamma(\R^d;X)},\abs{\varphi}}\\
  &\leq C_{d}  \, \tau_{2,X} \, \mu\lambda \,\nrmb{\sum_j \ind_{Q_j}}_{L^{p}(\R^d)} \sup_{\nrm{\varphi}_{L^{p'}(\R^d)}\leq 1}  \nrm{M\varphi}_{L^{p'}(\R^d)}\\
  &\leq C_{p,d}  \, \tau_{2,X} \, \mu\lambda \,\has{\sum_j\abs{Q_j}}^{1/p}.
\end{aligned}
\end{equation*}
Since $\sum_j\abs{Q_j}\leq (\mu\lambda)^{-2}\nrm{f}_{L^2(\R^d;X)}^2$ it follows that $\sum_j S_jb_j$ converges in $L^p(\R^d;\gamma(\R^d;X))$ as claimed and in particular we have
\begin{equation}\label{eq:Sestimate}
  \nrmb{\sum_{j} S_{j}b_j}_{L^p(\R^d;\gamma(\R^d;X))} \leq C_{p,d}  \, \tau_{2,X} \, (\mu\lambda)^{1-2/p} \nrm{f}_{L^2(\R^d;X)}^{2/p}.
\end{equation}

Next set
\begin{equation*}
\psi(s,t): = \sum_j \psi_{r_j}(s,t) \ind_{Q_j}(t)
\end{equation*}
and define for $h \in L^2(\R^d;X)$ and a.e. $s \in \R^d$
\begin{equation*}
  T_\psi h(s) := \has{(t,t') \mapsto K(s,t)\psi(t',t)h(t)}.
\end{equation*}
Note that since $T_K h(s) \in \gamma(\R^d;Y)$ by assumption, we also have $T_\psi h(s) \in \mc{L}(L^2(\R^d\times \R^d),Y)$. Moreover since $\nrm{\psi(\cdot,t)}_{L^2(\R^d)}=1$ we have
\begin{equation*}
  \nrmb{(t,t') \mapsto \ipb{K(s,t)\psi(t',t)h(t),y^*}}_{L^2(\R^d\times \R^d)} = \nrmb{t \mapsto \ipb{K(s,t)h(t),y^*}}_{L^2(\R^d)}
\end{equation*}
for every $y^* \in Y^*$. Thus by domination (see \cite[Theorem 9.4.1]{HNVW17})  it follows that $T_\psi h(s) \in \gamma(\R^d\times \R^d;Y)$ with
\begin{align*}
\nrmb{T_Kh(s)}_{\gamma(\R^d;Y)}
&=\nrmb{T_\psi h(s)}_{\gamma(\R^d\times \R^d;Y)}.
\end{align*}
Finally let
$$\widetilde{T}_K\colon\gamma(\R^d;L^p(\R^d;X)) \to \gamma(\R^d;L^{p,\infty}(\R^d;\gamma(\R^d;Y)))$$
 be the canonical extension of $T_K$, which is trivially bounded with norm $A_0$. By Lemma \ref{lemma:gammaL2} and the $\gamma$-Fubini embedding in Proposition \ref{proposition:gammaFubLorentz}, $\widetilde{T}_K$ is also bounded as an operator
 \begin{equation*}
   \widetilde{T}_K\colon L^p(\R^d;\gamma(\R^d;X)) \to  L^{p,\infty}(\R^d;\gamma(\R^d\times \R^d;Y))
 \end{equation*}
 with norm $C_p\, \tau_{2,Y} A_0$. Combined with \eqref{eq:Sestimate} this implies that $\sum_j\widetilde{T}_KS_{j}b_j$ is well-defined.

\bigskip

Using these auxiliary operators we now decompose as follows:
\begin{align*}
  \absb{\cbrace{\nrm{T_K b}_{\gamma(\R^d;Y)}>\lambda/2}} &= \absb{\cbrace{\nrm{T_{\psi}b}_{\gamma(\R^d\times\R^d;Y)}>\lambda/2}}\\
  &\leq \abss{\cbraceb{\nrmb{T_\psi b - \sum_j\widetilde{T}_KS_{j}b_j}_{\gamma(\R^d\times\R^d;Y)}>\lambda/4}\setminus O} \\ &\qquad +\abss{\cbraceb{\nrmb{\sum_j{\widetilde{T}_KS_{j}}b_j}_{\gamma(\R^d\times \R^d;Y)}>\lambda/4}} +\abs{O}\\
  &=: \text{\framebox[15pt]{A}}+\text{\framebox[15pt]{B}}+\text{\framebox[15pt]{C}}
\end{align*}
To estimate \framebox[15pt]{A} we first note that by Chebyshev's inequality and Lemma \ref{lemma:gammaL2} we have
\begin{equation*}
  \framebox[15pt]{A} \leq \tau_{2,Y}^2 \frac{16}{\lambda^2}\int_{\R^d \setminus O} \nrmb{T_\psi b-\sum_j\widetilde{T}_KS_{j}b_j}_{L^2(\R^d\times\R^d;Y)}^2
\end{equation*}
Using the fact that the $b_j$'s are disjointly supported on the cubes $Q_j \subseteq B_j$, Fubini's theorem, the $2$-H\"ormander condition and \eqref{eq:suppsi} we deduce
\begin{align*}
  &\int_{\R^d \setminus O} \nrmb{T_\psi b-\sum_j\widetilde{T}_KS_{j}b_j}_{L^2(\R^d \times \R^d;Y)}^2\\&\leq \sum_j   \int_{\R^d \setminus B_j'} \int_{B_j} \int_{\R^d} \nrmb{\hab{K(s,t) - K(s,t')}\psi_{r_j}(t',t)b_j(t)}_{Y}^2\dd t'   \dd t \dd s
  \\&\leq \sum_j{\int_{\R^d} \int_{B_j} \int_{\R^d \setminus B_j'} \nrm{{K(s,t) - K(s,t')}}^2\nrm{b_j(t)}_{X}^2 \dd s \dd t} \cdot   {\sup_{t \in B_j} \psi_{r_j}(t',t)^2 \dd t'}
  \\&\leq C_d \, \nrm{K}_{\Hormander{2}}^2 \sum_j\nrm{b_j}_{L^2(\R^d;X)}^2.
\end{align*}
Therefore by the norm estimate of the $b_j$'s in terms of $f$ we have
\begin{align*}
  \framebox[15pt]{A}  &\leq C_d\, \tau_{2,Y}^2  \nrm{K}_{\Hormander{2}}^2 \frac{\nrm{f}_{L^2(\R^d;X)}^2}{\lambda^2}.
\intertext{For \framebox[15pt]{B} we use the boundedness of $\widetilde{T}_K$ and  \eqref{eq:Sestimate} to obtain}
  \text{\framebox[15pt]{B}} &\leq C_p\, \frac{(\tau_{2,Y} A_0)^p}{\lambda^p} \nrmb{\sum_{j} S_{j}b_j}_{L^p(\R^d;\gamma(\R^d;X)))}^p \\&\leq C_{p,d} \hab{\tau_{2,X}\tau_{2,Y} A_0}^p \mu^{p-2}
\frac{\nrm{f}_{L^2(\R^d;X)}^2}{\lambda^2}
\intertext{and for \framebox[15pt]{C} we have by the estimate of $\abs{Q_j}$ in terms of $f$ that}
    \text{\framebox[15pt]{C}}&\leq \sum_j\abs{B_j'} \leq C_d \sum_j\abs{Q_j} \leq C_d\, \mu^{-2}\frac{\nrm{f}_{L^2(\R^d;X)}^2}{\lambda^2}.
  \end{align*}

Plugging the estimate for $g$ and the estimates for $b$ into \eqref{eq:fgbsplit}, we now have for all compactly supported $f \in L^{2}(\R^d;X)\cap L^{p}(\R^d;X)$ that
\begin{equation*}
  \nrm{T_Kf}_{L^{2,\infty}(\R^d;\gamma(\R^d;Y))} \leq C_{p,d} \has{\frac{\hab{\tau_{2,X}\tau_{2,Y}A_0\, \mu}^p+1}{\mu} + \tau_{2,Y} \nrm{K}_{\Hormander{2}}}\nrm{f}_{L^2(\R^d;X)},
\end{equation*}
where we used that $\tau_{2,X},\tau_{2,Y} \geq 1$. By density this estimate extends to all $f \in L^{2}(\R^d;X)$, so choosing $\mu := (\tau_{2,X}\tau_{2,Y}A_0)^{-1}$ finishes the proof of the weak $L^2$-endpoint.
\end{proof}

\begin{remark}\label{remark:endpointnotok}
In general one cannot expect $T_K \in \mc{K}_\gamma(L^2(\R^d))$ in Theorem \ref{theorem:extrapolationdown} , which is already clear from the scalar case. For instance the kernel $K(s,t) = \frac{1}{(s+t)^{1/2}}\ind_{s,t>0}$ of Example \ref{example:scalarcase} is a  $2$-H\"ormander kernel. However, $L^p$-boundedness holds only for $p\in (2, \infty)$.
\end{remark}

\subsection{Extrapolation for \texorpdfstring{$p<q<\infty$}{p<q<infty}}
We now turn our attention to extrapolation upwards for $\gamma$-integral operator. We will show that if $K \in \mc{K}_\gamma(L^{p,\infty}(\R^d))$ satisfies the $2$-H\"ormander condition, then also $K \in \mc{K}_\gamma(L^{q}(\R^d))$ for all $q \in (p,\infty)$ and we will prove a $\BMO$-endpoint result.

\begin{theorem}[Extrapolation upwards]\label{theorem:extrapolationup}
Let $X$ and $Y$ be Banach spaces and assume that $Y$ has type $2$. Let $p \in [2,\infty)$ and suppose $K \in \mc{K}_\gamma(L^{p,\infty}(\R^d))$ satisfies the $2$-H\"ormander condition. Then
\begin{enumerate}[(i)]
 \item \label{it:up1} $K \in \mc{K}_\gamma(L^{q}(\R^d))$ for all $q \in (p,\infty)$ with
 \begin{align*}
  \nrm{K}_{\mc{K}_\gamma(L^{q}(\R^d))} &\leq C_{p,q,d}\,\has{\nrm{K}_{\mc{K}_\gamma(L^{p,\infty}(\R^d))} + \tau_{2,Y} \nrm{K}_{\Hormander{2}}}.
\end{align*}
\item \label{it:up2} There exists a $\widetilde{T}_K \in \mc{L}(L^\infty(\R^d;X), \BMO(\R^d;\gamma(\R^d;Y)))$ such that
    \begin{align*}
  \nrmb{\widetilde{T}_K}_{L^\infty \to \BMO} &\leq C_{p,d}\,\has{\nrm{K}_{\mc{K}_\gamma(L^{p,\infty}(\R^d))}  + \tau_{2,Y} \nrm{K}_{\Hormander{2}}}.
  \end{align*}
  and $\widetilde{T}_Kf - T_Kf$ is constant for all $f \in L^p(\R^d;X)\cap L^\infty(\R^d;X)$.
 \end{enumerate}
\end{theorem}

\begin{remark}The extension of $T_K$ to all $f \in L^\infty(\R^d;X)$ in Theorem \ref{theorem:extrapolationup}\ref{it:up2} is not in the traditional sense, as even for $f \in L^p(\R^d;X)\cap L^\infty(\R^d;X)$ the extension $\widetilde{T}_Kf$ may not coincide with $T_Kf$. However, as $\widetilde{T}_Kf$ and $T_Kf$ only differ by a constant in this case, they represent the same function in the Banach space $$\BMO(\R^d;\gamma(\R^d;Y))/\gamma(\R^d;Y).$$ Furthermore we cannot claim uniqueness, as $L^p(\R^d;X)\cap L^\infty(\R^d;X)$ is not dense in $L^\infty(\R^d;X)$
\end{remark}

In order to prove Theorem \ref{theorem:extrapolationup}, we need to introduce local versions of the operator $T_K$. For any cube $Q \subseteq \R^d$ we define the local operator
$$T^Q_K:L^\infty(\R^d;X) \to L^p(Q;\gamma(\R^d;Y))$$
for  $s \in Q$ and $\varphi\in L^2(\R^d)$ by
\begin{equation*}
    T^Q_Kf(s)\varphi:= T_K(\ind_{B}f)(s)\varphi + \avint_Q \int_{\R^d \setminus B}\hab{K(s,t)-K(s',t)}f(t)\varphi(t)\dd t\dd s',
\end{equation*}
where $B$ is the ball with the same center as $Q$ and twice the diameter of $Q$. Note that $T^Q_K$ is well-defined since $\ind_{B}f \in L^p(\R^d;X)$ and for a.e. $s,s' \in Q$ we have
  \begin{equation}\label{eq:Tbestimate}
     \nrmb{\hab{K(s,\cdot)- K(s',\cdot)}\ind_{\R^d \setminus B}f}_{\gamma(\R^d;Y)} \leq \tau_{2,Y} \nrm{K}_{\Hormander{2}}\nrm{f}_{L^\infty(\R^d;X)}.
  \end{equation}
  by Lemma \ref{lemma:gammaL2}. Heuristically one may think about $T^Q_K$ as
  \begin{equation*}
    T^Q_Kf(s)= T_K(f)(s)+ \avint_Q \hab{K(s',\cdot)}f(\cdot)\ind_{\R^d \setminus B}(\cdot)\dd s',
\end{equation*}
 which is of course not well-defined in general.

   These operators satisfy the following properties:

\begin{lemma}\label{lemma:TQ}
Let $X$ and $Y$ and be Banach spaces and assume that $Y$ has type $2$. Let $p \in [2,\infty)$ and  suppose $K \in \mc{K}_\gamma(L^{p,\infty}(\R^d))$ satisfies the $2$-H\"ormander condition. Let $Q, Q'\subseteq \R^d$ be cubes, then
\begin{enumerate}[(i)]
  \item \label{it:TQlemma1} for all  $f \in L^\infty(\R^d;X)$ we have
\begin{equation*}
    \nrmb{T^Q_Kf}_{L^{p,\infty}(Q;\gamma(\R^d;Y))} \leq C_{p,d} \hab{ \nrm{K}_{\mc{K}_\gamma(L^{p,\infty}(\R^d))} + \tau_{2,Y} \nrm{K}_{\Hormander{2}} }\abs{Q}^{1/p}\nrm{f}_{L^\infty(\R^d;X)}.
  \end{equation*}
  \item \label{it:TQlemma2} for all  $f \in L^p(\R^d;X) \cap L^\infty(\R^d;X)$ there exists a $c \in \gamma(\R^d;Y)$ such that
\begin{align*}
    T_Kf(s) -T^Q_Kf(s) &= c
\end{align*}
for a.e.  $s \in Q$.
 \item \label{it:TQlemma3} for all  $f \in L^\infty(\R^d;X)$ there exists a $c \in \gamma(\R^d;Y)$ such that
\begin{align*}
  T^Q_Kf(s) - T_K^{Q'}f(s) &= c
  \end{align*}
  for a.e. $s \in Q\cap Q'$,
\end{enumerate}
\end{lemma}

\begin{proof}
   Let $B,B'\subseteq \R^d$ be the balls with the same center as $Q,Q'$ and twice the diameter of $Q,Q'$. Take $f \in L^\infty(\R^d;X)$, then by the assumption on $T_K$ we have
  \begin{align*}
    \nrm{T_K(\ind_{B}f)}_{L^{p,\infty}(\R^d;\gamma(\R^d;Y))} &\leq \nrm{K}_{\mc{K}_\gamma(L^{p,\infty}(\R^d))}  \nrm{\ind_{B}f}_{L^p(\R^d;X)}\\&\leq C_{d} \,\nrm{K}_{\mc{K}_\gamma(L^{p,\infty}(\R^d))} \abs{Q}^{1/p}\nrm{f}_{L^\infty(\R^d;X)}.
  \end{align*}
  Since $\nrm{\ind_{Q}}_{L^{p,\infty}(\R^d)} = C_p\,\abs{Q}^{1/p}$, \ref{it:TQlemma1} now readily follows using the definition of $T^Q_Kf$ and \eqref{eq:Tbestimate}.

  Next take $f \in L^p(\R^d;X) \cap L^\infty(\R^d;X)$ and let $s,s' \in Q$. Then if we define $c := \avint_Q T_K(\ind_{\R^d \setminus B}f)(s')\dd s'$, we have for a.e. $s \in Q$ that
\begin{align*}
  T_Kf(s) &= T_K(\ind_{B} f)(s)  +  T_K(\ind_{\R^d \setminus B}f)(s) - \avint_Q T_K(\ind_{\R^d \setminus B}f)(s')\dd s' + c\\
  &= T_{K}^Qf(s) +c
\end{align*}
proving \ref{it:TQlemma2}.

For \ref{it:TQlemma3} by considering a cube $Q''$ containing both $Q$ and $Q'$ we may assume without loss of generality that $Q'\subseteq Q$ and thus also $B'\subseteq B$. Fix $\varphi \in L^2(\R^d)$ and define
\begin{equation*}
  g(s,s',t):= \hab{K(s,t)-K(s',t)} f(t) \varphi(t).
\end{equation*}
Then we have for a.e. $s \in Q\cap Q'$ by Fubini's theorem
  \begin{align*}
    T^Q_K&f(s)\varphi - T_K^{Q'}f(s)\varphi\\ &= T_K(\ind_{B\setminus B'}f)(s)\varphi
    + \has{\avint_Q \int_{\R^d \setminus B} - \avint_{Q'} \int_{B \setminus B'} - \avint_{Q'} \int_{\R^d \setminus B}} g(s,s',t)\dd t\dd s'\\
&= \avint_{Q'} T_K(\ind_{B\setminus B'}f)(s')\varphi \dd s'+ \int_{\R^d \setminus B} \has{\avint_Q  -\avint_{Q'}} g(s,s',t)\dd s'\dd t\\
&= \avint_{Q'} T_K(\ind_{B\setminus B'}f)(s')\varphi \dd s'- \int_{\R^d \setminus B} \has{\avint_Q  -\avint_{Q'}} K(s',t) f(t) \varphi(t) \dd s'\dd t.
  \end{align*}
As the final right-hand side does not depend on $s$, this proves \ref{it:TQlemma3}.
\end{proof}

Using the properties of these local operators $T_K^Q$ we can prove an $L^\infty$-estimate of $T_K$ involving the sharp maximal operator.
\begin{proposition}
  \label{proposition:keyextrapolationup}
Let $X$ and $Y$ be Banach spaces and assume that $Y$ has type $2$. Let $p \in [2,\infty)$ and suppose $K \in \mc{K}_\gamma(L^{p,\infty}(\R^d))$ satisfies the $2$-H\"ormander condition. Then for all $f  \in L^p(\R^d;X) \cap L^\infty(\R^d;X)$ we have
\begin{align*}
    \nrmb{M^{\#}\hab{T_Kf}}_{L^\infty(\R^d)} \leq C_{d}\,\hab{ \,\nrm{K}_{\mc{K}_\gamma(L^{p,\infty}(\R^d))} + \tau_{2,Y} \nrm{K}_{\Hormander{2}}} \nrm{f}_{L^\infty(\R^d;X)}.
  \end{align*}
\end{proposition}

\begin{proof}
  Let $Q\subseteq \R^d$ be a cube and let $f \in L^p(\R^d;X) \cap L^\infty(\R^d;X)$. Using Lemma \ref{lemma:TQ}\ref{it:TQlemma2}, choose $c \in \gamma(\R^d;Y)$ such that for a.e. $s\in Q$
  \begin{align*}
    T_Kf(s)-T^Q_Kf(s) &= c.
  \end{align*}
  Therefore using \eqref{eq:weakLpembeddingL1} and Lemma \ref{lemma:TQ}\ref{it:TQlemma1}, we have
  \begin{align*}
    \avint_Q\nrms{T_Kf(s)-\avint_Q &T_Kf(t)\dd t}_{\gamma(\R^d;Y)}\dd s \\
    &\leq 2\avint_Q\nrmb{T_Kf-c}_{\gamma(\R^d;Y)} \\&= 2\avint_Q\nrmb{T^Q_Kf}_{\gamma(\R^d;Y)}\\
     &\leq 2\abs{Q}^{-1+1/p'} \,\nrmb{T^Q_Kf}_{L^{p,\infty}(Q;\gamma(\R^d;Y))}  \\
    &\leq C_{d}\,\hab{\nrm{K}_{\mc{K}_\gamma(L^{p,\infty}(\R^d))} + \tau_{2,Y} \nrm{K}_{\Hormander{2}} }\nrm{f}_{L^\infty(\R^d;X)}.
  \end{align*}
   Therefore we have
  \begin{align*}
    \nrmb{M^{\#}\hab{T_Kf}}_{L^\infty(\R^d)} \leq C_{d} \,(\nrm{K}_{\mc{K}_\gamma(L^{p,\infty}(\R^d))} + \tau_{2,Y} \nrm{K}_{\Hormander{2}} ) \nrm{f}_{L^\infty(\R^d;X)},
  \end{align*}
  which proves the proposition.
\end{proof}

Using Proposition \ref{proposition:keyextrapolationup}, the proof of Theorem \ref{theorem:extrapolationup}\ref{it:up1} is now a straightforward application of Stampacchia interpolation (see e.g. \cite[Theorem II.3.7]{GR85}).

\begin{proof}[Proof of Theorem \ref{theorem:extrapolationup}\ref{it:up1}]
Let $f \in L^p(\R^d;X) \cap L^\infty(\R^d;X)$. Note that trivially $M^{\#}f \leq 2\,M (\nrm{f}_{\gamma(\R^d;Y)})$, so by Lemma \ref{lemma:maximaloperator} we know that $M^{\#}$ is bounded from  $L^{p,\infty}(\R^d;\gamma(\R^d;Y))$ to $L^{p,\infty}(\R^d)$. Thus
\begin{align*}
  \nrmb{M^{\#}\ha{T_Kf}}_{L^{p,\infty}(\R^d)} &\leq C_{p,d} \nrmb{{T_K}f}_{L^{p,\infty}(\R^d;\gamma(\R^d;Y))}\\
  &\leq C_{p,d}\, \nrm{K}_{\mc{K}_\gamma(L^{p,\infty}(\R^d))}  \nrm{f}_{L^{p}(\R^d;X)}.
\intertext{Moreover by Proposition \ref{proposition:keyextrapolationup} we know that}
    \nrmb{M^{\#}\hab{T_Kf}}_{L^\infty(\R^d)} &\leq C_{p,d}\,\hab{\nrm{K}_{\mc{K}_\gamma(L^{p,\infty}(\R^d))}  + \tau_{2,Y} \nrm{K}_{\Hormander{2}} } \nrm{f}_{L^\infty(\R^d;X)},
  \end{align*}
We can therefore apply the Marcinkiewicz interpolation theorem (see e.g. \cite[Theorem 2.2.3]{HNVW16}), to conclude that for all $f \in L^p(\R^d;X)\cap L^q(\R^d;X)$ we have
\begin{equation*}
  \nrmb{M^{\#}\ha{T_Kf}}_{L^{q}(\R^d)} \leq C_{p,q,d}\hab{\nrm{K}_{\mc{K}_\gamma(L^{p,\infty}(\R^d))} +\tau_{2,Y}\nrm{K}_{\Hormander{2}}}\nrm{f}_{L^q(\R^d;X)}.
\end{equation*}
By Lemma \ref{lemma:sharpmaximal} we deduce
\begin{align*}
  \nrm{T_Kf}_{L^q(\R^d;\gamma(\R^d;Y))} &\leq C_{q,d} \, \nrmb{M^{\#}\ha{T_Kf}}_{L^{q}(\R^d)} \\&\leq C_{p,q,d}\hab{\nrm{K}_{\mc{K}_\gamma(L^{p,\infty}(\R^d))} +\tau_{2,Y}\nrm{K}_{\Hormander{2}}}\nrm{f}_{L^q(\R^d;X)}.
\end{align*}
for all $f \in L^p(\R^d;X)\cap L^q(\R^d;X)$. As this is a dense subspace of $L^q(\R^d;X)$, assertion \ref{it:up1} of Theorem \ref{theorem:extrapolationup} follows.
\end{proof}

Assertion \ref{it:up2} of Theorem \ref{theorem:extrapolationup} does not follow directly from Proposition \ref{proposition:keyextrapolationup}, since $L^p(\R^d;X) \cap L^\infty(\R^d;X)$ is not dense in $L^\infty(\R^d;X)$ and therefore the extension of $T_K$ to all functions in  $L^\infty(\R^d;X)$ is a nontrivial matter.

\begin{proof}[Proof of Theorem \ref{theorem:extrapolationup}\ref{it:up2}]
 Let $(Q_k)_{k=1}^\infty$ be an increasing sequence of cubes such that $\bigcup_{k=1}^\infty Q_k = \R^d$. For $f \in L^\infty(\R^d;X)$ define
  \begin{equation*}
    \widetilde{T}_Kf(s):=T^{Q_k}_Kf(s)- \avint_{Q_1}T^{Q_k}_Kf\qquad \text{if } s \in Q_k.
  \end{equation*}
  Then $\widetilde{T}_Kf \in L^1_{\loc}(\R^d;\gamma(\R^d;Y))$ is well-defined. Indeed, by Lemma \ref{lemma:TQ}\ref{it:TQlemma2} we have $T^{Q_k}_Kf \in L^1(Q_k;\gamma(\R^d;Y))$, so in particular the average over $Q_1$ is well-defined. Moreover if $j>k$, then by Lemma \ref{lemma:TQ}\ref{it:TQlemma3}  there is a $c \in \gamma(\R^d;Y)$ such that $T^{Q_j}_Kf(s) -T^{Q_k}_Kf(s)= c$ for a.e. $s \in Q_k$. Therefore
  \begin{align*}
    T^{Q_k}_Kf(s)- \avint_{Q_1}T^{Q_k}_Kf &= \hab{T^{Q_j}_Kf(s)-c} - \has{\avint_{Q_1} T^{Q_j}_Kf - c} \\&= T^{Q_j}_Kf(s) - \avint_{Q_1} T^{Q_j}_Kf,
  \end{align*}
  thus the definition of $\widetilde{T}_Kf(s)$ is independent of the choice of $Q_k \ni s$.

  If $f \in L^p(\R^d;X)\cap L^\infty(\R^d;X)$, then for any $k \in \N$ there exist $c_1,c_2 \in \gamma(\R^d;Y)$ such that for a.e. $s \in Q_k$
  \begin{align*}
    T^{Q_k}_Kf(s)-T_Kf(s) &= c_1,\\
    T^{Q_k}_Kf(s)-\widetilde{T}_Kf(s) &= c_2
  \end{align*}
  by Lemma \ref{lemma:TQ}\ref{it:TQlemma2} and the definition of $\widetilde{T}_Kf$. As $(Q_k)_{k=1}^\infty$ is increasing and $\bigcup_{k=1}^\infty Q_k = \R^d$, we see that $c_1$ and $c_2$ are independent of $k$, so $\widetilde{T}_Kf-T_Kf$ is indeed constant.

  It remains to show that $\widetilde{T}_Kf \in \BMO(\R^d;X)$ with the claimed norm estimate. Let $Q\subseteq \R^d$ be any cube and fix $k \in \N$ such that $Q \subseteq Q_k$. By Lemma \ref{lemma:TQ}\ref{it:TQlemma3} there exists a $c_3 \in \gamma(\R^d;Y)$ such that for a.e. $s\in Q$
  \begin{align*}
    T^{Q_k}_Kf(x)-T^Q_Kf(s) &= c_3.
  \end{align*}
  Therefore
  \begin{align*}
    \nrmb{\widetilde{T}_Kf}_{\BMO(\R^d;\gamma(\R^d;Y))} \leq \avint_Q\nrmb{\widetilde{T}_Kf-(c_3-c_2)}_{\gamma(\R^d;Y)} = \avint_Q\nrmb{T^Q_Kf}_{\gamma(\R^d;Y)}.
  \end{align*}
  Now $\avint_Q\nrmb{T^Q_Kf}_{\gamma(\R^d;Y)}$ can be estimated exactly as in the proof of Proposition \ref{proposition:keyextrapolationup}, which yields the claimed norm estimate in Theorem \ref{theorem:extrapolationup}\ref{it:up2}.
\end{proof}

\begin{remark}
  By inspection of the proof it can easily be seen that for the extrapolation down in Theorem \ref{theorem:extrapolationdown} one only needs
  \begin{equation*}
    \has{\int_{\R^d \setminus B} \nrm{\ha{K(s,t)-K(s,t')}x}_Y^2\dd t}^{1/2} \leq C \,\nrm{x}_X,\qquad  s,s'\in \tfrac{1}{2}B , \quad x \in X\\
  \end{equation*}
  which is implied by \eqref{eq:horm2} of the $2$-H\"ormander condition. For the extrapolation up in Theorem \ref{theorem:extrapolationup} one only needs the left hand side of \eqref{eq:Tbestimate} to be bounded, which is implied by \eqref{eq:horm1} of the $2$-H\"ormander condition.
\end{remark}

\begin{remark}\label{remark:KimKim}\
\begin{itemize}
\item In \cite{KK16} a real-valued extrapolation result was proved under a parabolic H\"ormander condition which allows one to extend $L^2(\Omega\times \R_+\times D)$-boundedness to $L^p(\Omega\times \R_+\times D)$-boundedness. In applications to SPDEs this result can be combined with ours to extrapolate $L^2(\Omega\times \R_+\times D)$-boundedness to $L^p(\Omega\times\R_+;L^q(D))$-boundedness for all $p\in (2, \infty)$ and $q\in [2, \infty)$ (see Examples \ref{ex:BesselHeat} and \ref{example:domainneumann}).
\item Another type of BMO end-point estimate was obtained in \cite{Ki15}, but the result seems incomparable with ours.
\end{itemize}
\end{remark}

\begin{corollary}[$\gamma$-convolution operator with values in a Hilbert space]\label{corollary:Hilbertspacesch}
Let $X$ be a Banach space and $Y$ be a Hilbert space. Suppose $k:\R^d\to \mc{L}(X,Y)$ is strongly measurable and satisfies the $2$-H\"ormander condition in \eqref{eq:hormconvolution}.
Let $K(s,t) = k(s-t)$. Then the following are equivalent:
\begin{enumerate}[(i)]
\item\label{it:Hilbertspaceschi} $\|t\mapsto k(t)x\|_{L^2(\R^d;Y)}\leq A_0 \,\|x\|$ for some $A_0>0$;
\item\label{it:Hilbertspaceschii} $K \in \mc{K}_\gamma(L^{p}(\R^d))$ for all $p\in [2, \infty)$.
\item\label{it:Hilbertspaceschiii} $K \in \mc{K}_\gamma(L^{p,\infty}(\R^d))$ for some $p\in [2, \infty)$;
\end{enumerate}
In particular we have for all $p \in [2,\infty)$ and $A_0$ as in \ref{it:Hilbertspaceschi}:
\begin{equation*}
  \nrm{K}_{\mc{K}_\gamma(L^p(\R^d))} \leq C_{p,d} \, (A_0 + \nrm{K}_{\Hormander{2}}).
\end{equation*}
\end{corollary}
\begin{proof}
The implication \ref{it:Hilbertspaceschi} $\Rightarrow$ \ref{it:Hilbertspaceschii} for $p=2$ follows from Proposition \ref{proposition:simplesufficient}\ref{it:sufficientL2} and for $p\in (2, \infty)$ we can apply Theorem \ref{theorem:extrapolationup}. The implication \ref{it:Hilbertspaceschii} $\Rightarrow$ \ref{it:Hilbertspaceschiii}  is trivial and \ref{it:Hilbertspaceschiii} $\Rightarrow$ \ref{it:Hilbertspaceschi} follows from Proposition \ref{proposition:Kconvolutionnecessary}.
\end{proof}

\section{Sparse domination for \texorpdfstring{$\gamma$}{y}-integral operators}\label{section:sparse}
In this section we will obtain weighted bounds for a $\gamma$-integral operator $T_K$ under an $(\omega,2)$-Dini condition on $K$. We will deduce these weighted bounds by estimating the operator $T_K$ pointwise by a much simpler operator. These simpler operators satisfy weighted bounds, which then imply weighted bounds for $T_K$.

Let us start by defining these simpler operators. We say that a collection of cubes $\mc{S}$ in $\R^d$ is $\eta$-sparse for some $\eta \in (0,1)$ if for every $Q \in \mc{S}$ there exists an $E_Q \subseteq Q$ such that $\abs{E_S}\geq \eta \abs{S}$ and such that the collection $\cbrace{E_Q:Q \in \mc{S}}$ is pairwise disjoint. Typically $\eta$ will only depend on the dimension $d$. We will dominate the $\gamma$-integral operators by operators of the form
\begin{equation*}
  f \mapsto \has{\sum_{Q \in \mc{S}}\has{\avint_Q \nrm{f}_X^2}\ind_Q}^{1/2}
\end{equation*}
for some $\eta$-sparse collection of cubes $\mc{S}$. These sparse operators are well-known to satisfy weighted bounds, see \cite[Proposition 4.1]{Lo19b}.

The sparse domination approach was developed in order to prove the so called $A_2$-conjecture, first solved by Hyt\"onen in \cite{Hy12}. The particular result we need stems from Lacey's simple proof of the $A_2$-conjecture \cite{La17b}, further clarified and simplified by Lerner \cite{Le16} and later by Lerner and Ombrosi \cite{LO19}. We will use a version of this result by the first author \cite{Lo19b}, which is adapted to our stochastic vector-valued setting.
In order to use this result we need to study a grand maximal truncation operator associated to  the operator that we wish to dominate by a sparse operator.

Let $X$ and $Y$ be Banach spaces, $p \in [1,\infty)$ and let $T$ be a bounded operator from $L^p(\R^d;X)$ to $L^{p,\infty}(\R^d;Y)$. Define for $f \in L^p(\R^d;X)$, $\alpha\geq 1$ and $s \in \R^d$
\begin{equation*}
  \mc{M}_{T,\alpha}^{\#}f(s) := \sup_{B \ni s}\, \esssup_{s',s''\in B} \,\nrmb{T(\ind_{\R^d \setminus \alpha B}f)(s')-T(\ind_{\R^d \setminus \alpha B}f)(s'')}_{Y},
\end{equation*}
where the first supremum is taken over all balls $B \subseteq \R^d$ containing $s$ and $\alpha B$ is the dilation of $B$ by a factor $\alpha$. The following theorem is a special case of \cite[Theorem 1.1]{Lo19b}.

\begin{theorem}[Abstract sparse domination]\label{theorem:domination}
Let $X$ and $Y$ be a Banach spaces, $p,r\in [1,\infty)$ and $\alpha \geq 6$. Assume the following conditions:
\begin{itemize}
  \item $T$ is a bounded linear operator from $L^{p}(\R^d;X)$ to $L^{p,\infty}(\R^d;Y)$.
  \item $\mc{M}_{T,\alpha}^{\#}$ is bounded from $L^{p}(\R^d;X)$ to $L^{p,\infty}(\R^d)$.
  \item For disjointly supported $f_1,\cdots,f_n \in L^{p}(\R^d;X)$ we have
\begin{equation}\label{eq:rsublinear}
  \qquad \nrms{T\has{\sum_{k=1}^n f_k}(s)}_Y\leq C_r \, \has{\sum_{k=1}^n \nrmb{Tf_k(s)}_Y^r}^{1/r},
  \qquad s \in \R^d.
\end{equation}
\end{itemize}
 Then there exists an $\eta \in (0,1)$ such that for any compactly supported $f \in L^{p}(\R^d;X)$ there is an $\eta$-sparse collection of cubes $\mc{S}$ such that
\begin{align*}
   \nrm{ Tf(s)}_Y&\lesssim_{\alpha,d} C_T \,C_r\, \has{\sum_{Q \in \mc{S}} \has{\avint_Q\nrm{f}_X^p}^{r/p} \ind_{Q}(s)}^{1/r}, \qquad s \in \R^d,
\end{align*}
where $C_T={\nrm{T}_{L^{p_1}\to L^{p_1,\infty}} + \nrm{\mc{M}_{T,\alpha}^{\#}}_{L^{p_2}\to L^{p_2,\infty}}}$.
\end{theorem}

In order to apply this theorem on a $K \in \mc{K}_{\gamma}(L^p(\R^d))$ we need to check weak $L^2$-boundedness of $T_K$ and $\mc{M}_{T_K,\alpha}^{\#}$ and we need to check  equation \eqref{eq:rsublinear} with $r=2$.
For a $(\omega,2)$-Dini kernel the boundedness of $\mc{M}_{T_K,\alpha}^{\#}$ is quite easy to check for $\alpha=6$:
\begin{lemma}[Boundedness of grand maximal truncation operator]\label{lemma:Lernertruncation}
   Let $X$ and $Y$ be Banach spaces and assume that $Y$ has type $2$. Let $p \in [2,\infty)$  and suppose $K \in \mc{K}_\gamma(L^{p,\infty}(\R^d))$ satisfies the $(\omega,2)$-Dini condition. Then for any $f \in L^p(\R^d;X)$ we have
   \begin{equation*}
    \mc{M}^{\#}_{T_K,6}f  \leq C_{d} \,\tau_{2,Y}\,\nrm{K}_{\Dini{\omega}{2}} \,M_2(\nrm{f}_X).
\end{equation*}
 In particular $\mc{M}^{\#}_{T_K,6}$ is bounded from $L^2(\R^d;X)$ to $L^{2,\infty}(\R^d)$  with
\begin{align*}
    \nrmb{\mc{M}^{\#}_{T_K,6}}_{L^2(\R^d;X) \to L^{2,\infty}(\R^d)} &\leq C_{d} \,\tau_{2,Y} \,\nrm{K}_{\Dini{\omega}{2}}.
\end{align*}
\end{lemma}

\begin{proof}
Let $f \in L^p(\R^d;X)$, $s \in \R^d$ and fix a ball $B \ni s$ with radius $\rho$. Take $s',s'' \in B$ and let $\varepsilon = 4\rho$. Then $\abs{s'-s''} \leq \frac{1}{2}\varepsilon$ and for any $t \in \R^d\setminus 6 B$ we have $\abs{s'-t} > \varepsilon$. Therefore applying Lemma \ref{lemma:gammaL2} and using the $(\omega,2)$-Dini condition we obtain
  \begin{align*}
    \nrm{T_K(\ind_{\R^d \setminus 6B}&f)(s') - T_K(\ind_{\R^d \setminus 6B}f)(s'')}_{\gamma(\R^d;Y)} \\
    &\leq  \tau_{2,Y} \has{\int_{\R^d\setminus 6B}\nrmb{\hab{K(s',t) - K(s'',t)}f(t)}^2_Y\dd t}^{1/2}\\
       &\leq \tau_{2,Y} \has{ \int_{\abs{s'-t}>\varepsilon}\omega\has{\frac{\varepsilon/2}{\abs{s'-t}}}^2\frac{1}{\abs{s'-t}^d}\nrm{f(t)}_X^2\dd t}^{1/2}\\
       &\leq \tau_{2,Y} \has{ \sum_{j=0}^\infty \frac{\omega(2^{-j-1})^2}{(2^j\varepsilon)^d} \int_{2^j\varepsilon<\abs{s'-t}\leq\varepsilon2^{j+1}}\nrm{f(t)}_X^2\dd t}^{1/2}\\
       &\leq C_d \,\tau_{2,Y}  \has{ \sum_{j=0}^\infty\omega(2^{-j-1})^2 \avint_{B(s',2^{j+1}\varepsilon)}\nrm{f(t)}_X^2\dd t}^{1/2}\\
       &\leq C_d \,\tau_{2,Y} \, \nrm{K}_{\Dini{\omega}{2}}   M_2\hab{\nrm{f}_X}(s) .
  \end{align*}
    where the last step follows from $s \in B(s',2^{j+1}\varepsilon)$ for all $j \in \N$.
  Now taking the essential supremum over $s',s'' \in B$ and the supremum over all balls $B \ni s$, we see that
  \begin{equation*}
    \mc{M}^{\#}_{T_K,6}f(s) \leq  C_d \,\tau_{2,Y} \, \nrm{K}_{\Dini{\omega}{2}}   M_2\hab{\nrm{f}_X}(s), \qquad s \in S.
  \end{equation*}
  The weak $L^2$-boundedness follows from the corresponding bound for $M_2$ in Lemma \ref{lemma:maximaloperator} and the density of $L^p(\R^d;X)$ in $L^2(\R^d;X)$.
\end{proof}

We can now prove sparse domination, and thus also weighted boundedness, for the $\gamma$-integral operators

\begin{theorem}[Sparse domination for $\gamma$-integral operators]\label{theorem:sparsedomination}
Let $X$ and $Y$ be Banach spaces with type $2$. Let $p \in [2,\infty)$ and suppose $K \in \mc{K}_\gamma(L^{p,\infty}(\R^d))$ satisfies the $(\omega,2)$-Dini condition. Then there is an $\eta \in (0,1)$ such that for every compactly supported $f \in L^2(\R^d;X)$ there exists an $\eta$-sparse collection of cubes $\mc{S}$ such that
\begin{equation*}
    \nrm{Tf(s)}_{\gamma(\R^d;Y)} \leq C_{X,Y,p,d}\,c_K\, \has{\sum_{Q \in \mc{S}}\has{\avint_Q \nrm{f}_X^2}\ind_Q(s)}^{1/2}, \qquad s \in \R^d
\end{equation*}
with $c_K:= \nrm{K}_{\mc{K}_\gamma(L^{p,\infty}(\R^d))} + \nrm{K}_{\Dini{\omega}{2}}$.
In particular:
\begin{enumerate}[(i)]
  \item $K \in \mc{K}_\gamma(L^{q}(\R^d,w))$
for all $q \in (2,\infty)$ and $w \in A_{q/2}$ with
  \begin{align*}
  \nrm{K}_{\mc{K}_\gamma(L^{q}(\R^d,w))}
  &\leq C_{X,Y,p,q,d}\,c_K\, [w]_{A_{q/2}}^{\max\cbraceb{\frac{1}{2},\frac{1}{q-2}}}.
  \intertext{\item $K \in \mc{K}_\gamma(L^{2,\infty}(\R^d,w))$
for all $w \in A_{1}$ with}
  \nrm{K}_{\mc{K}_\gamma(L^{2,\infty}(\R^d,w))} &\leq C_{X,Y,p,d} \,c_K\, [w]_{A_{1}}^{1/2}\log\hab{1+[w]_{A_\infty}}^{1/2}.
\end{align*}
\end{enumerate}
\end{theorem}

\begin{proof}
Since $K$ is an $(\omega,2)$-Dini kernel, it is also a $2$-H\"ormander kernel by Lemma \ref{lemma:hormanderdini} with
\begin{equation*}
     \nrm{K}_{\Hormander{2}} \leq C_{p,d}\, \nrm{K}_{\Dini{\omega}{2}}.
\end{equation*}
Therefore by Theorem \ref{theorem:extrapolationdown} we know that $T$ is bounded from $L^2(\R^d;X)$ to $L^{2,\infty}(\R^d;\gamma(\R^d;Y))$ with norm
\begin{equation*}
  \nrm{T}_{L^2 \to L^{2,\infty}} \leq C_{p,d}\,\has{\tau_{2,X} \tau_{2,Y} \nrm{K}_{\mc{K}_\gamma(L^{p,\infty}(\R^d))} + \tau_{2,Y} \nrm{K}_{\Dini{\omega}{2}}}.
\end{equation*}
By Lemma \ref{lemma:Lernertruncation} we  also know that $\mc{M}^{\#}_{T,6}$ is bounded from $L^2(\R^d;X)$ to $L^{2,\infty}(\R^d)$ with norm
\begin{equation*}
  \nrmb{\mc{M}_{T,6}^{\#} }_{L^2 \to L^{2,\infty}} \leq C_{d}\, \tau_{2,Y} \nrm{K}_{\Dini{\omega}{2}}.
\end{equation*}
Moreover for $f_1,\cdots,f_n \in L^2(\R^d;X)$ with disjoint support we have for a.e. $s \in \R^d$ that $T_Kf_1(s)\cdots T_Kf_n(s) $ have disjoint support as well and thus \eqref{eq:rsublinear} with $r=2$ follows from Lemma \ref{lemma:2sublinear}. The sparse domination therefore follows by applying Theorem \ref{theorem:domination} to $T_K$.
The weighted bounds follow directly from \cite[Proposition 4.1]{Lo19b} and the density of compactly supported $L^2$-functions in $L^q(\R^d,w;X)$ for all $q \in [2,\infty)$.
\end{proof}

\begin{remark}~\label{remark:weightedlessassumptions}
\begin{enumerate}[(i)]
  \item \label{it:weightedlessassumptions1} If we omit the type $2$ assumption for $X$ in Theorem \ref{theorem:sparsedomination} we can still conclude that $T_K$ is sparsely dominated by larger sparse operator
      \begin{equation*}
  f \mapsto \has{\sum_{Q \in \mc{S}}\has{\avint_Q \nrm{f}_X^p}^{p/2}\ind_Q}^{1/2}
\end{equation*}
 In the proof one then has to skip the step where Theorem \ref{theorem:extrapolationdown} is applied. This is in particular useful when $p=2$.
  \item  $A_{p/2}$ is the largest class of weights one can expect in Theorem \ref{theorem:sparsedomination}, since in the case that $X=Y=\K$ and $K(s,t)=k(s-t)$, Theorem \ref{theorem:sparsedomination} can be reduced to a statement about deterministic convolution operators with positive kernel (see Subsection \ref{section:scalarcase}).
      It is standard to check that the weighted boundedness of for example
\begin{equation*}
Tf(s) :=  \int_{\R^d} \lambda^d \ee^{-\lambda|s-t|} f(t) \dd t, \qquad s \in \R^d,
\end{equation*}
for all $\lambda \in \R_+$ implies the $A_p$-condition, see e.g. \cite[Section 7.1.1]{Gr14a}. Also the dependence on the weight characteristic is sharp, see Proposition \ref{proposition:weightsharp} below
\end{enumerate}
\end{remark}

Under a Dini type condition we obtain the following further characterization if $Y$ is a Hilbert space. The proof is immediate from Corollary \ref{corollary:Hilbertspacesch}, Theorem \ref{theorem:sparsedomination} and Remark \ref{remark:weightedlessassumptions}\ref{it:weightedlessassumptions1}.
\begin{corollary}\label{corollary:Hilbertspaceschweight}
Let $X$ be a Banach space and $Y$ be a Hilbert space. Suppose $k:\R^d\to \mc{L}(X,Y)$ is strongly measurable and satisfies the $(\omega,2)$-Dini condition in \eqref{eq:diniconvolution}.
Let $K(s,t) := k(s-t)$. Then statements \ref{it:Hilbertspaceschi}--\ref{it:Hilbertspaceschiii} in Corollary \ref{corollary:Hilbertspacesch} are equivalent to
\begin{enumerate}[(iv)]
\item\label{it:Hilbertspaceschweakii} $K \in \mc{K}_\gamma(L^{p}(\R^d,w))$ for all $p\in (2, \infty)$ and all $w\in A_{p/2}$.
\end{enumerate}
In particular we have for all $p \in (2,\infty)$, $w \in A_{p/2}$ and $A_0$ as in \ref{it:Hilbertspaceschi} of Corollary \ref{corollary:Hilbertspacesch}:
\begin{equation*}
  \nrm{K}_{\mc{K}_\gamma(L^p(\R^d))} \leq C_{p,d} \, (A_0 + \nrm{K}_{\Dini{\omega}{2}}) [w]_{A_{q/2}}^{\max\cbraceb{\frac{1}{2},\frac{1}{q-2}}}.
\end{equation*}
\end{corollary}

We will show next that the dependence on the weight characteristic $[w]_{A_{p/2}}$ in the bounds for $T_K$ in Theorem \ref{theorem:sparsedomination} is actually optimal. Therefore Theorem \ref{theorem:sparsedomination} can be thought of as a $\gamma$-analog of the $A_2$-theorem in the deterministic setting.
\begin{proposition} \label{proposition:weightsharp}
Let $X$ and $Y$ be Banach spaces and $p\in (2,\infty)$ and $\beta \geq 0$. There exists a kernel $K\colon \R^d \to \mc{L}(X,Y)$ satisfying the assumptions of Theorem \ref{theorem:sparsedomination} such that if for all $w \in A_{p/2}$ we have
\begin{equation*}
    \nrm{K}_{\mc{K}_\gamma(L^p(\R^d,w))} \lesssim [w]_{A_{p/2}}^{\beta},
\end{equation*}
then $\beta \geq \max\cbraceb{\tfrac{1}{2},\tfrac{1}{q-2}}$.
\end{proposition}

\begin{proof}
  By considering one dimensional subspaces, we may assume without loss of generality that $X=Y=\K$. Define
\begin{equation*}
  K\hab{(s_1,\bar{s}),(t_1,\bar{t}\hspace{1pt})} :=
  \frac{(\abs{s_1}+\abs{t_1})^{1/2}}{\absb{(\abs{s_1},\bar{s})+(\abs{t_1},\bar{t}\hspace{1pt})}^{(d+1)/2}}, \qquad (s_1,\bar{s}),(t_1,\bar{t}\hspace{1pt})\in \R\times \R^{d-1}.
\end{equation*}
Then by Lemma \ref{lemma:standardkernelderivatives} we know that $K$ is a $(1,2)$-standard kernel.

Set $\R_+^d:= \cbrace{(s_1,\bar{s}) \in \R \times \R^{d-1}: s_1\geq 0}$, $\R_-^d:= \R^d \setminus \R^d_+$ and define for $q \in (1,\infty)$ and $f \in L^q(\R^d_+)$
\begin{equation*}
  T_1f(s) := \frac{\Gamma((d+1)/2)}{\pi^{(d+1)/2}} \int_{\R_+^d}
  \frac{s_1+t_1}{\abs{s+t}^{d+1}}f(t) \dd t, \qquad s \in \R^d_+.
\end{equation*}
Then $T_1$ is a bounded operator on $L^q(\R_+^d)$ for all $q \in (1,\infty)$ with
$$
\nrm{T_1}_{L^{q}(\R^d_+)\to L^{q}(\R^d_+)} = \frac{1}{\sin\ha{\pi/q}},
$$
by \cite[Theorem 1]{Os17}.
For $g \in L^p(\R^d)$ we have
\begin{align*}
  \nrm{T_K g}_{L^p(\R^d)}^p &= \nrm{T_K (g\ind_{\R^d_+}+g \ind_{\R^d_-})}_{L^p(\R^d_+)}^p +\nrm{T_K (g\ind_{\R^d_+}+g \ind_{\R^d_-})}_{L^p(\R^d_-)}^p\\
  &\eqsim_d \, \nrm{T_1 h }_{L^{p/2}(\R^d_+)}^{p/2}
\end{align*}
where $h(s):= |g(s)+g(-s)|^2$ for $s \in \R^d_+$. Therefore
\begin{equation*}
  \nrm{K}_{\mc{K}_\gamma(L^p(\R^d))} \eqsim_d \nrm{T_1}_{L^{p/2}(\R^d_+)\to L^{p/2}(\R^d_+)}^{1/2} = \frac{1}{\sin\ha{2\pi/p}^{1/2}},
\end{equation*}
so $K$ satisfies the assumptions of Theorem \ref{theorem:sparsedomination}. Moreover
  \begin{align*}
    \alpha_K&:= \sup\cbraces{\alpha\geq 0: \forall \varepsilon>0, \limsup_{p\to 2}  \frac{\nrm{K}_{\mc{K}_\gamma(L^p(\R^d))}}{(p-2)^{-\alpha+\varepsilon}}=\infty} = \frac{1}{2}\\
    \gamma_K&:= \sup\cbraces{\gamma\geq 0: \forall \varepsilon>0, \limsup_{p\to \infty}  \frac{\nrm{K}_{\mc{K}_\gamma(L^p(\R^d))}}{p^{-\gamma+\varepsilon}}=\infty} = \frac{1}{2}.
  \end{align*}
Thus by \cite[Theorem 5.2]{FN18} it follows that if
\begin{equation*}
    \nrm{K}_{\mc{K}_\gamma(L^p(\R^d,w))} \lesssim [w]_{A_{p/2}}^{\beta},
\end{equation*}
then
  \begin{equation*}
    \beta \geq  \max\cbraceb{\alpha_K \, \tfrac{2}{q-2}, \gamma_K}= \max\cbraceb{\tfrac{1}{q-2},\tfrac{1}{2}}.\qedhere
  \end{equation*}
\end{proof}

\section{Extension to spaces of homogeneous type}\label{section:homogeneoustype}
In this section we will describe how Sections \ref{section:kernels}-\ref{section:sparse} can be generalized from $\R^d$ to a space of homogeneous type $(S,d,\mu)$. This will be quite useful in our applications, as we will often want to take $S = \R_+$ or $S = (0,T)$ with the Euclidean metric and the Lebesgue measure. While these examples are in a sense trivial spaces of homogeneous type, they do not follow directly from our theory on $\R^d$.

A space of homogeneous type $(S,d,\mu)$, originally introduced by Coifman and Weiss in \cite{CW71}, is a set $S$ equipped with a quasimetric $d$, i.e. a metric which satisfies
\begin{equation*}
  d(s,t) \leq c_d\, \hab{d(s,u)+d(u,t)}, \qquad s,t,u\in S
\end{equation*}
for some $c_d \geq 1$ instead of the triangle inequality, and a Borel measure $\mu$ that satisfies the doubling property, i.e.
\begin{equation*}
  \mu\hab{B(s,2r)} \leq c_\mu \,\mu\hab{B(s,r)}, \qquad s \in S,\quad r>0
\end{equation*}
for some $c_\mu\geq 1$. In addition we assume that all balls $B\subseteq S$ are Borel sets and that we have $0 <\mu(B)<\infty$. As $\mu$ is a Borel measure the Lebesgue differentiation theorem holds and $C_c(S)$ is dense in $L^p(S)$ for all $p \in [1,\infty)$, see \cite[Theorem 3.14]{AM15} for the details.

\bigskip

We will now describe how Sections \ref{section:kernels}-\ref{section:sparse} can be adapted to spaces of homogeneous type. When we say that a result remains valid when we replace $\R^d$ with a space of homogeneous type $(S,d,\mu)$, we mean implicitly that all cubes are replaced by balls with the same center and diameter and that the dependence on the dimension $d$ of the involved constants is replaced by dependence on the quasimetric constant $c_d$ and the doubling constant $c_\mu$.
\begin{itemize}
\item The weighted bounds for the Hardy--Littlewood maximal operator in Lemma \ref{lemma:maximaloperator} are still valid, since we can do a covering lemma argument similar to the one we did for $\R^d$, see \cite[Theorem 4.1]{HK12}.
\item The definition of the $2$-H\"ormander condition carries over directly to spaces of homogeneous type. For the $(\omega,2)$-Dini we replace $\abs{s-t}^d$ by $\mu(B(s,d(s,t)))$ in \eqref{eq:dini2} and by $\mu(B(t,d(t,s)))$ in \eqref{eq:dini3}.
\item
  Lemma \ref{lemma:hormanderdini} remains valid in general spaces of homogeneous type and Lemma \ref{lemma:standardkernelderivatives} as well if $S$ is a convex subset of $\R^d$ with the Euclidean distance and the Lebesgue measure. More generally, Lemma \ref{lemma:standardkernelderivatives} also remains valid if $S$ is a smooth domain in $\R^d$, as one can then locally reduce to the $\R^d_+$ case. Lemma \ref{lemma:standardkernelfractional} remains true for $(0,T)$ with $T\in (0,\infty]$.
  \item Theorem \ref{theorem:extrapolationdown} remains valid. The main part of the proof that should be adapted, is the $L^2$-Calder\'on--Zygmund decomposition.
       This decomposition in spaces of homogeneous type can be found in \cite[Theorem 3.1]{BK03} in the case $X = \C$ and the proof again carries over verbatim to the vector-valued setting. Note that this decomposition at level $\lambda>0$ holds only when
      \begin{equation*}
        \lambda \geq \has{\avint_S \abs{f}^2 \dd \mu}^{1/2},
      \end{equation*}
  where the right hand side is of course zero if $\mu(S) = \infty$. So if $\mu(S) <\infty$ we need another argument in the proof of Theorem \ref{theorem:extrapolationdown} in the case
  \begin{equation*}
    \mu\lambda \leq \has{\avint_S \nrm{f}_X^2 \dd \mu}^{1/2}.
  \end{equation*}
  But this case is trivial, since
  \begin{align*}
    \abs{\cbrace{\nrm{T_Kf}_{\gamma(\R^d;Y)}>\lambda}}^{1/2} &\leq \mu(S)^{1/2} \leq \frac{1}{\mu\lambda} \nrm{f}_{L^2(S;X)}.
  \end{align*}
  The other difference in this decomposition is that we do not obtain a disjoint decomposition of the ``bad'' part $b$, but a decomposition with bounded overlap.  One can easily check that this does not cause any problems in our proof. For instance in the inequality
  \begin{equation*}
    \nrmb{\sum_j \ind_{Q_j}}_{L^{p}(\R^d)} \leq \has{\sum_j\abs{Q_j}}^{1/p}
  \end{equation*}
  one needs to add a constant depending on the amount of overlap.
  \item Theorem \ref{theorem:extrapolationup} also remains valid. The main difficulty here is the Fefferman-Stein theorem (Lemma \ref{lemma:sharpmaximal}) in spaces of homogeneous type, which can be found in  \cite[Proposition 3.1 and Theorem 4.2]{Ma04b} or \cite[Theorem 2.4]{DK18}. When $\mu(S) < \infty$, this requires some extra care, since we then have
      \begin{equation*}
        \nrm{f}_{L^q(\R^d)} \lesssim_{q,c_d,c_\mu} \, \hab{\nrm{M^{\#}f}_{L^q(\R^d)} + \mu(S)^{-1/q'}\nrm{f}_{L^1(\R^d)}}.
      \end{equation*}
      In the proof of Theorem \ref{theorem:extrapolationup} this means that we also need to estimate $\mu(S)^{-1/q'}\nrm{T_Kf}_{L^{1}(S)}$ in terms of $\nrm{f}_{L^q(\R^d;X)}$. By the assumption that $K \in \mc{K}_\gamma(L^{p,\infty}(S))$, H\"older's inequality and \eqref{eq:weakLpembeddingL1} we have
      \begin{align*}
        \qquad\mu(S)^{-1/q'} \nrm{T_Kf}_{L^{1}(S;Y)} &\leq C_p \,\mu(S)^{1/p'-1/q'}\nrm{T_Kf}_{L^{p,\infty}(S;Y)}\\
        &\leq C_p \, \mu(S)^{1/p'-1/q'} \nrm{K}_{\mc{K}_\gamma(L^{p,\infty}(S;X))} \nrm{f}_{L^{p}(S;X)}\\
        &\leq C_p \, \nrm{K}_{\mc{K}_\gamma(L^{p,\infty}(S))} \nrm{f}_{L^{q}(S;X)},
      \end{align*}
      which is exactly the required estimate.
  \item The proof of Theorem \ref{theorem:sparsedomination} relies completely on the results in \cite{Lo19b}, which are proven in a space of homogeneous type. Therefore Theorem  \ref{theorem:sparsedomination} remains valid.
    \item Corollary \ref{corollary:Hilbertspacesch} and Corollary \ref{corollary:Hilbertspaceschweight} remain valid on $\R_+$.
\end{itemize}

\section{Applications to stochastic maximal regularity}\label{section:pindSMR}
We will now apply our results from Sections \ref{section:extrapolation}-\ref{section:homogeneoustype} to obtain stochastic maximal regularity of various SPDE's. For this we will first need some background on sectorial operators.

\subsection{Sectorial operators}
Let $X$ be a Banach space and define
\begin{equation*}
  \Sigma_\varphi = \{z\in \C\setminus\{0\}: |\arg(z)|<\varphi\}.
\end{equation*}
A closed operator $(A,D(A))$ on $X$ will be called {\em sectorial} if there is a $\varphi\in [0,\pi)$ such that $\C\setminus\overline{\Sigma}_{\varphi} \subseteq \rho(A)$ and there is a constant $A_0>0$ such that
\[\|(\lambda-A)^{-1}\| \leq A_0/|\lambda|, \ \ \lambda\in \C\setminus\Sigma_{\varphi}.\]
The infimum over all such $\varphi$ is called the \emph{angle of sectoriality} of $A$.
For details on sectorial operators we refer to \cite{EN00,Ha06,Ya10}. In particular, recall that for a sectorial operator one can define the fractional powers $A^{\alpha}$ for $\alpha\in \R$.
For $\theta>0$ and $p\in [1, \infty]$, the spaces $D_A(\theta,p)$ are defined by
\[D_A(\theta,p) = (X, D(A^n))_{\theta/n,p},\]
where $n\in \N$ is the least integer larger than $\theta$. In the above we used the real interpolation method.  The complex interpolation method will be used as well, and our notation for this will be $[X, D(A)]_{\theta}$. For details on interpolation we refer to \cite{Tr78, Ha06, Lu95}.

\subsection{Setting}
Many stochastic PDEs can be analyzed as stochastic evolution equations by using functional analytic tools. We refer to the monograph \cite{DPZ14} and the papers \cite{Br97, NVW08}.

Consider the following linear stochastic evolution equation on a Banach space $X$:
\begin{equation}\label{eq:SEEGdW}
\begin{cases}
du + A u\dd t = G \dd W_H \quad& \text{on $\R_+$},\\
u(0)=0.
\end{cases}
\end{equation}
Here $(A(t))_{t\in \R_+}$ is a family of closed operators on $X$, $W_H$ is $H$-cylindrical Brownian motion and $G:\R_+\times\Omega\to \gamma(H,X)$ is adapted to $\ms{F}$. In this paper we will focus on linear equations. Nonlinear stochastic evolution equations can be studied by using suitable estimates for the linear case (see \cite{Br97,DPZ14}). In particular, stochastic maximal regularity estimates have been applied to nonlinear SPDEs in \cite{Ag18,Br95,Ho18,KK18,Kr99,NVW12b,PV18}.

The mild solution to \eqref{eq:SEEGdW} is given by
\[u(t) = S\diamond G(t):=\int_0^t S(t,s) G(s)\dd W_H(s), \qquad t \in \R_+.\]
Here we have assumed that $-A$ generates the strongly continuous evolution family $(S(t,s))_{0\leq s\leq t}$. In the case $A$ does not depend on time, one has that $S(t,s) = e^{-(t-s)A}$ is a strongly continuous semigroup. For details and unexplained terminology on semigroups and evolution families we refer to \cite{EN00,Lu95,Pa83,Ta79,Ya10}.

\begin{definition}[Stochastic maximal regularity]\label{definition:SMR}
Let $X$ be a $\UMD$ Banach space with type $2$, $p\in [2, \infty)$ and let $w$ be a weight on $\R_+$. Let $Y\hookrightarrow X$. We say that $A$ has {\em stochastic maximal $L^p(w;Y)$-regularity} and write $A \in \SMR(L^p(w;Y))$ if for all $G\in L^p_{\ms{F}}(\Omega\times \R_+,w;\gamma(H,X))$ the mild solution $u$ to \eqref{eq:SEEGdW} satisfies
\begin{equation}\label{eq:SMRest}
\|u\|_{L^p(\Omega\times \R_+,w;Y)} \leq  C\, \|G\|_{L^{p}(\Omega\times \R_+,w;\gamma(H,X))}.
\end{equation}We omit the weight if $w \equiv 1$.
\end{definition}
Written out explicitly the estimate \eqref{eq:SMRest} becomes
\[\Big\|t\mapsto \int_0^t S(t,s) G(s)\dd W_H(s)\Big\|_{L^p(\Omega\times \R_+,w;Y)} \leq C \, \|G\|_{L^{p}(\Omega\times \R_+,w;\gamma(H,X))}.\]
Interesting choices for $Y$ are the complex and real interpolation spaces
\[Y  = [X, D(A)]_{\frac12} \quad \text{and} \quad Y = (X, D(A))_{\frac12,r}, \ \ r\in [2, \infty),\]
and the fractional domain spaces $Y = D((A+\lambda)^{1/2})$ for $\lambda$ such that $A+\lambda$ is sectorial. In several places we will use the homogenous fractional domain space $\dot{D}((A+\lambda)^{1/2})$ with norm
\[\|x\|_{\dot{D}((A+\lambda)^{1/2})} = \|(A+\lambda)^{1/2} x\|_X.\]
Note that if $A+\lambda$ is invertible, then $\dot{D}((A+\lambda)^{1/2}) = D((A+\lambda)^{1/2})$.

In \cite{NVW12} it has been shown that under certain geometric restrictions on $X$, the boundedness of the $H^\infty$-calculus of angle $<\pi/2$ of $A$ (see \cite{Ha06,HNVW17}) implies
$$A\in \SMR(L^p(\dot{D}(A^{\frac12}))).$$
Extensions to the case of time-dependent $A$ have been obtained in \cite{PV18}. Abstract properties of stochastic maximal regularity have been studied in \cite{AV19}, where in particular it was shown that if $A$ is time-independent and $A\in \SMR(L^p(D(A^{\frac12})))$, then $-A$ generates an exponentially stable analytic semigroup.
In case $A$  has a bounded $H^\infty$-calculus of some angle, then one has $[X, D(A)]_{\frac12} = D(A^{\frac12})$ (see \cite[Theorem 6.6.9]{Ha06}).

Stochastic maximal regularity can equivalently be formulated using the stochastic integral operators of Definition \ref{definition:SIOW}. In this case the kernel is given by
\[K(t,s) :=  S(t,s) \ind_{0\leq s < t}\in \mc{L}(X,Y).\]
Here we implicitly assume that $S(t,s)$ maps $X$ into $Y$.
Below we will apply the extrapolation theory of Section \ref{section:extrapolation} to study independence of $p$ and the weight for Definition \ref{definition:SMR}.

\subsection{Semigroup case} We first turn to the time independent case.

\begin{theorem}[Extrapolation in the semigroups case]\label{theorem:SMRsemigroup}
Suppose $X$ is a $\UMD$ Banach space with type $2$.
Let $A$ be sectorial of angle $<\pi/2$. Take $r\in [2, \infty)$ and assume that $Y$ is one of the spaces
\begin{align}\label{eq:spacesY}
 D(A^{1/2}), \quad \dot{D}(A^{1/2}),   \quad [X, D(A)]_{\frac12}, \quad \text{or} \quad (X, D(A))_{\frac12,r},
\end{align}
Suppose $A\in \SMR(L^p(Y))$ for some $p\in [2, \infty)$. Then for all $q\in (2, \infty)$ and $w\in A_{q/2}$ one has $A\in \SMR(L^q(w;Y))$. In particular, the mild solution $u$ to \eqref{eq:SEEGdW} satisfies
\begin{align*}
\|u\|_{L^q(\Omega\times \R_+,w;Y)} \leq C\, [w]_{A_{q/2}}^{\max\{ \frac{1}{2}, \frac{1}{q-2}\}}\|G\|_{L^{q}(\Omega\times \R_+,w;\gamma(H,X))},
\end{align*}
where $C$ only depends on $X, A, p, q$.
\end{theorem}

\begin{proof}
The space $Y$ has type $2$ with $\tau_{2,Y}\leq \tau_{2,X}$, which is trivial for $D(A^{1/2})$ and $\dot{D}(A^{1/2})$, follows from \cite[Proposition 7.1.3]{HNVW17} for $[X, D(A)]_{\frac12}$ and follows from \cite[Corollary 1]{Co83} for $(X, D(A))_{\frac12,r}$.

In all cases except for $\dot{D}(A^{1/2})$ it follows from the proof of \cite[Proposition 4.8]{AV19} that $A$ is invertible. We claim that in all cases
\[\|x\|_Y \leq C\, \|x\|^{\frac12}_X \|Ax\|^{\frac12}_X, \qquad  x\in D(A).\]
Indeed, this standard interpolation estimate follows from \cite[Corollary 1.2.7 and Proposition 2.2.15]{Lu95}, \cite[Theorem 1.10.3]{Tr78} and \cite[Proposition 6.6.4]{Ha06}. Since $t\|A e^{-tA}\|_{\mc{L}(X,Y)} \leq M$ for $t\geq 0$ and some $M>0$ (see \cite[Theorem II.4.6]{EN00}), the above interpolation estimate implies
\[\|A e^{-tA}\|_{\mc{L}(X,Y)} \leq C M^{3/2} t^{-\frac32}, \qquad  t\geq 0.\]
By assumption $K\in \mathcal{K}_{W}^H(L^p(\R_+))$.  Applying Propositions \ref{proposition:detcharacterization} and \ref{proposition:independenceH} we obtain that $K\in \mathcal{K}_\gamma(L^p(\R_+))$. Next we will check the conditions of Theorem \ref{theorem:sparsedomination} for the homogenous space $\R_+$ (see Section \ref{section:homogeneoustype}).

Let $k(t) = \ee^{-tA}$. By the analyticity of the semigroup and the above estimate, we find that for $t>0$ and
\[ \|k'(t)\|_{\mc{L}(X,Y)} = \|A e^{tA}\|_{\mc{L}(X,Y)}\leq C\,M\, t^{-\frac32}.\]
Therefore, by Lemma \ref{lemma:standardkernelderivatives} we know that $K$ is a $(1, 2)$-standard kernel with
\[\|K\|_{(1, 2)-{\rm Std}} \leq C\,M. \]
Now Theorem \ref{theorem:sparsedomination} gives that $K\in \mathcal{K}_\gamma(L^q(\R_+,w))$.
Propositions \ref{proposition:detcharacterization} and \ref{proposition:independenceH} then imply $K\in \mathcal{K}_W^H(L^q(\R_+,w))$ with the claimed estimate.
\end{proof}

\begin{remark}\label{remark:severalthingonSMR}
\
\begin{enumerate}[{(i)}]
\item \label{it:finitetimeSMR} Combining Theorem \ref{theorem:SMRsemigroup} with \cite[Section 5]{AV19}, similar results as in Theorem \ref{theorem:SMRsemigroup} hold on finite time intervals $(0,T)$. Alternatively, this can be deduced by applying Theorem \ref{theorem:sparsedomination} on $(0,T)$ (see also Section \ref{section:homogeneoustype})
\item In general the result of Theorem \ref{theorem:SMRsemigroup} does not hold in the endpoint $q=2$. A counterexample can be found in \cite[Section 6]{NVW12}.
\item Arguing as in the proof of Theorem \ref{theorem:SMRsemigroup} but with
$$K(t,s) = (t-s)^{-\theta} A^{\frac12-\theta} e^{-(t-s)A},$$ it follows that for any $\theta\in (0,\frac12)$ the property $A\in \SMR_{\theta}(p,\infty)$ introduced in \cite{AV19} is $p$-independent.
\item \label{it:admissibleY} From the proof it is clear that Theorem \ref{theorem:SMRsemigroup} holds for any space $Y\hookrightarrow X$ such that $e^{-tA}:X\to Y$ with
\[\|e^{-tA}\|_{\mc{L}(X,Y)}\leq C t^{-\frac12},  \qquad t>0.\]
\end{enumerate}
\end{remark}

\begin{corollary}\label{corollary:HScaseSMR}
Let $X$ be a Hilbert space and let $Y$ be any of the spaces in \eqref{eq:spacesY} with $r=2$. Suppose that $A$ is sectorial of angle $<\pi/2$ on $X$. Then the following are equivalent:
\begin{enumerate}[{(i)}]
\item\label{it:HScaseSMR2} There exists a constant $C>0$ such that
\[\|t\mapsto e^{-tA}x\|_{L^2(\R_+;Y)}\leq C\,\|x\|_X, \qquad x\in X;\]
\item\label{it:HScaseSMR3} For all $p\in (2, \infty)$ and all $w\in A_{p/2}$ (and $p=2$, $w\equiv 1$) we have $A\in \SMR(L^p(w;Y))$.
\item\label{it:HScaseSMR1} $A\in \SMR(L^p(Y))$ for some $p\in [2, \infty)$.
\end{enumerate}
\end{corollary}
\begin{proof}

Note that $Y$ is a Hilbert space. For \ref{it:HScaseSMR2}$\Rightarrow$\ref{it:HScaseSMR3} define $K(t,s)\in \mc{L}(X,Y)$ by
$$K(t,s) = e^{-(t-s)A} \ind_{s <t}.$$ From Proposition \ref{proposition:simplesufficient}\ref{it:sufficientL2} we obtain $K\in \mathcal{K}_\gamma(L^2(\R_+))$. Therefore $K\in \mathcal{K}_W^H(L^2(\R_+))$ by Propositions \ref{proposition:detcharacterization} and \ref{proposition:independenceH}, so the result follows from Theorem \ref{theorem:SMRsemigroup}. \ref{it:HScaseSMR3}$\Rightarrow$\ref{it:HScaseSMR1} is trivial and \ref{it:HScaseSMR1}$\Rightarrow$\ref{it:HScaseSMR2} follows from Proposition \ref{proposition:Kconvolutionnecessary} combined with Propositions \ref{proposition:detcharacterization} and \ref{proposition:independenceH}.
\end{proof}

\begin{remark}\label{remark:noHinftymagreg}~
\begin{enumerate}[(i)]
\item Corollary \ref{corollary:HScaseSMR}\ref{it:HScaseSMR2} is equivalent to the {\em admissibility} of $A^{1/2}$ and is connected to the {\em Weiss conjecture}, which was solved negatively (See \cite{JZ04}, \cite[Theorem 5.5]{Le03} and references therein).
\item It is well-known that there exist operators $A$ on a Hilbert space $X$ such that $-A$ generates an analytic semigroup which is exponentially stable and
    \begin{align*}
    \|t\mapsto A^{\frac12} e^{-tA} x\|_{L^2(\R_+;X)} \leq C\|x\|_X,
      \intertext{but}
      c\|x\|_X\nleq \|t \mapsto A^{\frac12} e^{-tA} x\|_{L^2(\R_+;X)}.
    \end{align*}
     Such $A$ can be constructed as in \cite[Theorem 5.5]{Le03} (see \cite[Section 5.2]{AV19} for details), and  does not have a bounded $H^\infty$-calculus. On the other hand, Corollary \ref{corollary:HScaseSMR} implies $A\in \SMR(L^p(w;D(A^{1/2})))$ for all $p\in [2, \infty)$ and $w\in A_{p/2}$ (with $w=1$ if $p=2)$, which shows that having a bounded $H^\infty$-calculus is not necessary for stochastic maximal regularity.
\end{enumerate}
\end{remark}

\begin{theorem}[Real interpolation scale]\label{theorem:realinterpMR}
Let $E$ be a $\UMD$ Banach space with type $2$, let $A$ be sectorial of angle $<\pi/2$ and assume $0\in \rho(A)$. Let $\theta\in (0,1)$ and $q\in [2, \infty)$. Define $X = D_A(\theta,q)$ and
$Y = D_A(\theta+\frac12,q)$.  Then for all $p\in (2, \infty)$ and $w\in A_{p/2}$, one has
$A\in \SMR(L^p(w,Y))$ (the case $p=2$ and $w=1$ is allowed as well if $q=2$). In particular, the solution $u$ to \eqref{eq:SEEGdW} satisfies
\begin{align}\label{eq:realintest}
\|A^{\frac12} u\|_{L^p(\Omega\times \R_+,w;D_A(\theta,q))} \lesssim [w]_{A_{p/2}}^{\max\{\frac{1}{2}, \frac{1}{p-2}\}}\|G\|_{L^{p}(\Omega\times \R_+,w;\gamma(H,D_A(\theta,q)))},
\end{align}
where the implicit constant only depends on $E, A,\theta, p, q$.
\end{theorem}
\begin{proof}[First proof]
Note that $X$ is a $\UMD$ Banach space with type $2$ by \cite[Proposition 4.2.17]{HNVW16} and
$-A$ is the generator of an exponentially stable analytic semigroup on $X$ with domain $D_A(\theta+1,q)$ by \cite[Proposition 2.2.7]{Lu95}. Moreover, we have $Y = (X,D_A(\theta+1,q))_{\frac12,q}$. It follows from \cite{DL98} (see also \cite[Theorem 5.1]{BH09}) and \cite[Theorem 5.2]{AV19} that $A\in \SMR(L^q(Y))$. Therefore, the required result follows from Theorem \ref{theorem:SMRsemigroup}. The claimed norm estimate follows since $A^{1/2}$ maps $D_A(\theta+\frac12,q)$ isomorphically to $D_A(\theta,q)$ (see \cite[Theorem 1.15.2]{Tr78}).
\end{proof}

Next we give a self-contained proof.
\begin{proof}[Second proof]
First consider the case $p=q=2$. By Propositions \ref{proposition:simplesufficient}\ref{it:sufficientL2}, \ref{proposition:detcharacterization} and \ref{proposition:independenceH} and \cite[Theorem 1.15.2]{Tr78} it suffices to show
\begin{align}\label{eq:squarefuncreal}
\text{\framebox[15pt]{A}}:=\Big(\int_0^\infty \|A^{\frac12} e^{-t A} x\|_{D_A(\theta,2)}^2\dd t\Big)^{\frac12}  \leq C \|x\|_{D_A(\theta,2)}.
\end{align}
Since $D_A(\theta,2) = D_{A^2}(\theta/2,2)$ (see \cite[Theorem 1.15.2]{Tr78}), by \cite[Theorem 1.14.5]{Tr78} we can write
\begin{align*}
\text{\framebox[15pt]{A}\hspace{1pt}}^2 &\eqsim \int_0^\infty \|A^{\frac12} e^{-t A} x\|_{D_{A^2}(\theta/2,2)}^2\dd t
\\ & \eqsim \int_0^\infty \int_0^{\infty} r^{4(1-\frac{\theta}{2})}  \|A^{\frac{5}2} e^{-(r+t) A} x\|_{E}^2 \frac{\ddn r}{r} \dd t
\\ & \lesssim \int_0^\infty  \int_0^\infty  (t+r)^{-3} r^{4(1-\frac{\theta}{2})}  \|A e^{-r A} x\|^2_E \frac{\ddn r}{r} \dd t
\\ &  = 2 \int_0^\infty   r^{2(1-\theta)}  \|A e^{-r A} x\|_{E}^2 \frac{\ddn r}{r}
 \simeq \|x\|_{D_A(\theta,2)}^2
\end{align*}
which gives the required estimate \eqref{eq:squarefuncreal}.

From the previous case and Theorem \ref{theorem:SMRsemigroup} we obtain stochastic maximal $L^p$-regularity for $p\in [2, \infty)$ in the case $q=2$. Thus using Propositions \ref{proposition:detcharacterization} and \ref{proposition:independenceH} to take $H=\R$, the mapping
$$S(G)(t) := \int_0^t A^{\frac12}e^{-(t-s) A} G(s)\dd W(s), \qquad t \in \R_+$$
is bounded from $L^p(\R_+;D_A(\theta,2))$ to $L^p(\R_+\times \Omega;D_A(\theta,2))$ for all $\theta\in (0,1)$ and $p\in [2, \infty)$. By \cite[1.10 and 1.18.4]{Tr78} one has
\[(L^{q}(\R_+;D_A(\theta-\varepsilon,2)),L^{q}(\R_+;D_A(\theta+\varepsilon,2)))_{\frac12,q} = L^{q}(\R_+;D_A(\theta,q))\]
for $\varepsilon\in (0,\min\{\theta,1-\theta\})$ and the same holds with $\R_+$ replaced by $\R_+\times \Omega$. It follows from \cite[Theorem 1.3.3]{Tr78} that $S$ is bounded from  $L^q(\R_+;D_A(\theta,q))$ into $L^q(\R_+\times \Omega;D_A(\theta,q))$. Applying
Propositions \ref{proposition:detcharacterization} and \ref{proposition:independenceH} once more to recover a general cylindrical Brownian motion $W_H$, we obtain the stochastic maximal regularity for $p=q\in [2, \infty)$. Now another application of Theorem \ref{theorem:SMRsemigroup} gives the result for all required $p$, $q$ and weights $w\in A_{p/2}$. The claimed norm estimate again follows since $A^{1/2}$ maps $D_A(\theta+\frac12,q)$ isomorphically to $D_A(\theta,q)$ (see \cite[Theorem 1.15.2]{Tr78}).
\end{proof}

\begin{remark}\label{remark:realint}~
\begin{enumerate}[(i)]
  \item By carefully checking the proofs of Theorems \ref{theorem:SMRsemigroup} and \ref{theorem:realinterpMR} (and in particular Proposition \ref{proposition:detcharacterization}) one sees that Theorem \ref{theorem:realinterpMR} actually holds for all martingale type $2$ spaces $E$. As mentioned in Remark \ref{remark:severalthingonSMR}\ref{it:finitetimeSMR}, Theorem \ref{theorem:realinterpMR} holds on finite time intervals as well and in this case we only need that $A+\lambda$ is sectorial of angle $<\pi/2$ for some $\lambda\in \R$.
  \item Theorem \ref{theorem:realinterpMR} extends \cite[Theorem 5.1]{BH09} and \cite{DL98} to the case where $p\neq q$ and to the weighted setting.  Note that even for $w=1$ one cannot obtain Theorem \ref{theorem:realinterpMR} from the case $p=q$ and a real interpolation argument. Indeed, in general for an interpolation couple $(X_0, X_1)$ one has (see \cite{Cw74})
\[(L^{p_0}(\Omega\times \R_+;X_0),L^{p_1}(S;X_1))_{\theta,q} \neq L^{p_{\theta}}(\Omega\times\R_+;(X_0, X_1)_{\theta,q}),\]
with $\frac{1}{p_{\theta}} = \frac{1-\theta}{p_0} + \frac{\theta}{p_1}$. The equality does hold if $q=p_{\theta}$.

\item The assumption $0\in \rho(A)$ in Theorem \ref{theorem:realinterpMR} is needed in general. Indeed, there exists a bounded sectorial operator $A$ on a Hilbert space $E$ such that \eqref{eq:squarefuncreal} does not hold (see \cite[Corollary 10.2.29 and Theorem 10.4.21]{HNVW16}). Since in this case $D_A(\theta,p) = E$ for all $\theta\in (0,2)$ and $p\in [1, \infty]$, Propositions \ref{proposition:detcharacterization}, \ref{proposition:independenceH} and \ref{proposition:Kconvolutionnecessary} imply that \eqref{eq:realintest} cannot hold.
\end{enumerate}
\end{remark}

We end this subsection with another result for real interpolation spaces. It extends \cite[(4.10)]{Br95} to the case $p\in (2, \infty)$ and to the setting of infinite time intervals.
\begin{theorem}
Let $E$ be a $\UMD$ Banach space with type $2$ and let $A$ be sectorial of angle $<\pi/2$ on $E$. Let $X = D_A(\frac12, 2)$ and $Y = \dot{D}(A)$. Then for all $p\in (2, \infty)$ and $w\in A_{p/2}$, one has
$A\in \SMR(L^p(w,Y))$ (the case $p=2$ and $w=1$ is allowed as well). In particular, the solution $u$ to \eqref{eq:SEEGdW} satisfies
\begin{align*}
\|A u\|_{L^p(\Omega\times \R_+,w;E)} \leq C\, [w]_{A_{p/2}}^{\max\{ \frac{1}{2}, \frac{1}{p-2}\}}\|G\|_{L^{p}(\Omega\times \R_+,w;\gamma(H,D_A(\frac12,2)))},
\end{align*}
where $C$ only depends on $E,A,p$.
\end{theorem}
\begin{proof}
Note that as in the first proof of Theorem \ref{theorem:realinterpMR}, $A$ is sectorial of angle $<\pi/2$ on the space $X$.
For $p=2$, as in the second proof of Theorem \ref{theorem:realinterpMR}, it suffices to prove the following variant of \eqref{eq:squarefuncreal}
\begin{align*}
\Big(\int_0^\infty \|A e^{-t A} x\|_{E}^2\dd t\Big)^{\frac12}  \leq C \|x\|_{D_A(\frac12,2)}.
\end{align*}
The latter estimate is immediate from the definition of $D_A(\frac12,2)$. It remains to apply Theorem \ref{theorem:SMRsemigroup}. For this (see Remark \ref{remark:severalthingonSMR}\ref{it:admissibleY}) it suffices to check $\|e^{-tA}\|_{\mc{L}(X,Y)}\leq Ct^{-\frac12}$, which follows from
\[\sup_{t>0}\|t^{\frac12} A e^{-tA}x\|_E  \leq \|x\|_{D_A(\frac12,\infty)}\lesssim \|x\|_{D_A(\frac12, 2)  },\]
where we used \cite[Theorems 1.3.3(d) and 1.14.5]{Tr78}.
\end{proof}

\subsection{Stochastic heat equation on $\R^d$}
Next we continue with the stochastic heat equation. We will show that using only extrapolation results for stochastic singular integrals one can deduce the stochastic maximal $L^p(L^q)$-regularity results in \cite{Kr00} and \cite{NVW12}. Moreover we actually obtain results with weights in time. One can check that the proof of \cite{NVW12} based on the boundedness of the $H^\infty$-calculus of $-\Delta$ actually also gives the result with weights in time, and moreover $A_q$-weights in spaces could be added as well. Still we find it illustrative to show in the example below that the $L^2(L^2)$-case can be combined with extrapolation arguments to deduce the weighted $L^p(L^q)$-case for all $p\in (2, \infty)$ and $q\in [2, \infty)$.
For details on Bessel potential spaces we refer to \cite{Tr78}.
\begin{example}[Stochastic heat equation in Bessel-potential spaces]\label{ex:BesselHeat}
Let $m\in \N$, $s\in \R$, $p,q\in (2, \infty)$ and $w\in A_{p/2}$ (or $p=q=2$, $w \equiv 1$). On $\R^d$ consider
\begin{equation}\label{eq:SEEGdWheat}
\begin{cases}
du + (-\Delta)^m u\dd t = G \dd W_H, & \text{on $\R_+$},\\
u(0)=0,
\end{cases}
\end{equation}
where $G\in L^{p}_{\ms{F}}(\Omega\times \R_+,w;H^{s,q}(\R^d;H))$. Then the mild solution $u$ to \eqref{eq:SEEGdWheat} satisfies
\begin{align*}
\|(-\Delta)^{\frac{m}2}u\|_{L^p(\Omega\times \R_+,w;H^{s,q}(\R^d))} \leq C\, [w]_{A_{p/2}}^{\max\{ \frac{1}{2}, \frac{1}{p-2}\}}\|G\|_{L^{p}(\Omega\times \R_+,w;H^{s,q}(\R^d;H))},
\end{align*}
where $C$ only depends on $p,q,d$.
\end{example}
\begin{proof}
By lifting we may assume $s=0$ (see \cite[Theorems 2.3.2-2.3.4]{Tr78}). First suppose $q=2$. It suffices to check Corollary \ref{corollary:HScaseSMR}\ref{it:HScaseSMR2}. Note for any $f\in L^2(\R^d)$ by Plancherel's theorem
\begin{align*}
\int_{\R_+} \int_{\R^d}|(-\Delta)^{\frac{m}2} e^{t\Delta} f(x)|^2 \dd x \dd t & = \int_{\R_+} \int_{\R^d} (2\pi|\xi|)^{2m} e^{-2(2\pi|\xi|)^{2m} t} |\widehat{f}(\xi)|^2 \dd \xi \dd t
\\ & = \int_{\R^d} \int_{\R_+}(2\pi|\xi|)^{2m} e^{-2(2\pi|\xi|)^{2m} t} |\widehat{f}(\xi)|^2 \dd t \dd \xi
\\ & = \frac12 \int_{\R^d} |\widehat{f}(\xi)|^2 \dd \xi = \frac12 \|f\|_{L^2(\R^d)}^2.
\end{align*}
Therefore by Corollary \ref{corollary:HScaseSMR} we find the desired result for $q=2$.

From the extrapolation theorem \cite[Theorem 5.2 and Example 5.4]{KK16} we obtain the result for $w=1$ and $p=q\in (2, \infty)$. An application of Theorem \ref{theorem:SMRsemigroup} gives the required estimate for all $p,q\in (2, \infty)$ and $w\in A_{p/2}$.
\end{proof}

Next we prove a similar result on Besov spaces. For details on Besov space we refer to \cite{Tr78}.
\begin{example}[Stochastic heat equation in Besov spaces]\label{ex:BesovHeat}
Let $m\in \N$, $r\in [2,\infty)$, $s\in \R$, $p,q\in (2, \infty)$ and $w\in A_{p/2}$ (or $p=q=2$, $w\equiv 1$). On $\R^d$ consider
\begin{equation}\label{eq:SEEGdWheatalmost}
\begin{cases}
du +(1-\Delta)^m u\dd t = G \dd W_H, & \text{on $\R_+$},\\
u(0)=0,
\end{cases}
\end{equation}
where $G\in L^{p}_{\ms{F}}(\Omega\times \R_+,w;B^{s}_{r,q}(\R^d;H))$.
Then the mild solution $u$ to \eqref{eq:SEEGdWheatalmost} satisfies
\begin{align*}
\|(1-\Delta)^{\frac{m}2}u\|_{L^p(\Omega\times \R_+,w;B^{s}_{r,q}(\R^d))} \leq C\, [w]_{A_{p/2}}^{\max\{ 1, \frac{1}{p-2}\}}\|G\|_{L^{p}(\Omega\times \R_+,w;B^{s}_{r,q}(\R^d;H))},
\end{align*}
where $C$ only depends on $p$, $q$, $r$, $s$, $d$.
\end{example}

\begin{proof}
Again by lifting (see \cite[Theorem 2.3.4]{Tr78}) we may assume $s := 2m\theta\in (0,2m)$. Let $E = L^r(\R^d)$ and define
\begin{equation*}
  (A,D(A)) := ((1-\Delta)^m, W^{2m,r}(\R^d)).
\end{equation*}
 Then $A$ is sectorial of angle $0$ and $0\in \rho(A)$ on $E$. Since $D_A(\theta,q) = B^{s}_{r,q}$ (see \cite[Remark 2.4.2.4]{Tr78}) the result follows from Theorem \ref{theorem:realinterpMR}.
\end{proof}

\begin{remark}~
\begin{enumerate}[(i)]
\item There is an inconsistency between the equations \eqref{eq:SEEGdWheat} and \eqref{eq:SEEGdWheatalmost} ($-\Delta$ vs. $1-\Delta$). The reason to consider $1-\Delta$ is that one has the restriction $0\in \rho(A)$ in Theorem \ref{theorem:realinterpMR}. With a different proof one can also consider Example \ref{ex:BesovHeat}  with $1-\Delta$ replaced by $-\Delta$. For example one can obtain this by a real interpolation argument in Example \ref{ex:BesselHeat}. To avoid adaptedness problems in the interpolation argument one can first consider deterministic $G$ and afterwards apply Proposition \ref{proposition:detcharacterization}.
\item The results of Examples \ref{ex:BesselHeat} and \ref{ex:BesovHeat} are incomparable except if $r=q=2$ (see \cite[Theorem 2.3.9]{Tr83}). A similar example could be proved for Triebel--Lizorkin spaces, by using \cite{NVW12} and the boundedness of the $H^\infty$-calculus of $1-\Delta$ on $F^{s}_{r,q}(\R^d)$, which can be proved as in \cite{HNVW17} with the Mihlin multiplier theorem \cite[Theorem 2.3.7]{Tr83}. We do not see a way to prove this using just extrapolation.
\end{enumerate}

\end{remark}

\subsection{Stochastic heat equation on a wedge}
Our next application is an $L^p(L^q)$-version of the stochastic maximal regularity result in \cite{CKLL18} for the stochastic heat equation on an angular domain. The deterministic setting was considered in \cite[Theorem 1.1]{So01} and later improved in \cite[Theorem 1.1]{Na01} and \cite[Corollary 5.2]{PS07}. At the moment it is unclear whether the Dirichlet Laplacian $-\Delta$ on an angular domain has a bounded $H^\infty$-calculus, and how to characterize $D((-\Delta)^{1/2})$ in terms of weighted Sobolev spaces. Therefore, we cannot apply \cite{NVW12} and instead we will use \cite{CKLL18} and extrapolation theory to derive $L^p(L^q)$-regularity results.
\begin{example}\label{ex:wedge}
Let $\kappa\in (0,2\pi)$.
On the wedge
$$D:=\{x\in \R^2: x = (r\cos(\varphi), r\sin(\varphi)), r>0, \varphi\in (0,\kappa)\}$$ consider the stochastic heat equation:
\begin{equation}\label{eq:SEEGdWheatwedge}
\begin{cases}
du -\Delta u\dd t = G \dd W_H, & \text{on $\R_+$},\\
u(0)=0.
\end{cases}
\end{equation}
Let $q\in [2,\infty)$ and assume $\theta$ is such that
\begin{align}\label{eq:condtheta}
(1-\frac{\pi}{\kappa})q <\theta<(1+\frac{\pi}{\kappa})q.
\end{align}
Then for all $p\in (2, \infty)$ and $w\in A_{p/2}$ (where $p=2$ and $w=1$ if $q=2$ is allowed as well) the mild solution $u$ to \eqref{eq:SEEGdWheatwedge} satisfies
\begin{equation}
\label{eq:maxregheatwedge}
\begin{aligned}
\|u\|_{L^p(\Omega\times \R_+,w;\dot{W}^{1,q}(D,\abs{\cdot}^{\theta-2}))} & \leq C \|G\|_{L^{p}(\Omega\times \R_+,w;L^{q}(D,\abs{\cdot}^{\theta-2};H))}
\\ \|u\|_{L^p(\Omega\times \R_+,w;L^{q}(D,\abs{\cdot}^{\theta-2-q}))} & \leq C \|G\|_{L^{p}(\Omega\times \R_+,w;L^{q}(D,\abs{\cdot}^{\theta-2};H))},
\end{aligned}
\end{equation}
where $C$ only depends on $d,p,q,[w]_{A_{p/2}},\theta$ and $\kappa$. Here $\dot{W}^{1,q}(D,\abs{\cdot}^{\theta-2})$ denotes the usual homogenous Sobolev space of distributions  $u$ such that $\partial_j u\in L^q(D,\abs{\cdot}^{\theta-2})$.
\end{example}
\begin{proof}
In \cite{CKLL18} \eqref{eq:maxregheatwedge} was proved for $p=q$ and $w=1$, where it was stated for bounded intervals $(0,T)$. Since it holds with $T$-independent constants one can let $T\to \infty$ to find the result on $\R_+$. In order to prove the result for $p\neq q$ we will use Theorem \ref{theorem:SMRsemigroup} with
\begin{align*}
X &:= L^q(D,\abs{\cdot}^{\theta-2}))\\
  Y&:=\dot{W}^{1,q}(D,\abs{\cdot}^{\theta-2})\cap L^q(D,\abs{\cdot}^{\theta-2-q}).
\end{align*}
By Proposition \ref{proposition:propertiesheatwedge} $-\Delta$ is sectorial of angle $<\pi/2$ and $\|e^{t \Delta}\|_{\mc{L}(X,Y)}\leq C t^{-1/2}$ for $t>0$, so that $Y$ is allowed in Theorem \ref{theorem:SMRsemigroup} (see Remark \ref{remark:severalthingonSMR}\ref{it:admissibleY}), and hence the result follows.
\end{proof}

\subsection{Non-autonomous case with time-dependent domains}\label{subsection:nonauto}
In this subsection we prove extrapolation results under the conditions introduced by Acquistapace and Terreni \cite{AT87} (see also \cite{Ac88, AT92, Am95, Sc04, Ta97} and references therein). In the deterministic case extrapolation of maximal $L^p$-regularity was proved in \cite{CF14, CK18} under the Acquistapace--Terreni conditions and the Kato--Tanabe condition. Here the authors consider maximal $L^p$-regularity on $\R$ and $\R_+$ respectively. Below we prefer to consider maximal regularity results on finite intervals $(0,T)$ in order to avoid exponential stability assumptions. This is possible due to Section \ref{section:homogeneoustype} and a version of this theory could also be applied in the deterministic setting.

\hypertarget{asmp:AT}{
Fix $T\in (0,\infty)$. Next we introduce the (AT)-conditions due to Acquistapace and Terreni on a family of closed operators $(A(t))_{t\in [0,T]}$ on a Banach space $X$. Let us write $A_w(t) = A(t) + w$. We start with a uniform sectoriality condition:}
\setlist[enumerate]{leftmargin=40pt}
\begin{enumerate}[({AT}1)]
  \item\label{asmp:AT1} There exists a $\vartheta\in (0,\pi/2)$, $w\geq 0$ and $M>0$ such that for every $t\in [0,T]$, one has $\sigma(A_w(t)) \subseteq \Sigma_\vartheta $ and
  \[\|R(\lambda, A_w(t))\|_{\mc{L}(X)}\leq \frac{M}{|\lambda|+1}, \ \ \lambda \in \Sigma_\vartheta^c ,\]
      where
    \begin{equation*}
      \Sigma_{\vartheta} := \cbraceb{\lambda \in \C \setminus \cbrace{0}: \abs{\arg \lambda}<\vartheta}.
    \end{equation*}
\end{enumerate}
The next condition is a H\"older continuity assumption, which depends on the change of the domains $D(A(t))$.
\begin{enumerate}[({AT}1)]
\setcounter{enumi}{1}
  \item\label{asmp:AT2} There exist $0<\mu,\nu\leq 1$ with $\mu+\nu>1$ and $M \geq 0$ such that for all $s,t\in [0,\infty)$ and $\lambda \in \Sigma_{\vartheta}^c$,
    \begin{align*}
    |\lambda|^{\nu}\nrmb{A_w(t)R(\lambda,A_w(t))(A_w(t)^{-1}-A_w(s)^{-1})}_{\mc{L}(X)} \leq M\,|t-s|^{\mu}.
    \end{align*}
\end{enumerate}\setlist[enumerate]{leftmargin=25pt}
When $(A(t))_{t\in [0,T]}$ satisfies both \ref{asmp:AT1} and \ref{asmp:AT2} we say that it satisfies \hyperlink{asmp:AT}{(AT)}.

If the domains $D(A(t))$ all equal a fixed Banach space $X_1$ and
$$\nrm{A(t)-A(s)}_{\mc{L}(X_1,X)} \leq C \, \abs{t-s}^{\mu}$$
for some $\mu>0$, then $(A(t))_{t\in [0,T]}$ satisfies \ref{asmp:AT2} with $\nu=1$. Indeed, this follows directly from the equation  $A_w(t)^{-1} - A_w(s)^{-1} = A_w(t)^{-1}(A_w(s)-A_w(t))A_w(s)^{-1}$.

The following generation result is due to Acquistapace and Terreni (see \cite{Ac88, AT92, Sc04} for details).
We denote $D := \{(t,s) \in [0,T]\times [0,T]:\ t \geq s\}$.
\begin{proposition}[Evolution family]\label{proposition:exist-evol}
Assume \hyperlink{asmp:AT}{(AT)} for $(A(t))_{t \in [0,T]}$. There exists a unique strongly continuous map $S\colon D\to \mc{L}(X)$ such that
\begin{align*}
    S(t,t) &= I, &&t \in [0,T],\\
    S(t,s)S(s,r) &= S(t,r), &&t\geq s\geq r,\\
    \tfrac{d}{dt} S(t,s) &= A(t)S(t,s), && t>s.
  \end{align*}
Moreover for all $0\leq \alpha\leq 1$ there exists a constant $C>0$
such that
\begin{align*}
  \|A_w(t)^\alpha S(t,s)\|_{\mc{L}(X)} &\leq C\,(t-s)^{-\alpha}, \qquad (t,s) \in D.
\end{align*}
\end{proposition}

Given $S$ as in Proposition \ref{proposition:exist-evol}, we call $(S(t,s))_{t \geq s}$ the evolution family generated by $(A(t))_{t\in [0,T]}$. In order to state our extrapolation result we will need some notation. For $0<\alpha\leq 1$ and $t \in \R$ define
\begin{equation*}
  X_\alpha^t := D(A_w(t)^{\alpha}).
\end{equation*}
endowed with the graph norm. Moreover set $X_0^t = \overline{ D(A_w(t))}^{\nrm{\cdot}_X}$. Note that since $-w \in \rho(A(t))$ we have
\begin{equation*}
  \nrm{x}_{X_\alpha^t} \leq C \, \nrm{A_w(t)^{\alpha}x}_X , \qquad x \in D(A_w(t)^{\alpha}).
\end{equation*}

\begin{lemma}\label{lemma:holdercont}
Let $0<\alpha\leq 1$.
Let $(\widetilde{X}_{\beta})_{\beta\in [0,\alpha]}$ be an interpolation scale and assume for $\beta \in [0,\alpha]$ one has $X_{\beta}^t \hookrightarrow \widetilde{X}_{\beta}$ uniformly in $t \in \R$.
Then
\[\|S(t,s) - I\|_{\mc{L}(X_{\alpha}^s,\widetilde{X}_{\beta})}\leq C\,(t-s)^{\alpha-\beta},\qquad (t,s) \in D.\]
\end{lemma}
\begin{proof}
The result for $\beta =\alpha$ is clear from the assumption and \cite[(2.19)]{Sc04}. For $\beta = 0$, the result follows from \cite[(2.16)]{Sc04}. The result for $0<\alpha<\beta$ follows by interpolation.
\end{proof}

We can now prove our extrapolation theorem for $(A(t))_{t\in [0,T]}$ in the setting of Acquistapace and Terreni:
\begin{theorem}[Extrapolation in the evolution family case]\label{theorem:SMRevolution}\label{theorem:ATextrapolatie} Let $\alpha \in  (\tfrac{1}{2},1]$ and let $(\widetilde{X}_{\beta})_{\beta\in [0,\alpha]}$ be an interpolation scale.
Assume the following conditions:
\begin{itemize}
  \item Both $(A(t))_{t\in [0,T]}$  and $(A(t)^*)_{t\in [0,T]}$ satisfy \hyperlink{asmp:AT}{(AT)}.
  \item For $\beta \in [0,\alpha]$ one has $X_{\beta}^t \hookrightarrow \widetilde{X}_{\beta}$ uniformly in $t \in [0,T]$.
  \item $\widetilde{X}_{\frac12}$ is a UMD Banach space with type $2$
\end{itemize}
Suppose $A\in \SMR(L^p(0,T;\widetilde{X}_{\frac12}))$ for some $p\in [2, \infty)$. Then for all $q\in (2, \infty)$ and $w\in A_{q/2}$ one has $A\in \SMR(L^q((0,T),w;\widetilde{X}_{\frac12}))$.
\end{theorem}

\begin{proof}
Set $Y:= \widetilde{X}_{\frac12}$ and let $K\colon[0,T]^2\to \mc{L}(X,Y)$ be the kernel given by
$$K(t,s) = S(t,s) \ind_{t >s}.$$
Then by our assumptions, Propositions \ref{proposition:exist-evol},  \ref{proposition:detcharacterization} and \ref{proposition:independenceH} we know that $K \in \mc{K}_\gamma(L^p(0,T))$. Therefore by Theorem \ref{theorem:sparsedomination} it suffices to check the $(\epsilon,2)$-standard kernel conditions for $K$.

To do so take $t>s$ and note that by Proposition \ref{proposition:exist-evol} for $0\leq s<t\leq T$,
\[\|K(t,s)\|_{\mc{L}(X,Y)}\leq C\,\|A(t)^{\frac12} S(t,s)\|_{\mc{L}(X)} \leq C(t-s)^{-1/2}.\]

To check \eqref{eq:dini2} on $[0,T]$ let $\alpha \in (\frac12,1]$ be such that the conclusion of Lemma \ref{lemma:holdercont} holds and take $|t-t'|\leq \frac12|t-s|$. If $t<s$, then also $t'<s$ and there is nothing to prove. Thus it suffices to consider the case $t,t'>s$. If $t'>t$, then
\begin{equation}
\label{eq:berekeningKdiff}
\begin{aligned}
 \|K(t',s) - K(t,s)\|_{\mc{L}(X,Y)}& = \|K(t',t) - I\|_{\mc{L}(X_{\alpha}^{t},Y)} \|K(t,s)\|_{\mc{L}(X,X_{\alpha}^{t})}
\\ & \leq C \,(t'-t)^{\alpha-\frac12} (t-s)^{-\alpha}
\\ & = C \, \Big|\frac{t-t'}{t-s}\Big|^{\alpha-\frac12} |t-s|^{-1/2},
\end{aligned}
\end{equation}
where we used Lemma \ref{lemma:holdercont} and Proposition \ref{proposition:exist-evol}.
In the case $t>t'$ the same estimate holds with $t$ and $t'$ interchanged. Since
$t'-s\geq \tfrac12(t-s)$, \eqref{eq:dini2} also follows in this case.

Next we check \eqref{eq:dini3}. By \cite[Theorem 6.4]{AT92} we have for $0\leq s<t\leq T$
\begin{align}\label{eq:dualestAT}
\|S(t,s) A(s)\|_{\mc{L}(X)}\leq C(t-s)^{-1}.
\end{align}
Therefore using Proposition \ref{proposition:exist-evol}  we have
\begin{align*}
\nrmb{\tfrac{d}{ds} K(t,s)}_{\mc{L}(X,Y)} & = \|S(t,s)A(s)\|_{\mc{L}(X,Y)}\\ & \leq C\, \|A(t)^{1/2} S(t,s) A(s)\|_{\mc{L}(X)}
\\ & \leq C\,\|A(s)^{1/2} S(s,\tfrac{s+t}{2})\|_{\mc{L}(X)} \|S(\tfrac{s+t}{2},t) A(t)\|_{\mc{L}(X)}
\\ & \leq C\,{(s-t)^{-\frac32}}.
\end{align*}
As in the proof of Lemma \ref{lemma:standardkernelderivatives} we obtain that \eqref{eq:dini3} holds with $\omega(r) = C r$.
We can therefore conclude that $K$ is an $(\alpha-\frac12,2)$-standard kernel, which finishes the proof.
\end{proof}

\begin{remark}
  If $X$ and $\widetilde{X}_{\frac{1}{2}}$ are Hilbert spaces, the assumption that $A\in \SMR(L^2(0,T;\widetilde{X}_{\frac12}))$ in Theorem \ref{theorem:SMRevolution} can be checked by showing
\[\|t\mapsto S(t,s)x \ind_{t>s}\|_{L^2(0,T;\widetilde{X}_{\frac{1}{2}})}\lesssim \|x\|_X, \qquad x\in X, \quad s\in [0,T],\]
using Proposition \ref{proposition:simplesufficient}\ref{it:sufficientL2}. By the proof of \cite[Theorem 4.3]{Ve10} it is therefore sufficient to check
\[\|t\mapsto A_w(s)^{\frac12}e^{t A(s)}x\|_{L^2((0,T);X)}\lesssim \|x\|, \qquad  s\in [0,T],\quad x\in X.\]
\end{remark}

As an application we deduce stochastic maximal $L^p$-regularity for an operator family which was previously considered in \cite{Ac88,Sc04,Yag91} in the deterministic setting and in \cite{SV03} and \cite[Example 8.2]{Ve10} in the stochastic setting. In particular, stochastic maximal $L^2(L^2)$-regularity was derived in the latter. Below we extend this to an $L^p(L^q)$-setting.
\begin{example}[Stochastic heat equation on domains with time-dependent Neumann boundary condition]\label{example:domainneumann}
Let $\epsilon \in (0,\frac12)$ and $T \in (0,\infty)$. On a bounded $C^{3+\epsilon}$-domain $D\subseteq \R^d$ consider
\begin{equation}\label{eq:SEEGdWheattime}
\begin{cases}
du +A u\dd t = G \dd W_H, & \text{on $[0,T]\times D$},
\\
C u = 0 & \text{on $[0,T]\times \partial D$},
\\ u(0)=0.
\end{cases}
\end{equation}
Here the differential operator $A$ and boundary operator are given by
\begin{align*}
A(t,x)u & = -\sum_{i,j=1}^d \partial_i a_{ij}(t,x) \partial_j u,
\\ C(t,x) u & = \sum_{ij=1}^d a_{i,j}(t,x) n_i(x) \partial_j u,
\end{align*}
where for $x\in \partial D$, $n(x)\in \R^d$ denotes the outer normal of $D$. Assume that the coefficients are real-valued and satisfy
\begin{align*}
a_{ij}\in C^{1+\frac{\epsilon}{2}, 2+\epsilon}([0,T]\times \overline{D}),
\end{align*}
for all $i,j\in \{1, \ldots, d\}$.

We further assume $(a_{ij})$ is symmetric  and that there exists a $\kappa>0$ such that
\[\sum_{i,j=1}^n a_{ij}(t,x) \xi_i\xi_j\geq \kappa|\xi|^2, \ \ x\in D, t\in [0,T], \xi\in \R^d.\]
Then for all $p,q\in (2, \infty)$ and $w\in A_{p/2}$ (where $p=q=2$ and $w=1$ is allowed as well) the mild solution $u$ to \eqref{eq:SEEGdWheattime} satisfies
\begin{align*}
\|u\|_{L^p(\Omega\times (0,T),w;W^{1,q}(D))} \leq C \|G\|_{L^{p}(\Omega\times (0,T),w;L^q(\R^d;H))},
\end{align*}
where $C$ does not depend on $G$.
\end{example}

\begin{remark}
  Example \ref{example:domainneumann} for $p=q=2$ and $w \equiv 1$ has been shown in \cite[Example 8.2]{Ve10} assuming only a $C^2$ domain and
  \begin{align*}
a_{ij}&\in C^{1/2+\epsilon}([0,T];C(\overline{D})), \\
a_{ij}(t,\cdot) &\in C^1(\overline{D}), \\
\partial_k a_{ij}&\in C([0,T]\times\overline{D})
  \end{align*}
for all $i,j,k\in \{1, \ldots, d\}$ and $t\in [0,T]$ and some $\mu\in (\tfrac12,1]$. This in turn can be extrapolated to $p \in [2,\infty)$, $q=2$ and $w \in A_{p/2}$ as in Step 2 of the following proof. However, in this situation the kernel estimates in \cite[Theorem 1.1]{EI70} are not strong enough to check the parabolic H\"ormander condition needed in Step 1 of the following proof. Therefore only $L^p(L^2)$ theory can be obtained in this setting using the kernel estimates available in literature.
\end{remark}

\begin{proof}[Proof of Example \ref{example:domainneumann}]
In \cite[Example 8.2]{Ve10} the result has been shown for $p=q=2$ and $w \equiv 1$, we will use extrapolation techniques to deduce the general case. For this note that in \cite{Ac88,Sc04,Yag91} it is shown that for $q \in (1,\infty)$ the realization of $(A(t))_{t \in [0,T]}$ on $L^q(D)$ with domain
\begin{equation*}
  D(A(t)) := \cbraceb{u \in W^{2,q}(D):\Tr_{\partial D}(C(t,\cdot)u)=0}
\end{equation*}
satisfies \hyperlink{asmp:AT}{(AT)}.

\textbf{Step 1:} We will first use \cite[Theorem 2.5]{KK16} to deduce the result for $p=q$ and $w\equiv1$. For this let $\Gamma$ denote the Green kernel of the evolution family associated to the realization of $A$ on $L^2(D)$, which exists by \cite[Theorem 1.1]{EI70}. Then the mild
solution $u$ to \eqref{eq:SEEGdWheattime} is given by
\begin{equation*}
  u(t,x) = \int_0^t \int_D \Gamma(t,s,x,y) G(s,y)\dd y \dd W_H(s), \qquad (t,x) \in (0,T)\times D,
\end{equation*}
For $\abs{\alpha}= 1$ define
\begin{equation*}
  K_{\alpha}(t,s,x,y) = \partial_x^\alpha \Gamma(t,s,x,y), \qquad t,s \in (0,T),\quad  x,y \in D.
\end{equation*}
Then by \cite[Theorem 1.1]{EI70} we have for all $\abs{\beta}=1$, $t>s$ and $x,y \in D$
\begin{align*}
  \abs{\partial^\beta_x K_{\alpha}(t,s,x,y)} &\leq C\, \frac{1}{\ha{t-s}^{(\abs{\alpha}+d+1)/2}} \exp \has{-c\,\frac{\abs{x-y}}{(t-s)^{1/2}}}\\
  \abs{\partial_t  K_{\alpha}(t,s,x,y)} &\leq C\, \frac{1}{(t-s)^{(d+3)/2}} \exp \has{-c\,\frac{\abs{x-y}}{(t-s)^{1/2}}}
\end{align*}
from which the assumption, and therefore the conclusion of Lemma \ref{lemma:21hormanderdini} follows using Lemma \ref{lemma:checktechnical}. If we extend $K_\alpha$ by zero for $t,s \geq T$ the same conclusion holds on $\R_+ \times D$, which has infinite measure.
Combined with the case  $p=q=2$ from \cite[Example 8.2]{Ve10}  we have checked the assumptions of \cite[Theorem 2.5]{KK16} for the operators
$$
T_\alpha \colon L^{2}(\Omega\times \R_+;L^2(\R^d;H)) \to L^2(\Omega\times \R_+;W^{1,2}(D))
$$
given by
\begin{equation*}
  T_\alpha G(t,x) := \int_0^t \int_D K_\alpha(t,s,x,y) G(s,y)\dd y \dd W_H(s), \quad (t,x) \in \R_+\times D,
\end{equation*}
for all $\abs{\alpha} = 1$ and thus the result for $p=q$ and $w\equiv1$ follows.

\textbf{Step 2:} For the general case let $\widetilde{X}_{\beta} = W^{2\beta,q}(D)$ for $\beta\in (0,1]$ and $\widetilde{X}_{0} = L^q(D)$. Then $X_\beta^t \hookrightarrow \widetilde{X}_{\beta}$ for all $\beta \in [0,1]$ (see \cite[Example 2.8]{Sc04}) and by Step 1 we have $A\in \SMR(L^q(0,T;\widetilde{X}_{\frac12}))$. Therefore the result in the general case follows from Theorem \ref{theorem:ATextrapolatie}.
\end{proof}

\subsection{Volterra equations}

In \cite{DL13} the results of \cite{NVW12} have been extended to the setting of integral equations: \begin{align*}
U(t) +  A\int_0^t \frac{1}{\Gamma(\alpha)} (t-s)^{\alpha-1} U(s) \dd s =  \int_0^t \frac{1}{\Gamma(\beta)} (t-s)^{\beta-1} U(s) \dd W_H(s),
\end{align*}
where $\alpha\in (0,2)$ and, $\beta\in (\tfrac12,2)$.  The solution $U$ is given by
\begin{align*}
U(t) = \int_0^t S_{\alpha \beta}(t-s) G(s) \dd  W_H(s), \qquad t \in \R_+,
\end{align*}
where $S_{\alpha \beta}$ is the so-called resolvent associated with $A$, $\alpha$ and $\beta$.
The maximal regularity result in \cite[Theorem 3.1]{DL13} gives $L^p$-estimates for $A^{\theta} \partial^\eta_t U$ in terms of $G$, where $\beta-\alpha\theta-\eta = \frac12$ with $\theta\in [0,1)$ and $\eta\in (-1,1)$. In this case one has to estimate a stochastic convolution with kernel $k(t) = A^{\theta} \partial^\eta_t S_{\alpha\beta}(t)$. We will not go into details on Volterra equations further now, but restrict ourselves to checking that $K(s,t):=k(s-t)$ is an $(\epsilon, 2)$-standard kernel for suitable $\epsilon\in (0,\tfrac12)$. Consequently our extrapolation theorem can be applied to this setting as well.

First consider $\eta\in(-\tfrac12,1)$. Choose $\epsilon\in (0,\tfrac12)$ such that $\eta+\epsilon\in (0,1)$, then there is an $M>0$ such that (see \cite[Remark 2.4]{DL13})
\begin{align}\label{eq:condkernelvolterra}
\|\partial^{\epsilon} k(t)\|\leq M\, t^{-\epsilon-\frac12}.
\end{align}
Writing $K(t) = \partial^{-\epsilon} \partial^{\epsilon}  k(t)$, it follows from Lemma \ref{lemma:standardkernelfractional} that $K$ is an $(\epsilon, 2)$-standard kernel.

If $\eta\in(-1,-\tfrac12)$, we let $\epsilon = -\eta$. Then $k(t) = \partial^{-\epsilon}_t  A^{\theta} S_{\alpha\beta}(t)$ and there is an $M>0$ such that (see \cite[Remark 2.4]{DL13})
\[\|A^{\theta} S_{\alpha\beta}(t)\|\leq M t^{-\frac12-\epsilon}\]
Therefore, $K$ is an $(\epsilon, 2)$-standard kernel.

\section{\texorpdfstring{$p$}{p}-Independence of the \texorpdfstring{$\mc{R}$}{R}-boundedness of convolutions}\label{section:Rbddness}
In this final section we prove the $p$-independence of a Banach space property which was introduced in \cite{NVW15}. In order to state the condition we need to introduce the notion $\mc{R}$-boundedness of a family of operators.

For details on $\mc{R}$-boundedness we refer to \cite[Chapter 8]{HNVW17}. For us it will be enough to recall the definition. Let $X$ and $Y$ be Banach spaces and let $(\varepsilon_j)_{j\geq 1}$ be a Rademacher sequence on a probability sequence $(\Omega, \mathcal{A},\P)$. A family of operators $\mc{T}\subseteq \mc{L}(X,Y)$ is called {\em $\mc{R}$-bounded} if there exists a constant $C$ such that for all $T_1, \ldots, T_n\in \mathcal{T}$ and $x_1, \ldots, x_n\in X$ one has
\begin{align*}
\Big\|\sum_{j=1}^n \varepsilon_j T_j x_j\Big\|_{L^2(\Omega;X)}\leq C\Big\|\sum_{j=1}^n \varepsilon_jx_j\Big\|_{L^2(\Omega;X)}.
\end{align*}

Let $X$ be a Banach space with type $2$.
For $\lambda\in \C$ with $\re(\lambda)>0$ let $k_{\lambda}:\R_+\to \C$ be given by
\[k_{\lambda}(s) = \lambda^{1/2} e^{-\lambda s}, \qquad s \in \R_+,\]
and define $T_{\lambda}:L^p(\R_+;X)\to L^p(\R_+;\gamma(\R_+;X))$ by
\[T_{\lambda} f(s) = k_{\lambda}(s-\cdot) f(\cdot), \qquad s\in \R_+.\]
Then by Proposition \ref{proposition:scalarkernels}
\begin{equation}\label{eq:klambdanorm}
\nrm{k}_{\mathcal{K}_\gamma(L^p(\R_+))}\leq \tau_{2,X} \Big(\frac{|\lambda|}{2\re(\lambda)}\Big)^{1/2}.
\end{equation}
The following $p$-dependent condition was introduced in \cite{NVW15,NVW15b}.
\setlist[enumerate]{leftmargin=40pt}\hypertarget{asmp:Cp}{
\begin{enumerate}[($C_p$)]
\item \label{it:Cp} For each $\theta\in [0,\pi/2)$ the family $\mathcal{T} = \{T_{\lambda}: |\arg(\lambda)|\leq \theta\}$ is $\mc{R}$-bounded
from $L^{p}(\R_+;X)$ into $L^{p}(\R_+;\gamma(\R_+;X))$.
\end{enumerate}}
\setlist[enumerate]{leftmargin=25pt}

Note that \eqref{eq:klambdanorm} implies that $\mathcal{T}$ is uniformly bounded.
In \cite{NVW15b} the condition \ref{it:Cp} was combined with the boundedness of the $H^\infty$-calculus in order to derive stochastic maximal $L^p$-regularity (see Definition \ref{definition:SMR}).

From  \cite[Theorems 4.7 and 7.1]{NVW15} it can be seen
that in the following case the condition \ref{it:Cp} holds for all $p\in (2, \infty)$:
\begin{itemize}
\item $X$ is a $2$-convex Banach function space and the dual of its concavification $X^2$ is an $\HL$-space, i.e. the lattice Hardy--Littlewood maximal operator is bounded on $L^p(\R^d;(X^2)^*)$ for some (all) $p \in (1,\infty)$.
\end{itemize}
In particular, UMD Banach function spaces are HL-spaces, but also $L^\infty$ is an $\HL$-space. In particular, the space $L^q$ satisfies \ref{it:Cp} for any $q\in [2, \infty)$ and $p\in (2, \infty)$. In the case $q=2$ one can additionally allow $p=2$. On the other hand, $L^q$ for $q>2$ fails \hyperlink{asmp:Cp}{$(C_2)$} (see \cite[Theorem 6.1]{NVW12} and the proof of \cite[Theorem 7.1]{NVW15b}). A Banach function space with UMD and type $2$ for which we do not know whether \ref{it:Cp} holds for $p\in (2, \infty)$ is for instance $\ell^2(\ell^4)$. Some evidence against this can be found in \cite[Theorem 8.2]{NVW15}.

It was an open problem whether \ref{it:Cp} is $p$-independent. Below we settle this issue. In the special case of Banach function spaces one could also derive this by rewriting \ref{it:Cp} as a square function result (see \cite[Theorem 7.1]{NVW15}) and using operator-valued Calder\'on--Zygmund theory (see \cite{GR85}).
\begin{theorem}
Let $X$ be Banach space with type $2$  and let $p\in [2, \infty)$. If \ref{it:Cp} holds,
then for all $\theta\in [0,\pi/2)$, $q\in (2, \infty)$ and  $w\in A_{\frac{q}{2}}(\R_+)$ the family $$\mathcal{T} = \{T_{\lambda}: |\arg(\lambda)|\leq \theta\}$$ is $\mc{R}$-bounded
from $L^{q}(\R_+,w;X)$ into $L^{q}(\R_+,w;\gamma(\R_+;X))$.
In particular \hyperlink{asmp:Cp}{$(C_q)$} holds for all $q\in (2, \infty)$.
\end{theorem}

\begin{proof}
Fix $n\in \N$. Let $\lambda_1, \ldots, \lambda_n\in \Sigma_{\theta}$ and $f_1, \ldots, f_n\in L^{q}(\R_+,w;\gamma(\R_+;X))$. Let $\Rad_n(X)$ be the space $X^n$ endowed with the norm
\[\|(x_j)_{j=1}^n\|_{\Rad_n(X)} := \Big\|\sum_{j=1}^n \varepsilon_jx_j\Big\|_{L^2(\Omega;X)},\]
where $(\varepsilon_j)_{j=1}^n$ is a Rademacher sequence. Replacing the $L^2(\Omega;X)$-norm by $L^r(\Omega;X)$ with $r\in [1, \infty)$ leads to an equivalent norm by the Kahane--Khintchine inequalities (see \cite[Theorem 6.2.4]{HNVW17}).
Define a diagonal operator $k:\R_+\to \mc{L}(\Rad_n(X))$ by
\[(k(s)x)_j = k_{\lambda_j}(s) x_j, \qquad j\in \{1, \ldots, n\}, \quad x \in \Rad_n(X),\]
and set $K(s,t) := k(s-t)$. To prove the required $\mc{R}$-boundedness of $\mathcal{T}$, by the Kahane--Khintchine inequalities, Fubini's theorem and Proposition \ref{proposition:gammaFubLorentz} it suffices to prove that
$\nrm{K}_{\mathcal{K}_\gamma(L^q(\R_+,w))}\leq C$ where $C$ is independent of $n$. Now by \ref{it:Cp} we know the latter is true for $w=1$ and $q=p$. Therefore, by Theorem \ref{theorem:sparsedomination} (see also Section \ref{section:homogeneoustype}) it suffices to check that $K$ satisfies the required Dini condition with constants only depending on $\theta$. For this we check the condition of Lemma \ref{lemma:standardkernelderivatives}. Moreover, since $K$ is of convolution type it suffices to check that $\|K'(s)\|\leq C s^{-3/2}$. Since $k'(s)$ is a diagonal operator we have for $x \in \Rad_n(X)$:
\begin{align*}
\|K'(s)x\|_{\Rad_n(X)} = \Big\|\sum_{j=1}^n \varepsilon_j k_{\lambda_j}'(s) x_j\Big\|_{L^2(\Omega;X)} \leq C s^{-3/2} \Big\|\sum_{j=1}^n \varepsilon_j  x_j\Big\|_{L^2(\Omega;X)},
\end{align*}
where we used the Kahane contraction principle (see \cite[Proposition 6.1.13]{HNVW17}) together with \begin{align*}
\abs{s^{3} k_{\lambda_j}'(s)^2} \leq \sup_{\lambda\in \Sigma_{\theta}} |\lambda|^3 e^{-2\re(\lambda)} \leq \frac{27}{8e^3 \cos^3(\theta)} := C^2.
\end{align*}
This implies the required estimates for $K$ and therefore finishes the proof.
\end{proof}

\begin{remark}
One could replace $k_{\lambda}$ by a class of functions which satisfies the $(\omega,2)$-Dini condition uniformly for one fixed function $\omega$. Moreover, a similar result holds on other spaces of homogenous type.
\end{remark}

\appendix

\section{Technical estimates}\label{section:technical}

\subsection*{Heat kernel estimates on a wedge}
\begin{lemma}\label{lemma:techincalkest}
Assume $\kappa\in (0,2\pi)$, $q\in (1, \infty)$, $\sigma>0$ and $\theta\in \R$.
Let $\mu = \frac{\pi}{\kappa}\in (\tfrac12,\infty)$. For $j\in \{0,1,2\}$ let $k_t^{j}:\R^2\to [0,\infty)$ be defined by
\begin{align*}
k_t^{j}(x,y) = \zeta^{\mu-j}(t,x) \zeta^{\mu}(t,y) t^{-1}\exp(-\sigma|x-y|^2/t),
\end{align*}
where $\zeta(t,x) = \frac{|x|}{|x|+\sqrt{t}}$. For $j-\mu<\frac{\theta}{q}<2+\mu$ one has
\begin{align*}
\sup_{t>0,y\in \R^2 }\int_{\R^2} k_t^{j}(x,y) |x|^{\frac{\theta}{q}} |y|^{2-\frac{\theta}{q}} \frac{{\rm d} x}{|x|^2}&<\infty,
\\ \sup_{t>0,x\in \R^2 }\int_{\R^2} k_t^{j}(x,y) |x|^{\frac{\theta}{q}} |y|^{2-\frac{\theta}{q}} \frac{{\rm d} y}{|y|^2}&<\infty.
\end{align*}
\end{lemma}
\begin{proof}
By a substitution replacing $x$ and $y$ by $x \sqrt{t}$ and $y \sqrt{t}$, one can check that it suffices to consider $t=1$, and we set $k^j(x,y) = k_1^j(x,y)$.
It suffices to consider $\sigma\in (0,1]$. Moreover, since $\zeta(1,x) \leq \zeta(1,x/\sigma) \leq \frac{1}{\sigma}\zeta(1,x)$, by a substitution one can reduce to $\sigma=1$. Let $a = \mu+2-\frac{\theta}{q}$. Then $a>0$ by the assumptions in the lemma, and a simple rewriting shows that
\begin{align*}
k^j(x,y) |x|^{\frac{\theta}{q}} |y|^{2-\frac{\theta}{q}}  = \frac{|x|^{2\mu+2-j-a}}{(|x|+1)^{\mu-j}} \frac{|y|^{a}}{(|y|+1)^{\mu}} e^{-|x-y|^2} .
\end{align*}

\textbf{Step 1:} First consider the integral with respect to $x$. One has
\begin{align*}
\int_{\R^2} k^j(x,y) |x|^{\frac{\theta}{q}} |y|^{2-\frac{\theta}{q}} \frac{{\rm d} x}{|x|^2} &= \int_{\R^2} \frac{|x|^{2\mu-j}}{(|x|+1)^{\mu-j}(|y|+1)^{\mu}}  \has{\frac{\abs{y}}{\abs{x}}}^{a} \ee^{-|x-y|^2}\dd x
\\ & = S_{1} + S_{2} + S_{3},
\end{align*}
where $S_{1}$ is the integral over $|x|\leq \tfrac{1}{2}|y|$,  $S_{2}$ is the integral over $\tfrac12|y|<|x|< \tfrac{3}{2}|y|$ and $S_{3}$ is the integral over $|x|\geq \tfrac{3}{2}|y|$.

For $S_1$ note that $|x-y|\geq |y| - |x| \geq \tfrac{1}{2} |y|$. Therefore, $e^{-|x-y|^2}\leq \ee^{-\frac14 |y|^2}$ and we find
\begin{align*}
S_1 & \leq |y|^a \ee^{-\frac14 |y|^2}  \int_{|x|\leq \frac12 |y|} |x|^{2\mu-j-a} (|x|+1)^{j-\mu} \dd x
\\ & \leq 2\pi (|y|+1)^{|j-\mu|+a} \ee^{-\frac14 |y|^2}  \int_{0}^{\frac12 |y|} r^{2\mu-j-a+1}   \dd r
\\ &\eqsim (|y|+1)^{|j-\mu|+a}  |y|^{2\mu-j-a+2} \ee^{-\frac14 |y|^2} \leq C,
\end{align*}
where we used $2\mu-j-a+2 = \mu-j+\frac{\theta}{q}>0$.

For $S_2$ if $|y|\leq1$, then
\begin{align*}
S_2 &\leq \int_{\frac12|y|<|x|< \frac{3}{2}|y|} \Big(\frac{|x|}{|x|+1}\Big)^{\mu-j} \Big(\frac{|x|}{|y|+1}\Big)^{\mu}  \dd x
\\ &  \eqsim \int_{\frac12|y|<|x|< \frac{3}{2}|y|} |x|^{2\mu-j} \dd x\eqsim |y|^{2\mu-j+1}\leq C,
\end{align*}
where we used $2\mu-j+1>2-j\geq 0$.
If $|y|>1$, then $S_1  \eqsim  \int_{\R^2} e^{-|x-y|^2} \dd x = C$.

For $S_3$, note that $|x-y|\geq |x|-|y|\geq \frac13 |x|$. Thus $e^{-|x-y|}\leq e^{-\frac19 |x|^2}$. Now if $|y|>1$, then
\begin{align*}
S_3 &  \leq |y|^{a-\mu} \int_{|x| > \frac32|y|} |x|^{\mu-a} \ee^{-\frac19 |x|^2} \dd x \\ & = 2\pi |y|^{a-\mu} \int_{\frac32|y|}^\infty r^{\mu-a+1} \ee^{-\frac19 r^2} \dd r
\\ & = 2\pi \int_{\frac32}^\infty s^{\mu-a+1} |y|^2 \ee^{-\frac19 |y|^2 s^2} \dd s
\\ & \leq  2\pi \,|y|^2 \ee^{-\frac1{36} |y|^2} \int_{\frac32}^\infty s^{\mu-a+1}\ee^{-\frac1{36} |s|^2}  \dd s \leq C,
\end{align*}
where we used $|y|s \geq \tfrac{1}{2}(|y|+s)$ for $s,y>1$.
If $|y|\leq 1$, then
\begin{align*}
S_3 & \leq \int_{|x| > \frac32|y|} |x|^{2\mu-j-a} (1+|x|)^{j-\mu} \ee^{-\frac19 |x|^2} \dd x
\\ & \leq 2\pi \int_{0}^\infty r^{2\mu-j-a+1} (1+r)^{\abs{\mu-j}} \ee^{-\frac19 r^2} \dd x <\infty,
\end{align*}
because $2\mu-j-a+2 = \mu-j+\frac{\theta}{q}>0$.

\textbf{Step 2:} Next consider the integral with respect to $y$. One has
\begin{align*}
\int_{\R^2} k(x,y) |x|^{\frac{\theta}{q}} |y|^{2-\frac{\theta}{q}} \frac{{\rm d} y}{|y|^2} &= \int_{\R^2} \frac{|x|^{2\mu-j}}{(|x|+1)^{\mu-j}(|y|+1)^{\mu} }  \has{\frac{\abs{y}}{\abs{x}}}^{a-2}\ee^{-|x-y|^2} \dd y
\\ & = T_{1} + T_{2} + T_{3},
\end{align*}
where $T_{1}$ is the integral over $|y|\leq \tfrac{1}{2}|x|$,  $T_{2}$ is the integral over $\tfrac12|x|<|y|< \tfrac{3}{2}|x|$ and $T_{3}$ is the integral over $|y|\geq \tfrac{3}{2}|x|$.

For $T_1$ note that $|x-y|\geq |x| - |y| \geq \tfrac{1}{2} |x|$. Therefore, $e^{-|x-y|^2}\leq e^{-\frac14 |x|^2}$ and we find
\begin{align*}
T_1 & \leq |x|^{2\mu+2-a-j} (|x|+1)^{|\mu-j|} \ee^{-\frac14 |x|^2}  \int_{|y|\leq \frac12 |x|} |y|^{a-2} \dd y
\\ & = 2\pi |x|^{2\mu+2-a-j} (|x|+1)^{|\mu-j|} \ee^{-\frac14 |x|^2}   \int_{0}^{\frac12 |x|} r^{a-1} \dd r
\\ & \eqsim |x|^{2\mu+2-j} (|x|+1)^{|\mu-j|} \ee^{-\frac14 |x|^2}   \leq C,
\end{align*}
where we used $a>0$ and $2\mu+2-j>0$.

For $T_2$ if $|x|\leq 1$ then we can write
\begin{align*}
T_2&\lesssim \Big(\frac{|x|}{|x|+1}\Big)^{2\mu-j} \int_{\frac12|x|<|y|< \frac{3}{2}|x|} \ee^{-|x-y|^2}\dd y \lesssim  |x|^{2\mu-j+2}  \leq C.
\end{align*}
where we used $2\mu+2-j> 0$. If $|x|\geq 1$, then
\[T_2\lesssim \int_{\R^2} e^{-|x-y|^2}\dd y = C.\]

For $T_3$, note that $|x-y|\geq |y|-|x|\geq \frac13 |y|$. Thus $e^{-|x-y|}\leq e^{-\frac19 |y|^2}$.
If $|x|>1$ we can write
\begin{align*}
T_3 & \lesssim \int_{|y|>\frac{3}{2}|x|} \Big(\frac{|y|}{|x|}\Big)^{a-2-\mu} e^{-\frac19|y|^2}\dd y
\\ & = 2\pi\int_{\frac{3}{2}|x|}^\infty \Big(\frac{r}{|x|}\Big)^{a-2-\mu} e^{-\frac19r^2} r \dd r
\\ & = 2\pi\int_{\frac{3}{2}}^\infty s^{a-1-\mu} |x|^2 \ee^{-\frac19 |x|^2 s^2}  \dd s
\\ & \leq 2\pi |x|^2 e^{-\frac1{36} |x|^2} \int_{\frac{3}{2}}^\infty s^{a-1-\mu} \ee^{-\frac1{36} s^2}  \dd s\leq C.
\end{align*}
If $|x|\leq 1$, then since $2\mu-j-a+2\geq 0$,
\begin{align*}
T_3 & \lesssim \int_{|y|>\frac{3}{2}|x|} |y|^{a-2} \ee^{-\frac19|y|^2}\dd y \leq 2\pi \int_{0}^\infty r^{a-1} e^{-\frac19r^2}\dd r <\infty.
\end{align*}
This finishes the proof.
\end{proof}

Let $\kappa\in (0,2\pi)$. On the wedge
$$D:=\{x\in \R^2: x = (r\cos(\varphi), r\sin(\varphi)), r>0, \varphi\in (0,\kappa)\}$$ consider heat equation:
\begin{equation}\label{eq:appheatwedge}
\begin{cases}
u_t= \Delta u, & \text{on $\R_+\times D$},\\
u(0,y)=f(y).
\end{cases}
\end{equation}
Let $\Gamma$ denote the Green kernel of the heat semigroup associated to \eqref{eq:appheatwedge}. The solution to \eqref{eq:appheatwedge} is given by  (see \cite[Lemma 3.7]{KN14})
\[e^{t \Delta} f(x) = \int_{D} \Gamma(x,y,t) f(y) \dd y.\]
In the next proposition we collect some properties of the heat semigroup on the wedge $D$.
\begin{proposition}\label{proposition:propertiesheatwedge}
Assume $\kappa\in (0,2\pi)$, $q\in (1, \infty)$, $\theta\in \R$ and set
\begin{align*}
  X &= L^q(D,\abs{\cdot}^{\theta-2}),\\
  Y &=\dot{W}^{1,q}(D,\abs{\cdot}^{\theta-2})\cap L^q(D,\abs{\cdot}^{\theta-2-q})
\end{align*}
The following assertions hold:
\begin{enumerate}[(i)]
\item \label{it:wedgesectorial} If $-\frac{\pi}{\kappa}<\frac{\theta}{q}<2+\frac{\pi}{\kappa}$, then $-\Delta$ is a sectorial operator of angle $<\pi/2$ on $X$. In particular, $(e^{t\Delta})_{t\geq 0}$ is a bounded analytic semigroup on $L^q(D,\abs{\cdot}^{\theta-2})$;
\item \label{it:wedgedini}  If $1-\frac{\pi}{\kappa}<\frac{\theta}{q}<2+\frac{\pi}{\kappa}$, then $\sup_{t>0} t^{\frac12}\|e^{t\Delta}\|_{\mc{L}(X,Y)}<\infty$.
\end{enumerate}
\end{proposition}
Although $-\Delta$ is sectorial of angle $<\pi/2$ for a large range of values of $\theta$, we do not know its domain on the full range.
Although we do not need it we note that if $2-\frac{\pi}{\kappa}<\frac{\theta}{q}<2+\frac{\pi}{\kappa}$, then (see \cite[Corollary 5.2]{PS07})
\[D(\Delta) = \cbraceb{u: u,u/\abs{\cdot}^2, \partial^{\alpha} u\in L^q(D,\abs{\cdot}^{\theta-2}) \text{ for } |\alpha|=2}.\]
The domain for other values of $\theta$ seems more difficult to characterize.

\begin{proof}
Let $\mu = \pi/\kappa \in (1/2,\infty)$. For \ref{it:wedgesectorial} first suppose $2-\mu<\frac{\theta}{q}<2+\mu$. It follows from \cite[Corollary 5.2]{PS07} that $-\Delta$ has deterministic maximal regularity. Thus in this case  \ref{it:wedgesectorial}  follows from \cite[Section 4]{Do00}.
 The range $-\mu<\frac{\theta}{q}<\mu$ follows by a duality argument from the range  $2-\mu<\frac{\theta}{q}<2+\mu$, since
\begin{equation*}
  \hab{L^q(D,\abs{\cdot}^{\theta-2})}^* = L^{q'}(D,\abs{\cdot}^{\widetilde{\theta}-2})
\end{equation*}
with $\widetilde{\theta} = (2q-\theta)/(q-1)$. The remaining range $\mu \leq \frac{\theta}{q}\leq 2-\mu$ follows by complex interpolation (see \cite[Theorem 1.18.5]{Tr78}).

For \ref{it:wedgedini} we use the following estimates for $\Gamma$ (see \cite[Theorem 3.10]{KN14}):
\begin{align*}
|\partial_x^{\alpha}\Gamma(x,y,t)| \leq C\zeta^{\mu-|\alpha|}(t,x) \zeta^{\mu}(t,y) t^{-\frac{2+|\alpha|}{2}} \exp\Big(-\frac{\sigma|x-y|^2}{t}\Big),\qquad \abs{\alpha}\leq 1
\end{align*}
where $\zeta(t,x) = \frac{|x|}{|x|+\sqrt{t}}$. Therefore it suffices to prove for $f \in L^q(D,\abs{\cdot}^{\theta-2})$
\begin{align*}
\Big\|x\mapsto \int_{D} k_t(x,y) f(y) \dd y\Big\|_{L^q(D,\abs{\cdot}^{\theta-2})} \leq C\|f\|_{L^q(D,\abs{\cdot}^{\theta-2})},
\end{align*}
where $k_t(x,y)$ is either
\begin{align}
\label{eq:kernel1}  & \zeta^{\mu-1}(t,x) \zeta^{\mu}(t,y) t^{-1} \exp\Big(-\frac{\sigma|x-y|^2}{t}\Big),    \intertext{or}
\label{eq:kernel2} & \zeta^{\mu}(t,x)  \zeta^{\mu}(t,y) \abs{x}^{-1}  t^{-1/2} \exp\Big(-\frac{\sigma|x-y|^2}{t}\Big),
\end{align}
where \eqref{eq:kernel1} and \eqref{eq:kernel2} correspond  to the boundedness in  $\dot{W}^{1,q}(D,|x|^{\theta-2})$ and $ L^q(D,|x|^{\theta-2-q})$ respectively. Since \eqref{eq:kernel2}$\leq$\eqref{eq:kernel1}  it suffices to prove the boundedness for the case \eqref{eq:kernel1}.
A simple rewriting shows that it is enough to prove for $g \in L^q(D,\abs{\cdot}^{-2})$
\begin{align*}
\Big\|x\mapsto \int_{D} k_t(x,y) |x|^{\frac{\theta}{q}} |y|^{2-\frac{\theta}{q}} g(y) \frac{{\rm d} y}{\abs{y}^2} \Big\|_{L^q(D,\abs{\cdot}^{-2})} \leq C\, \|g\|_{L^q(D,\abs{\cdot}^{-2})},
\end{align*}
To prove the latter by Schur's lemma (see \cite[Apendix A]{Gr14b}) it suffices to show
\begin{align*}
\sup_{t>0,y\in \R^2 }\int_{D} k_t(x,y) |x|^{\frac{\theta}{q}} |y|^{2-\frac{\theta}{q}} \frac{{\rm d} x}{|x|^2} &<\infty,
\\ \sup_{t>0,x\in \R^2 }\int_{D} k_t(x,y) |x|^{\frac{\theta}{q}} |y|^{2-\frac{\theta}{q}} \frac{{\rm d} y}{|y|^2} &<\infty,
\end{align*}
which follows from Lemma \ref{lemma:techincalkest}.
\end{proof}

\subsection*{Parabolic Horm\"ander and Dini conditions on a smooth bounded domain}
Define the parabolic norm on $\R \times \R^d$ by
\begin{equation*}
  \abs{(t,x)}_{(2,1)} := \max\cbrace{\abs{t}^{1/2}, \abs{x}}, \qquad (t,x), \in \R \times \R^d.
\end{equation*}
Let $D \subseteq \R^d$ be a smooth bounded domain and fix $T \in (0,\infty]$. We equip $(0,T) \times D$ with the parabolic metric induced by  $\abs{\cdot}_{(2,1)}$, which turns it into a space of homogeneous type (see also Section \ref{section:homogeneoustype}). We will show that versions of Lemma \ref{lemma:hormanderdini} and Lemma \ref{lemma:standardkernelderivatives} work in this setting.

\begin{lemma}\label{lemma:checktechnical}
  Fix $T \in (0,\infty]$ and let
  $K\colon(0,T)\times (0,T) \times D \times D \to \C$
 be measurable such that $K(t,s,x,y)=0$ for $t<s$. Suppose there exist $C,c>0$ such that for $\abs{\alpha} \leq 1$,
 \begin{align}\label{eq:technicaldini1}
  \absb{\partial^\alpha_x K(t,s,x,y)} &\leq C\, \frac{1}{(t-s)^{(\abs{\alpha}+d+1)/2}} \,\exp\has{-c\frac{\abs{x-y}}{(t-s)^{1/2}}},\\
    \absb{\partial_t K(t,s,x,y)} &\leq C\, \frac{1}{(t-s)^{(d+3)/2}} \,\exp\has{-c\frac{\abs{x-y}}{(t-s)^{1/2}}},\label{eq:technicaldini2}
\end{align}
for all $t>s$ and $x,y \in D$. Then
 \begin{equation*}
   \absb{K\ha{t,s,x,y}- K\ha{t',s,x',y}} \lesssim \, \frac{\abs{(t-t',x-x')}_{(2,1)}}{\abs{(t-s,x-y)}_{(2,1)}^{d+2}}
 \end{equation*}
for all $\abs{(t-t',x-x')}_{(2,1)}\leq \frac12 \abs{(t-s,x-y)}_{(2,1)}$.
\end{lemma}

\begin{proof}
Take $\abs{(t-t',x-x')}_{(2,1)}\leq \frac12 \abs{(t-s,x-y)}_{(2,1)}$ and define
\begin{align*}
  \text{\framebox[15pt]{A}} &:= \absb{K\ha{t,s,x,y}-K\ha{t',s,x,y}},\\
  \text{\framebox[15pt]{B}} &:= \absb{K\ha{t',s,x,y}-K\ha{t',s,x',y}}.
\end{align*}
By the triangle inequality it suffices to estimate \framebox[15pt]{A} and \framebox[15pt]{B} separately. Let us first consider $\framebox[15pt]{A}$. If $t,t'<s$ there is nothing to prove. If $t'<s<t$ we have $t-t'>t-s$ and thus also $\frac12\abs{x-y}^2 \geq t-t'$. Therefore using \eqref{eq:technicaldini1} with $\alpha=0$ we have the estimate
\begin{align*}
  \text{\framebox[15pt]{A}} &\lesssim   \frac{\ha{t-s}^{1/2}}{\ha{t-s}^{(d+2)/2}} \exp\has{-c\frac{\abs{x-y}}{(t-s)^{1/2}}} \lesssim
  \frac{\abs{(t-t',x-x')}_{(2,1)}}{\abs{t-s,x-y}_{(2,1)}^{d+2}}.
\end{align*}
If $s<t'<t$ we first consider the case that $\abs{x-y} \leq \ha{t-s}^{1/2}$. Then by \eqref{eq:technicaldini2} and the fundamental theorem of calculus we have
\begin{align*}
  \text{\framebox[15pt]{A}}
  &\lesssim \frac{\ha{t-t'}}{\ha{t'-s}^{(d+3)/2}}  \leq 2 \frac{\abs{(t-t',x-x')}_{(2,1)}^2}{\abs{t-s,x-y}_{(2,1)}^{d+3}},
\end{align*}
since in this case
\begin{equation}\label{eq:tt'comparable}
  \abs{t'-s} \geq \abs{t-s} - \abs{t-t'} \geq \tfrac12 \abs{t-s}.
\end{equation}
Next if $\ha{t-s}^{1/2} \leq \abs{x-y}$, then again by \eqref{eq:technicaldini2} and the fundamental theorem of calculus we have
\begin{align*}
  \text{\framebox[15pt]{A}}
  &\lesssim \frac{\abs{t-t'}}{\abs{x-y}^{d+3}}  \leq \frac{\abs{(t-t',x-x')}_{(2,1)}^2}{\abs{t-s,x-y}_{(2,1)}^{d+3}}
\end{align*}
The cases $t<s<t'$ and $s<t<t'$ are treated similarly with the roles of $t$ and $t'$ interchanged.

Now for \framebox[15pt]{B} suppose that $t'>s$. Let $\gamma:[0,1]\to D$ be a smooth curve from $x$ to $x'$ such that
\begin{align*}
\sup_{r \in [0,1]}\gamma'(r) &\lesssim \abs{x-x'}\\
  \inf_{r \in [0,1]} \abs{\gamma(r)-y} &\gtrsim \min\cbraceb{\abs{x-y},\abs{x'-y}},
\end{align*}
which exists since $D$ is smooth and bounded.
 We first consider the case that $\abs{x-y} \leq \abs{t-s}^{1/2}$. Then by \eqref{eq:technicaldini1} and the fundamental theorem of calculus we have
\begin{align*}
  \text{\framebox[15pt]{B}} &= \absb{\int_0^1 \frac{\ddn}{\ddn r} K\hab{t,s,\gamma(r),y} \dd r}\\
  &\lesssim \frac{\abs{x-x'}}{\abs{t'-s}^{(d+2)/2}}  \leq 2 \frac{\abs{(t-t',x-x')}_{(2,1)}}{\abs{(t-s,x-y)}_{(2,1)}^{d+2}},
\end{align*}
since \eqref{eq:tt'comparable} is valid in this case. Similarly if $\abs{t-s}^{1/2} \leq \abs{x-y}$ we have
\begin{align*}
  \text{\framebox[15pt]{B}} &\lesssim \frac{\abs{x-x'}}{\min\cbrace{\abs{x-y},\abs{x'-y}}^{d+2}}  \leq 2 \frac{\abs{(t-t',x-x')}_{(2,1)}}{\abs{t-s,x-y}_{(2,1)}^{d+2}},
\end{align*}
since in this case
\begin{equation*}
  \abs{x'-y} \geq \abs{x-y} - \abs{x-x'} \geq \tfrac12 \abs{x-y}.\qedhere
\end{equation*}
\end{proof}

Now we give an analog of Lemma \ref{lemma:hormanderdini} in the parabolic setting:
\begin{lemma}\label{lemma:21hormanderdini}
  Fix $T \in (0,\infty]$, let
  $K\colon(0,T)\times (0,T) \times D \times D \to \C$
 be measurable  and take $\epsilon \in (0,1]$. Suppose there is an $A_0>0$ such that
 \begin{equation*}
   \absb{K\ha{t,s,x,y}-K\ha{t',s,x',y}} \leq A_0 \, {\frac{\abs{(t-t',x-x')}^\epsilon_{(2,1)}}{\abs{(t-s,x-y)}_{(2,1)}^{\epsilon+d+1}}}
 \end{equation*}
for all $\abs{(t-t',x-x')}_{(2,1)}\leq \frac12 \abs{(t-s,x-y)}_{(2,1)}$. Then for all $(t,x),(t',x') \in (0,T)\times D$ we have
\begin{equation*}
  \has{\int_0^T \has{\int_{I(s)} \abs{K\ha{t,s,x,y}-K\ha{t',s,x',y}} \dd y}^2\dd s}^{1/2} \leq C_\epsilon \, A_0
\end{equation*}
where
\begin{equation*}
  I(s) := \cbraceb{y \in D: \abs{(t-t',x-x')}_{(2,1)}\leq \tfrac12 \abs{(t-s,x-y)}_{(2,1)}}.
\end{equation*}
\end{lemma}

\begin{proof}
Take $(t,x),(t',x') \in (0,T)\times D$ and define $\rho = \abs{(t-t',x-x')}_{(2,1)}$. Then by assumption we have
\begin{align*}
\has{\int_0^T \has{\int_{I(s)} &\abs{K\ha{t,s,x,y}-K\hab{t',s,x',y} \dd y}}^2\dd s}^{1/2}\\
 &\leq A_0\,
\has{\int_0^T \has{\int_{I(s)} {\frac{\abs{(t-t',x-x')}^\epsilon_{(2,1)}}{\abs{(t-s,x-y)}_{(2,1)}^{\epsilon+d+1}}} \dd y}^2\dd s}^{1/2}\\
&\leq A_0\,
\has{\int_\R \has{\int_{I_1(r)\cup I_2(r)} {\frac{\rho^\epsilon}{\abs{(r,z)}_{(2,1)}^{\epsilon+d+1}}} \dd z}^2\dd r}^{1/2}\\
&\leq A_0\,(S_1+S_2)
\end{align*}
where $S_1$ and $S_2$ are the parts of the inner integral over $I_1(r)$ and $I_2(r)$ respectively with
\begin{align*}
 I_1(r) &:= \cbraceb{z \in D:  \abs{(r,z)}_{(2,1)}\geq 2\rho \text{ and }\abs{(r,z)}_{(2,1)} = \abs{z} }\\
 I_2(r) &:= \cbraceb{z \in D:  \abs{(r,z)}_{(2,1)}\geq 2\rho \text{ and }\abs{(r,z)}_{(2,1)} = \abs{r}^{1/2}}
\end{align*}
For $S_1$ we have
   \begin{align*}
     S_1 &= \nrms{r \mapsto \int_{I_1(r)} \frac{\rho^{\epsilon}}{\abs{z}^{\epsilon+d+1}}\dd z}_{L^2(\R)}\\
     &\lesssim \nrms{r \mapsto \int_{\max\cbrace{2\rho,\abs{r}^{1/2}}}^\infty \frac{\rho^{\epsilon}}{u^{\epsilon+2}}\dd u}_{L^2(\R)}\\
     &\lesssim  \nrms{r \mapsto \int_{2\rho}^\infty  \frac{\rho^{\epsilon}}{u^{\epsilon+2}}\dd u}_{L^2(0,4\rho^2)}+ \nrms{r \mapsto \int_{\abs{r}^{1/2}}^\infty \frac{\rho^{\epsilon}}{u^{\epsilon+2}}\dd u}_{L^2(4\rho^2,\infty)}\\
     &\lesssim \rho^{-1} \has{\int_{0}^{4\rho^2} \dd r}^{1/2} + \rho^\epsilon \has{\int_{4\rho^2}^\infty \frac{1}{r^{\epsilon+1}}\dd r}^{1/2} \leq C_\epsilon,
   \end{align*}
   and for $S_2$ we have
   \begin{align*}
    S_2 &\leq 2 \nrms{r \mapsto \int_{\abs{z} \leq r^{1/2}} \frac{\rho^{\epsilon}}{r^{(\epsilon+d+1)/2}}\dd z}_{L^2(4\rho^2,\infty)}\\
    &\lesssim  \rho^{\epsilon}\nrmb{r \mapsto  \frac{1}{r^{(\epsilon+1)/2}}}_{L^2(4\rho^2,\infty)}\leq C_\epsilon,
   \end{align*}
   which proves the lemma.
  \end{proof}

\bibliographystyle{alpha-sort}
\bibliography{thesisbib}

\appendix
\end{document}